\documentclass[11pt,a4paper]{amsart}
\usepackage{amssymb}
\usepackage{amscd}
\usepackage{amsmath,amsfonts,stmaryrd,graphicx}
\input xy 
\xyoption{all}
\usepackage{tikz}
\usetikzlibrary{arrows,shapes,chains,matrix,positioning,scopes,cd}
\usepackage{hyperref}
\usepackage{array}
\usepackage{caption}
\usepackage{extarrows}
\newcolumntype{M}[1]{>{\centering\arraybackslash}m{#1}}
\newcolumntype{N}{@{}m{0pt}@{}}

\theoremstyle{plain}
\newtheorem{thm}{Theorem}[section]
\newtheorem{lem}[thm]{Lemma}

\newtheorem{pro}[thm]{Proposition}
\newtheorem{cor}[thm]{Corollary}
\theoremstyle{definition}
\newtheorem{dfn}[thm]{Definition}
\newtheorem{exa}[thm]{Example}
\newtheorem{rem}[thm]{Remark}

\DeclareMathOperator{\Hom}{Hom}

\DeclareMathOperator{\End}{End}
\DeclareMathOperator{\Ext}{Ext}

\DeclareMathOperator{\modu}{mod}

\newcommand{\Ga}{\Gamma}

\newcommand{\op}{\mathrm{op}}

\DeclareMathOperator{\Gamod}{\Ga\text{-}mod}
\DeclareMathOperator{\Lamod}{\La\text{-}mod}
\DeclareMathOperator{\eLamod}{e\La e\text{-}mod}
\DeclareMathOperator{\Demod}{\De\text{-}mod}
\DeclareMathOperator{\Simod}{\Si\text{-}mod}

\DeclareMathOperator{\pd}{pd}
\DeclareMathOperator{\id}{id}
\DeclareMathOperator{\idim}{id}
\DeclareMathOperator{\pdim}{pd}
\DeclareMathOperator{\gldim}{gldim}

\newcommand{\oTo}{\xymatrix{ \ar@{^{(}->}[r]|{\mathbf{O}}& }} 
\newcommand{\cTo}{\xymatrix{ \ar@{^{(}->}[r]|{\mathbf{|}}& }} 
\newcommand{\coTo}{\xymatrix{ \ar@{^{(}->}[r]|{\mathbf{O}}|{\mathbf{|}}& }} 

\DeclareMathOperator{\Ker}{ker}
\DeclareMathOperator{\coKer}{coker}
\DeclareMathOperator{\Bild}{Im}

\DeclareMathOperator{\kdual}{D}

\DeclareMathOperator{\cogen}{cogen}
\DeclareMathOperator{\copres}{copres}
\DeclareMathOperator{\pres}{pres}
\DeclareMathOperator{\gen}{gen}

\DeclareMathOperator{\rad}{rad}
\newcommand{\kLRcluster}{$k$-$(L,R)$-cluster tilting module}
\newcommand{\kLRclusters}{$k$-$(L,R)$-cluster tilting modules}
\newcommand{\nLRcluster}{$(n-2)$-$(L,R)$-cluster tilting module}

\newcommand{\Si}{\Sigma}

\newcommand{\La}{\Lambda}
\newcommand{\al}{\alpha}

\newcommand{\De}{\Delta}

\newcommand{\N}{\mathbb{N}}

\newcommand{\Z}{\mathbb{Z}}

\newcommand{\mcI}{\mathcal{I}}

\newcommand{\mcP}{\mathcal{P}}

\newcommand{\mbF}{\mathbf{F}}

\DeclareMathOperator{\add}{add}

\DeclareMathOperator{\cotilt}{cotilt}

\DeclareMathOperator{\domdim}{domdim}

\usepackage{todonotes}

\begin{document}

\title{On faithfully balanced modules, \\
F-cotilting and F-Auslander algebras}

\begin{abstract} We revisit faithfully balanced modules. These are faithful modules having the double centralizer property. 
For finite-dimensional algebras our main tool is the category $\cogen^1(M)$ of modules with a copresentation
by summands of finite sums of $M$ on which $\Hom(-,M)$ is exact. 
For a faithfully balanced module $M$ the functor $\Hom(-,M)$ is a duality on these categories - for cotilting modules this is the Brenner-Butler theorem. 
We also study new classes of faithfully balanced modules combining cogenerators and cotilting modules. 
Then we turn to relative homological algebra in the sense of Auslander-Solberg and define a 
relative version of faithfully balancedness which we call $1$-$\mbF$-faithful. We find relative versions of the best known classes of 
faithfully balanced modules (including (co)generators ,(co)tilting and cluster tilting modules). Here we characterize the corresponding modules over the endomorphism ring of the faithfully balanced module - 
this is what we call a \emph{correspondence}. Two highlights are the relative (higher)  Auslander correspondence and the relative cotilting correspondence - the second is a generalization of a relative cotilting correspondence of Auslander-Solberg to an involution (as the usual cotilting correspondence is).  
\end{abstract}
\author{Biao Ma}
\address{Biao Ma, Department of Mathematics, Nanjing University, 22 Hankou Road, Nanjing 210093, People's Republic of China}
\email{biao.ma@math.uni-bielefeld.de}

\author{Julia Sauter}
\address{Julia Sauter\\ Faculty of Mathematics \\
Bielefeld University \\
PO Box 100 131\\
D-33501 Bielefeld }
\email{jsauter@math.uni-bielefeld.de}


\date{\today}
\maketitle
\section{Introduction}
Let $\La$ be a ring and $M$ a left $\La$-module. We write endomorphisms of ${_\La}M$ on the left, thus for two endomorphisms $f,g\in \End_{\La} (M)$ and an element $m\in M$ the image of $m$ under $gf$ is $g(f(m))$. Then $M$ can be considered naturally as a left 
$\End_{\La} (M)$-module, and moreover as a left $\End_{\La} (M)$-left $\La$-bimodule. We say $M$ is \emph{faithful}/ \emph{balanced}/ \emph{faithfully balanced} \footnote{In \cite[section 4]{AF}, faithfully balancedness is only defined for bimodules. A faithfully balanced module ${_\La}M$ in this paper is called faithful and balanced in loc. cit., which is equivalent to say that ${_\La}M_{\End_{\La}(M)^{\op}}$ is a faithfully balanced bimodule. For example in \cite{SW} faithfully balancedness is used for left modules as in this article. } if the natural map of rings $\La \to \End_{\End_{\La}(M)} (M)$ is injective/surjective/bijective. Balanced modules are also known as modules with the \emph{double centralizer property}, see for example \cite{DR}. In \cite{Wi}, $M$ is faithfully balanced means $\La$ is $M$-static. In \cite{BS}, a faithfully balanced module is a module of faithful dimension at least 2.

In this paper, we restrict to study finite-dimensional algebras and finite-dimensional modules over them. 
For a module ${_{\La}}M$ we define $\add (M)$ to be the category consisting of direct summands of finite direct sums of $M$ and 
\[ \cogen^1(M)= \{ X \mid \exists\ 0 \to X \to M_0\to M_1\ \text{exact}, M_i \in \add(M)\ \text{and}\ \Hom_{\La}(-,M) \text{ exact on it}\}. \]
Dually, one can define $\gen_1(M)$. If ${_{\La}}M$ is a faithfully balanced module, then we have a duality 
\[ 
\Hom_{\La}(-,{_{\La}}M) \colon\cogen^1({}_{\La}M) \longleftrightarrow \cogen^1({}_{\Ga}M) \colon \Hom_{\Ga}(-,{_{\Ga}}M)
\]
where $\Ga=\End_{\La}(M)$. Buan and Solberg \cite{BS} first observed the symmetry:  $\La\in  \cogen^1(_{\La} M)$ is equivalent to $\kdual \La \in \gen_1({}_{\La} M)$ and both are equivalent to $M$ being faithfully balanced (see also Lemma \ref{lagencogen}).
We will consider tuples $(\La , M_1, \ldots , M_t)$ consisting of an algebra and several modules 
up to an equivalence relation which identifies two such tuples $(\La , M_1, \ldots , M_t)$ and $(\La', M_1', \ldots , M_t')$ if there is a Morita equivalence from $\La$ to $\La'$ which sends each $\add (M_i)$ to $\add (M_i')$. We denote by $[\La, M_1, \ldots , M_t]$ the equivalence class of $(\La, M_1, \ldots , M_t)$. It is easy to see that faithfully balancedness of a module is preserved under this equivalence (cf. \cite{CR}). 
\begin{itemize}
\item[$(*)$] The assignment $[\La , {_\La}M] \mapsto 
[\End_{\La}(M), {_{\End_{\La}(M)}}M] $ is an involution on the set of pairs $[\La, {_\La}M]$ with $M$ a faithfully balanced module. 
\end{itemize}
It is a generally intriguing problem to establish an \emph{Endo}-dictionary explaining which properties of $\La$ and ${_{\La}}M$ are translated into which properties of $\End_{\La}(M)$ and $ {_{\End_{\La}(M)}}M$.
A restriction of $(*)$ to a bijection between two sets of such pairs (or related tuples) will be called a \emph{correspondence}.

Two classes of well-studied faithfully balanced modules are (co)tilting modules \cite{BrenBut, M} and (co)generators \cite[Theorem 3]{Tach-Spl} - and their special cases:  
generator-cogenerators \cite{Morita71, Tach-PJM} and Auslander generators (i.e., the additive generator of the module category \cite{AQueenMary}). Starting with M\"uller's results \cite{Mue} there are also \emph{higher} versions of any of these.   
One of our motivations was to understand the interplay between correspondences and relative homological algebra in the sense of Auslander-Solberg \cite{ASoI}.  
In this paper we explain a relative version of faithfully balancedness and then systematically look at the relative analogues of the well-known correspondences. Let us give a brief overview of previously studied correspondences (see the following table) in representation theory of finite-dimensional algebras (or more generally, artin algebras). The relative versions can be found in the corresponding theorems in the second column.  

\begin{table}[ht]
\begin{tabular}{M{9cm}|M{5cm} N}
\hline
\textbf{Classical case} & \textbf{Relative case} \\
\hline
 (co)generator correspondence \\
 (=Wedderburn correspondence and $\Hom (-,\text{ring})$)    & Corollary \ref{relgencogencorres} (1) (2) &  \\[3pt]
\hline
 Morita-Tachikawa correspondence\\ (=generator-cogenerator correspondence) &  Corollary \ref{relgencogencorres} (3)&   \\[2pt]
\hline
 M\"uller correspondence & Lemma \ref{relduality} &\\[3pt]
\hline
 (higher) Auslander correspondence      & Theorem \ref{relAuscorres}  &  \\[3pt]
\hline
 Auslander-Solberg correspondence       & Theorem \ref{relAScorres}  &  \\[3pt]
\hline
(co)tilting correspondence\\ 
 (=Brenner-Butler theorem)              & Theorem \ref{relBBthm}  &  \\[2pt]
\hline
 correspondence of Gorenstein algebras  & Corollary \ref{relGor} &  \\[2pt]
\hline
\end{tabular}
\label{correstable}
\end{table}

 We give a summary of the content (but in the introduction we restrict to the easy versions).
 
In section 2, we study some basic properties of faithfully balanced modules and dualities (equivalences) of subcategories. 

We start the relative theory in section 5. We consider an additive subbifunctor $\mbF\subseteq \Ext^1(-,-)$ of the form $\mbF=\mbF_G=\mbF^H$ for a generator $G$ and a cogenerator $H$ - this is equivalent to consider the exact structure on finite-dimensional $\La$-modules induced by the functor $\mbF$ (cf. \cite{DRSS}), meaning an exact sequence is $\mbF$-exact if and only if it remains exact after applying the functor $\Hom_{\La}(G,-)$ (or equivalently after applying the functor $\Hom_{\La} (-,H)$). 
We define $\cogen^1_{\mbF}(M) \subseteq \cogen^1(M)$ to be the full subcategory of modules $X$ such that there exists an exact sequence $0\to X \to M_0 \to M_1 $ with $M_0, M_1\in \add(M)$ and $\Hom_{\La}(-,H\oplus M)$ is exact on it (analogously we define  $\gen_{1}^{\mbF}(M)$).
We also introduce the notion of $1$-$\mbF$-faithfulness (meaning $G \in \cogen^{1}_{\mbF}(M)$) as the relative analogue of the notion of faithfully balancedness.
Let ${_{\La}}M$ be $1$-$\mbF$-faithful, then we have a duality
\[  
\Hom_{\La}(-,{_{\La}}M) \colon \cogen^1_{\mbF^H}(M) \longleftrightarrow \cogen^1_{\mbF^R} (M) \colon \Hom_{\Ga}(-,{_{\Ga}}M)
\]
where $\Ga=\End_{\La}(M)$ and $R=\kdual\Hom_{\La}(M,H)$. There is also a dual version of the above duality which involves the modules $G$ and $L:=\Hom_{\La}(G,M)$. Then we observe the following relationship between $G,H$ and $L,R$
\[
\xymatrix@C=1em@R=1em{ {_{\La}}G  \ar@{<->}@/_1pc/[dd]_{(-, M)}  \ar@{-->}@/^1pc/[rrrr]^{\tau} & & & & {_{\La}}H \ar@{-->}@/^1pc/[llll]_{\tau^{-}}\\
\\
{_{\Ga}}L  \ar@{-->}@/^1pc/[rrrr]_{\Omega_M^{-2}} & & & & {_{\Ga}}R \ar@{<->}@/_1pc/[uu]_{\kdual(M,-)} \ar@{-->}@/^1pc/[llll]^{\Omega_M^{2}}
}
\]
Here the upper dashed arrows mean $H=\tau G \oplus \kdual\La$ and $G=\tau^-H \oplus \La$ whereas the lower dashed arrows mean $R = {_{\Ga}}M\oplus \Omega_M^{-2} L $ and  $L = {_{\Ga}}M\oplus \Omega_M^{2} R$. 
As in the classical case, we have $G \in \cogen^{1}_{\mbF}(M)$ is equivalent to $H \in \gen_{1}^{\mbF}(M)$ (Theorem \ref{Rel-fb}).

We consider the assignment (AS) (referring to Auslander and Solberg):
\begin{itemize}
\item[(AS)] The assignment $[\La , {_\La}M, G] \mapsto 
[\Ga=\End_{\La}(M), {_{\Ga}}M, L=\Hom_{\La}(G,M)] $ with $M$ faithfully balanced and $G$ a generator.  
\end{itemize}

\paragraph{\textbf{Generator correspondence.}}
A generator ${_\La}G$ is, by definition, a module such that $\La\in \add(G)$ which is automatically a faithfully balanced module. 
\begin{thm} {\rm (Generator correspondence, \cite{Az,Tach-Spl})}
    \emph{The assignment} $(*)$ \emph{restricts to a bijection} 
    \[ \{[\La, G]\colon  G\text{ is a generator} \} \xlongleftrightarrow{1:1} \{[\Ga, P]\colon P\text{ is a f.b. projective module} \}\]
\emph{where f.b. is the abbreviation of faithfully balanced.}
\end{thm}
 The generator correspondence can also be expressed as Auslander's Wedderburn correspondence composed with the duality $\Hom_{\Ga}(-,\Ga)$, see \cite{ASoIII}.

Our relative generalization is the following 
\begin{thm}{\rm (= Corollary \ref{relgencogencorres} (1))}
\emph{The assignment (AS) restricts to an involution on the set of triples}
$[\La , M , G]$ \emph{with} $\La \oplus M \in \add (G)$ \emph{and} $M$ is $1$-$\mbF_G$-\emph{faithful}   
\end{thm}

In Corollary \ref{relgencogencorres} we also give relative versions of the cogenerator correspondence and the generator-cogenerator correspondence (also known as the Morita-Tachikawa correspondence \cite{Morita71, Tach-PJM}, see also \cite{RingelNotes}). The most famous special case is the Auslander correspondence, see below. 


\paragraph{\textbf{Auslander-Solberg and Auslander correspondence}}
The Auslander-Solberg correspondence, which is defined by Iyama and Solberg \cite{IS}, characterizes algebras $\La$ with $\domdim \La\geq k+1 \geq \id \La$. In the case $k=1$, this result is due to Auslander-Solberg \cite{ASo-Gor}. 

Our relative generalization is the following 
\begin{thm}{\rm (= Theorem \ref{relAScorres}, $k=1$)}
\emph{The assignment (AS) restricts to an involution on the set of triples}
$[\La , M , G]$ \emph{with} $\La \in \add (G)$,  $\mbF=\mbF_G$, $M$ \emph{is both} $1$-$\mbF$-\emph{faithful and} $\mbF$-\emph{projective}-\emph{injective}, \emph{and} $\domdim_{\mbF} \La \geq 2 \geq \idim_{\mbF} G$.   
\end{thm}

As a special case of the Auslander-Solberg correspondence, Iyama's higher Auslander correspondence (\cite{IyAC}) characterizes algebras $\La$ with $\domdim \La\geq k+1 \geq \gldim \La$. The case $k=1$ is the well-known Auslander correspondence \cite{AQueenMary}.


Our relative generalization is the following
\begin{thm}{\rm (= Theorem \ref{relAuscorres}, $k=1$)}
\emph{The assignment (AS) restricts to an involution on the set of triples}
$[\La , M , G]$ \emph{with} $\La \in \add (G)$, $\mbF=\mbF_G$, $M$ \emph{is both} $1$-$\mbF$-\emph{faithful and} $\mbF$-\emph{projective}-\emph{injective}, \emph{and} $\domdim_{\mbF} \La \geq 2 \geq \gldim_{\mbF} \La$.   
\end{thm}

\paragraph{\textbf{Cotilting correspondence}}
The main result on relative cotilting modules (cf. Definition \ref{Fcotilting}) of Auslander-Solberg is the following 
\begin{thm} $($\cite[Theorem 3.13]{ASoII}, \cite[Theorem 2.8]{ASoIII}$)$
\emph{The assignment (AS) restricts to a bijection between the following two sets of triples}
\begin{itemize}
\item[(1)] $[\La , M , G]$ \emph{with} $\La \in \add (G)$, $\mbF=\mbF_G$, $M$ \emph{is $\mbF$-cotilting, and} 
\item[(2)] $[\Ga ,N, L]$ \emph{with} $N \in \add (L)$, $L\in \cogen^1(N)$ \emph{and} $L$ \emph{is a cotilting module.} 
\end{itemize}
\end{thm}

To improve this result, we need the 4-tuple assignment
\[ [\La , M, L, G] \mapsto [\Gamma, N, \widetilde{L}, \widetilde{G} ]  \]
with $\Gamma = \End_{\La}(M) $, $N={}_{\Ga} M$, $\widetilde{L} = \Hom_{\La}(G, M)$, $\widetilde{G} = \Hom_{\La}(L,M)$. Then we have 
\begin{thm}{\rm (= Theorem \ref{relBBthm})}
The $4$-tuple assignment restricts to an involution on the set of $4$-tuples $[\La , M, L, G]$ satisfying $\La \in \add (G)$, $\mbF=\mbF_G$, $L$ is $\mbF$-cotilting and $L\in \cogen^1_{\mbF}(M)$. 
\end{thm}
It is well known that a cotilting module will induce a triangle duality, see \cite{Hap, CPS}. We prove a relative analogue of this result (Proposition \ref{relder-equi}): In the situation of the previous theorem 
we have a triangle duality between $\mathsf{D}^b_{\mbF_G}(\Lamod)$ and $\mathsf{D}^b_{\mbF_{\widetilde G}}(\Gamod)$ where $\Ga=\End_{\La}(M)$ and ${\widetilde G}=\Hom_{\La}(L,M)$.

We illustrate the above results by the following easy examples which are special cases of $\mbF$-Auslander algebras from Example \ref{ex-FAusalg-2}(4).
\begin{exa}
\begin{itemize}
\item[(1)] Let $\La$ be the path algebra of the quiver $1\to 2\to 3$ and consider the $\La$-modules  $M=P_1\oplus(P_2\oplus \tau^-P_2)$, $G=P_3\oplus M$ and $H=I_1 \oplus M$.  Then we have $\mbF_G=\mbF^H=:\mbF$. It is easy to see that $\domdim_{\mbF}\La=2=\gldim_{\mbF}\La$, and hence $\La$ is a 1-$\mbF$-Auslander algebra. Now, we see $\End_{\La}(M)\cong \La$, ${_{\End_{\La}(M)}}M\cong {_{\La}}M$ and $\Hom_{\La}(G,M)\cong {_{\La}}G$. It follows that the triple $[\La, M, G]$ is a fixed point of the assignment (AS).

\item[(2)] 
The same idea leads to a $2$-$\mbF$-Auslander algebra structure (i.e. $\domdim_{\mbF}\La\geq 3\geq \gldim_{\mbF}\La$) on $\La=K(1\to 2\to 3\to 4)$. Consider $M=P_1\oplus(P_2\oplus \tau^-P_2)\oplus (P_3\oplus \tau^-P_3\oplus \tau^{-2}P_3)$, $G=P_4\oplus M$ and $H=I_1\oplus M$. Then we have $\mbF_G=\mbF^H=:\mbF$ 
and one easily sees $\La$ is a 2-$\mbF$-Auslander algebra. 
 We also define $L=\tau^-P_4\oplus M$, the 
 $\mbF$-exact sequence  sequence $0\to L\to \tau^-P_3\oplus M \to \tau^{-2}P_3\oplus M\to H\to 0$ can be used to show that $L$ is a 2-$\mbF$-cotilting module and $L\in \cogen^1_{\mbF}(M)$. Then we have 
\[
\xymatrix @-1.5pc{
&&& \bullet \ar[dr]&&&&\\
\Ga=\End_{\La}(M)\colon &&\bullet \ar[ur]\ar[dr]\ar @{..}[rr]&&  \bullet \ar[dr]&&& {{_{\Ga}}M={\begin{smallmatrix} & & 1 & &\\ & 0 & & 1 &\\ 0 & & 0 & & 1\end{smallmatrix}} \oplus {\begin{smallmatrix} & & 1 & &\\ & 1 & & 1 &\\ 0 & & 1 & & 1\end{smallmatrix}}\oplus {\begin{smallmatrix} & & 1 & &\\ & 1 & & 1 &\\ 1 & & 1 & & 0\end{smallmatrix}} \oplus {\begin{smallmatrix} & & 1 & &\\ & 1 & & 0 &\\ 1 & & 0 & & 0\end{smallmatrix}}} \\
&\bullet \ar[ur]\ar @{..}@/_1pc/[rrrr] & & \bullet \ar[ur]& & \bullet &&
}
\]
$\widetilde{G}:=\Hom_{\La}(L,M)=\begin{smallmatrix} & & 0 & &\\ & 0 & & 1 &\\ 0 & & 1 & & 0\end{smallmatrix} \oplus \Ga$, $\widetilde{H}:=
\begin{smallmatrix} & & 1 & &\\ & 0 & & 1 &\\ 0 & & 0 & & 1\end{smallmatrix} \oplus \kdual\Ga$ and $\mbF_{\widetilde{G}}=\mbF^{\widetilde{H}}=:\widetilde{\mbF}$. We also define  $\widetilde{L}:=\Hom_{\La}(G,M)
= \begin{smallmatrix} & & 0 & &\\ & 0 & & 0 &\\ 0 & & 0 & & 1\end{smallmatrix} \oplus \begin{smallmatrix} & & 0 & &\\ & 0 & & 1 &\\ 0 & & 0 & & 1\end{smallmatrix} \oplus \begin{smallmatrix} & & 0 & &\\ & 0 & & 1 &\\ 0 & & 1 & & 1\end{smallmatrix} \oplus {_{\Ga}}M$ 
and an $\widetilde{\mbF}$-exact sequence
\[0\to \widetilde{L} \to {_{\Ga}}M \oplus (\begin{smallmatrix} & & 1 & &\\ & 0 & & 1 &\\ 0 & & 0 & & 1\end{smallmatrix})^2\oplus \begin{smallmatrix} & & 1 & &\\ & 1 & & 1 &\\ 0 & & 1 & & 1\end{smallmatrix} \to {_{\Ga}}M \oplus (\begin{smallmatrix} & & 1 & &\\ & 1 & & 0 &\\ 1 & & 0 & & 0\end{smallmatrix})^2 \oplus \begin{smallmatrix} & & 1 & &\\ & 1 & & 1 &\\ 1 & & 1 & & 0\end{smallmatrix} \to \widetilde{H} \to 0.\]
can be used to see that 
$\widetilde{L}$ is a 2-$\widetilde{\mbF}$-cotilting module and $\widetilde{L}\in \cogen^1_{\widetilde{\mbF}}(M) $.
This example is an instance of a more general class of examples which we call special (co)tilting modules systematically studied in section \ref{specialTilt}, for this particular example see subsubsection \ref{ExOf4tuple}.
\end{itemize}
\end{exa}

\paragraph{\bf Acknowledgements:} The first author is supported by the China Scholarship Council. The second author is supported by the Alexander von Humboldt Foundation in the framework of the Alexander von Humboldt Professorship endowed by the German Federal Ministry of Education and Research. The second author also wishes to thank William Crawley-Boevey who contributed Lemma 2.2.

\section{On categories generated or cogenerated by a module}

We fix a finite-dimensional algebra $\La$ (over a field $K$) and denote by $\Lamod$ the category of finitely generated (or equivalently, finite-dimensional) left $\La$-modules. Let $M\in \Lamod$ and $\Ga = \End_{\La} (M)$ be its endomorphism ring. Then $M$ can be naturally viewed as a left $\Ga$-module. We write ${}_{\Ga}M$ when we consider $M$ as a left ${\Ga}$-module. We will study the following four contravariant functors 
\[ 
\begin{aligned}
\Hom_{\La}(-,M) &\colon \Lamod \longleftrightarrow \Gamod  \colon \Hom_{\Ga}(-,M)\\
\kdual \Hom_{\La}(M,-) &\colon \Lamod \longleftrightarrow \Gamod \colon \kdual \Hom_{\Ga}(M,-) 
\end{aligned}
\]
where $D=\Hom_K(-,K)$ is the standard $K$-dual functor. 

In order to keep the formulas and diagrams in reasonable length we will often use the conventions $(-,{_{\La}M}):=\Hom_{\La}(-,M)$ and $\kdual (_{\La}M,-):=\kdual \Hom_{\La}(M,-)$. If there is no ambiguity we may omit the subscript and write $(-,{_{\La}M})$ (or $(-,{_{\Ga}M})$) as $(-,M)$. 

We begin with the Yoneda embedding which is known as \emph{projectivization} (\cite{ARSm}).

\begin{lem}\label{projectivization} $($\cite[Lemma 3.3]{ASoII}\cite[Proposition 2.1]{ARSm}$)$
Let $M\in \Lamod$ and $\Ga=\End_{\La}(M)$. \begin{itemize}
\item[(1)] $((-,{_\La}M), (-,{_\Ga}M))$ is an adjoint pair of contravariant functors and it restricts to a duality
\[
(-,{_\La}M):\add({_\La}M) \longleftrightarrow \add(\Ga)=\mcP(\Ga):(-,{_\Ga}M).
\]

\item[(2)] $(\kdual (_{\La}M,-), \kdual (_{\Ga}M,-))$ is an adjoint pair of contravariant functors and it restricts to a duality
\[
\kdual (_{\La}M,-):\add({_\La}M) \longleftrightarrow \add(\kdual\Ga)=\mcI(\Ga):\kdual (_{\Ga}M,-).
\]
\end{itemize}
\end{lem}

For every non-negative integer $k$ we associate to a module $M\in \Lamod$ two full subcategories of $\Lamod$
\[ 
\begin{aligned}
\cogen^k(M) &:= \bigg\{ N \;\bigg\vert  \;\begin{aligned} \exists \; \text{exact seq.}\; 0\to & N \to M_0 \to \cdots \to M_k \; \text{with}\; M_i \in \add (M),\;\text{and s.t.}\\  (M_k,M)\to & \cdots \to (M_0,M) \to (N,M) \to 0 \;\;\text{is exact}\end{aligned} \bigg\} \\ 
\gen_k (M) &:= \bigg\{ N \;\bigg\vert  \;\begin{aligned} \exists \;\text{exact seq.}\;  M_k & \to \cdots \to M_0 \to N \to 0\; \text{with}\; M_i \in \add (M),\;\text{and s.t.}\\  (M, M_k)\to & \cdots \to (M, M_0) \to (M, N) \to 0 \;\;\text{ is exact}\end{aligned} \bigg\}.
\end{aligned}
\]
Recall that a map $f:N\to M_0$ with $M_0\in\add({_{\La}}M)$ is called a left $\add(M)$-approximation if the map $\Hom_{\La}(f,M):\Hom_{\La}(M_0,M)\to\Hom_{\La}(N,M)$ is an epimorphism, and this approximation is called minimal if any endomorphism $\theta:M_0\to M_0$ satisfying $\theta f=f$ is an automorphism. It is well-known that every left $\add(M)$-approximation has a minimal version which is unique up to isomorphism, see \cite[Theorem 2.4]{ARSm}. Dually, we can define right (minimal) $\add(M)$-approximation.  
We define $\cogen^{\infty} (M)$ to be the full subcategory consisting of modules $N$ such that there exists an exact sequence $0 \to N \xrightarrow{f_0} M_0 \xrightarrow{f_1} M_1 \cdots \xrightarrow{f_n} M_n \to \cdots  $ such that $f_i$ factors through $\coKer f_{i-1} \to M_{i}$ which is a minimal left $\add(M)$-approximation for every $i\geq 0$. The definition of $\gen_{\infty}(M)$ is dual. 

The following lemma will be used frequently, the case $k=0$ is well known and can be found in \cite[Lemma VI 1.8]{ASS}.

\begin{lem} \label{cogen-k} Let $1\leq k \leq \infty$. 
 \begin{itemize}
\item[(1)]
The following are equivalent for $N\in\Lamod$. 
\begin{itemize}
\item[(1a)] $N \in \cogen^k (M)$.
\item[(1b)] The natural map $N \to \Hom_{\Ga}(\Hom_{\La}(N,M),M)=((N,M),M)$, $n \mapsto (f \mapsto f(n))$ is an isomorphism and $\Ext^i_{\Ga} (\Hom_{\La}(N,M),M)=0$ for $1\leq i \leq k-1$.
\end{itemize}
\item[(2)]
The following are equivalent for $N\in\Lamod$. 
\begin{itemize}
\item[(2a)] $N \in \gen_k(M)$.
\item[(2b)] The natural map $\kdual (M,\kdual (M,N)) \cong \Hom_{\La} (M,N)\otimes_{\Ga} M\to N$, $f\otimes m \mapsto f(m)$ is an isomorphism and $\Ext^i_{\Ga}(M, \kdual (M,N))=0$
for $1\leq i \leq k-1$.
\end{itemize}
\end{itemize}
\end{lem}

\begin{proof}
\begin{itemize}
\item[(1)] Let $N\in \cogen^k(M)$, that means we have an exact sequence 
\[ 0 \to N \to M_0 \to \cdots \to M_k \]
with $M_i \in \add (M)$ and such that the functor $\Hom_{\La} (-, M)$ is exact on it, i.e., we get an exact sequence 
\[ (M_k , M) \to \cdots \to (M_0, M) \to  (N,M) \to 0 .\]
This sequence is a projective resolution of $\Hom_{\La} (N,M)$ as a left $\Ga$-module. 
Applying the functor $\Hom_{\Ga} (-,M)$ to it yields a complex 
\[
0 \to ((N,M),M) \to ((M_0,M),M) \to \cdots \to ((M_k,M),M).
\]
Now, consider the natural map $N \to \Hom_{\Ga}(\Hom_{\La}(N,M),M)$, this gives a commutative diagram, 
\[
\xymatrix{
0 \ar[r] &((N,M),M)\ar[r] & \ar[r]((M_0,M),M)&\cdots \ar[r]&((M_k,M),M)\\
0 \ar[r]&N \ar[u]\ar[r] & M_0 \ar[u]\ar[r] & \cdots \ar[r]& M_k \ar[u] 
}
\]
The map $M' \to \Hom_{\Ga}(\Hom_{\La}(M',M),M)$ is an isomorphism for $M' \in \add (M)$ because it is in the case of $M' =M$. This implies that all vertical maps are isomorphisms, in particular $N \to \Hom_{\Ga}(\Hom_{\La}(M,N),M)$ is an isomorphism and since the second row is exact, the complex in the first row is also exact. This implies $\Ext^i_{\Ga} (\Hom_{\La}(N,M), M) =0$ for $1 \leq i \leq k-1$. \\
For the other direction, by Lemma \ref{projectivization} (1) we can take a projective resolution of $\Hom_{\La} (N,M)$ as a left $\Ga$-module as follows
\[ (M_k, M) \to \cdots \to (M_0,M) \to (N,M) \to 0 \]
and apply $\Hom_{\Ga} (-,M)$ to compute $\Ext^i_{\Ga} (\Hom_{\La}(N,M), M)$, $1 \leq i \leq k-1$. 
Since by assumption $\Ext^i_{\Ga} (\Hom_{\La}(N,M), M)=0$, $1 \leq i \leq k-1$ and 
$N \to \Hom_{\Ga}(\Hom_{\La}(N,M),M)$ is an isomorphism. The complex gives an exact sequence
\[ 0 \to N \to M_0 \to \cdots \to M_k. \]
If we apply $\Hom_{\La} (-,M)$ to this sequence we get the projective resolution from before, so it is exact which shows that $N$ is in $\cogen^k(M)$. 
\item[(2)] By using the facts that $N\in \gen_k (M)$ if and only if $\kdual N\in \cogen^k (\kdual M)$ and $\End_{\La^{op}}(\kdual M)\cong \End_{\La}(M)^{op}$, we see that the statement (2) can be deduced from the right module version of (1). 
\end{itemize}
\end{proof}

\begin{cor} 
For $1 \leq k \leq \infty$, the categories $\cogen^k(M)$ and $\gen_k(M)$ are closed under direct sums and summands. Furthermore, we have 
\[ 
\cogen^{\infty} (M) = \bigcap_{1\leq k <\infty } \cogen^k(M), \quad 
\gen_{\infty} (M) = \bigcap_{1\leq k <\infty } \gen_k(M).
\]
\end{cor}

We will need the following useful lemma which already appeared for the specific situation of a relative cotilting module in \cite[Lemma 3.3 (b)]{ASoII} and \cite[Proposition 3.7]{ASoII}. 
For a finite-dimensional algebra $\La$ we write $\nu_{\La} =\kdual (-, \La), \nu_{\La}^-=(\kdual \La ,-)$ for the Nakayama functors (cf. \cite{ASS}). 

\begin{lem} \label{fullfaith} Let $M\in \Lamod$ and $\Gamma = \End_{\La} (M)$.  
\begin{itemize}
\item[(1)] A module $X\in \cogen^1(M)$ if and only if the natural map
\[\Hom_{\La}(Y,X) \to \Hom_{\Ga} ((X,M), (Y,M))\] 
is an isomorphism for all $Y\in \Lamod$. Furthermore, in this case we have   
\[\nu_{\Gamma} (X,M) = \kdual ((X,M), (M,M)) \cong \kdual (M,X). \]
Dually, a module $Y \in \gen_1(M)$ if and only if the natural map 
\[\Hom_{\La} (Y,X) \to \Hom_{\Ga} (\kdual(M,X), \kdual (M,Y))\] 
is an isomorphism for all $X\in \Lamod$. Furthermore, in this case 
\[\nu^-_{\Ga} \kdual (M,Y) = (\kdual (M,M), \kdual (M,Y)) \cong  (Y,M).\]
\item[(2)] For $k \geq 1$, $X\in \cogen^{k+1}(M)$ if and only if the natural maps 
\[
\Ext^i_{\La} (Y,X) \to \Ext^i_{\Ga} ((X,M), (Y,M)),\ 0\leq i \leq k
\]
are isomorphisms for all $Y \in \bigcap_{i=1}^k\Ker \Ext^i_{\La}(-,M)$.
Dually, $Y \in \gen_{k+1} (M)$ if and only if the natural maps 
\[ \Ext^i_{\La}(Y,X) \to \Ext^i_{\Ga}(\kdual (M,X), \kdual (M,Y)),\ 0\leq i \leq k\]
are isomorphisms for all $X \in \bigcap_{i=1}^k \Ker \Ext^i_{\La}(M,-)$.
\end{itemize}
\end{lem}
\begin{proof}
\begin{itemize}
\item[(1)] Assume $X\in \cogen^1(M)$, then there exists an exact sequence $0 \to X \to M_0\to M_1$ such that $M_i \in \add (M)$ and $\Hom_{\La} (-,M)$ is exact on it. We apply $\Hom_{\La} (Y,-)$ to get an exact sequence 
\[
0 \to \Hom_{\La} (Y,X) \to \Hom_{\La} (Y,M_0) \to \Hom_{\La} (Y,M_1).
\]
Now, we consider the commutative diagram
\[
\xymatrix{ 
0 \ar[r] & (Y,X) \ar[r]\ar[d]_{(-,M)} & (Y,M_0) \ar[d]_{(-,M)}^{\cong}\ar[r] & (Y,M_1)\ar[d]_{(-,M)}^{\cong} \\
0 \ar[r] & ((X,M),(Y,M))\ar[r] & ((M_0,M),(Y,M)) \ar[r] & ((M_1,M),(Y,M)).
}
\]
The second row also can be obtained by applying first $\Hom_{\La} (-,M)$ then $\Hom_{\Ga} (-,(Y,M))$ to the exact sequence $0 \to X \to M_0\to M_1$, so it remains exact. The induced isomorphism of the kernels is the map in the claim. Conversely, by taking $Y=\La$ we obtain a natural isomorphism $X\xrightarrow{\cong} ((X,M),M)$ which implies $X\in \cogen^1(M)$.
\item[(2)] Assume $X\in \cogen^{k+1}(M)$, then we have an exact sequence $0 \to X \to M_0\to \cdots \to M_{k+1}$ such that $M_i \in \add (M)$ and $\Hom_{\La} (-,M)$ is exact on it. Applying $\Hom_{\La} (-,M)$ yields an exact sequence $(M_{k+1}, M) \to \cdots \to (M_0,M)\to (X,M) \to 0$ which is a projective resolution of $(X,M)$ as a left $\Ga$-module. Now assume $Y\in \bigcap_{i=1}^k\Ker \Ext^i_{\La}(-,M)$. To compute $\Ext^i_{\Ga} ((X,M), (Y,M))$ for $1 \leq i \leq k$ we apply $\Hom_{\Ga} (-,(Y,M))$ to this projective resolution and delete the term $((X,M),(Y,M))$ to get a complex $\cdots  \to 0 \to ((M_0,M),(Y,M))\to ((M_1,M),(Y,M))\to \cdots \to ((M_{k+1},M),(Y,M)) \to 0 \to \cdots$ which fits into the following commutative diagram
\[
\xymatrix @C=0.3cm{ 
\cdots \ar[r]  & 0\ar[r] & (Y,M_0) \ar[r]\ar[d]_{(-,M)}^{\cong} & (Y,M_1) \ar[r]\ar[d]_{(-,M)}^{\cong} &  \cdots \ar[r] & (Y,M_{k+1})\ar[d]_{(-,M)}^{\cong} \ar[r] & 0 \ar[r]& \cdots \\
\cdots \ar[r] & 0 \ar[r]  & ((M_0,M),(Y,M)) \ar[r] & ((M_1,M),(Y,M)) \ar[r] & \cdots \ar[r] & ((M_{k+1},M),(Y,M)) \ar[r] & 0 \ar[r] & \cdots
}
\]
where the complex in the first row is obtained by applying $\Hom_{\La} (Y,-)$ to $0 \to X \to M_0\to \cdots \to M_{k+1}$ and deleting the term $(Y,X)$. Our assumption $Y\in \bigcap_{i=1}^k\Ker \Ext^i_{\La}(-,M)$ implies that the i-th cohomology of the first row is $\Ext^i_{\La} (Y,X)$. Now the isomorphism of the two complexes induces the claimed natural isomorphisms. To prove the other implication, just take $Y=\La$.
\end{itemize}
\end{proof}

We also prove the following simple criterion. 
\begin{lem} \label{isocogen1} Let $M, N\in \Lamod$ and $\Gamma = \End(M)$. If $N$  and $\Hom_{\Ga}(\Hom_{\La}(N,M),M)$ are isomorphic as ${\La}$-modules, 
then we have $N\in \cogen^1({_{\La}}M)$.
\end{lem}

\begin{proof}
The essential image of the functor $(-,M)$ is contained in $\cogen (M)$ since if $Y=(Z,M)$, then we may choose a projective cover $P\to Z$ and apply $(-,M)$ to see that $Y\in \cogen (M)$. \\
This means $N \cong ((N,M),M)\in \cogen (M)$. This implies that the natural map 
$N \to ((N,M),M)$ mapping $n \mapsto (f\mapsto f(n))$ is a monomorphism. Since both vector spaces have the same dimension it is an isomorphism. This implies by Lemma \ref{cogen-k} that $N \in \cogen^1(M)$. 
\end{proof}

\subsection{Faithfully balanced modules}
 
Faithfully balanced modules can be defined for any ring. For finite-dimensional algebras, Lemma \ref{cogen-k} allows us to give the following internal definition.
 
 \begin{dfn}\label{fb} We call a finitely generated (left or right) $\La$-module $M$ faithfully balanced if $\La \in \cogen^1(M)$.  
 \end{dfn}
 


The following surprising and also well-known result says every module becomes faithfully balanced when considering as a module over its endomorphism ring.

\begin{lem} $($\cite[Proposition 4.12]{AF} \cite[Lemma 2.2]{ASo-Gor}$)$ Let $M\in \Lamod$ and $\Ga = \End_{\La}(M)$ and consider $M$ as a left $\Ga$-module. Then ${}_{\Ga}M$ is faithfully balanced. 
\end{lem}

In \cite{BS}, a faithfully balanced module is also known as a module of faithful dimension at least 2. The following lemma (the same as \cite[Proposition 2.2]{BS}), which characterizes modules of faithful dimension at least $k+1$, can be obtained as an immediate consequence of Lemma \ref{cogen-k}.   

\begin{lem}\label{lagencogen} 
The following are equivalent for every $1\leq k \leq \infty$. 
\begin{itemize}
\item[(1)] $\La \in \cogen^{k}(M)$. 
\item[(2)] The natural map $\La \to \End_{\Ga}(M)$ is an isomorphism and $\Ext^i_{\Ga}(M,M)=0$, $1\leq i \leq k-1$. 
\item[(3)] $\kdual \La \in \gen_k(M)$.
\end{itemize}
\end{lem}

\begin{proof} The equivalence between (1) and (2) is a special case of Lemma \ref{cogen-k}. 
The equivalence to (3) follows again by seeing that the equivalence between (1) and (2) also works for right modules. Then pass with the duality from the right module statement for (1) to (3). 
\end{proof}

The following lemma plays a fundamental role in this paper.
\begin{lem} \label{duality}$($cf. \cite[Proposition 5.1]{X}$)$ 
Let $M$ be a faithfully balanced $\La$-module and $\Ga=\End_{\La}(M)$. Then the functors $(-,{_{\La}}M) \colon \Lamod \longleftrightarrow \Gamod  \colon (-,{_{\Ga}}M)$ restrict to a duality of categories 
\[\cogen^1({}_{\La}M) \longleftrightarrow \cogen^1({}_{\Ga}M).\] 
They restrict further to a duality 
\[ 
\cogen^k({}_{\La}M) \longleftrightarrow \cogen^1({}_{\Ga}M) \cap \bigcap_{i=1}^{k-1} \Ker \Ext^i_{\Gamma} (-,{}_{\Ga}M).
\]
Dually, the functors $\kdual({_{\La}}M,-) \colon \Lamod \longleftrightarrow \Gamod \colon \kdual ({_{\Ga}}M,-) $ restrict to a duality of categories 
\[\gen_1({}_{\La}M) \longleftrightarrow \gen_1({}_{\Ga}M).\] 
They restrict further to a duality 
\[ 
\gen_k({}_{\La}M) \longleftrightarrow \gen_1({}_{\Ga}M) \cap \bigcap_{i=1}^{k-1} \Ker \Ext^i_{\Gamma} ({}_{\Ga} M,-).
\]
\end{lem}
\begin{proof} 
By Lemma \ref{fullfaith} the functor $(-, {}_{\La}M)$ is fully faithful on $\cogen^1({}_{\La}M)$. Let ${}_{\La\text{-}\Ga} M$ be a $\La$-$\Ga$-bimodule and ${}_{\La}N$ a left $\La$-module, and ${}_{\Ga}N'$ a left $\Ga $-module. We denote by 
$\al_N\colon N \to ((N,M),M)$ and $\al_{N'}\colon N' \to ((N',M),M)$ the two natural maps. Then the compositions 
\[ 
\xymatrix@R=0.3cm{
(N,M) \ar[r]^-{\al_{(N,M)}} & (((N,M),M),M) \ar[r]^-{(\al_N,M)} & (N,M)\\
(N,M) \ar[r]^-{\al_{(N',M)}} & (((N',M),M),M) \ar[r]^-{(\al_{N'},M)} & (N',M)
}
\]
are both identities, since by Lemma \ref{projectivization} the functors $(-, {_{\La}}M)$ and $(-,{_{\Ga}}M)$ form an adjoint pair. Therefore, if $\al_N$ (resp. $\al_{N'}$) is an isomorphism, then so is $\al_{(N,M)}$ (resp. $\al_{(N',M)}$).
Since $M$ is faithfully balanced the dualities follow from Lemma \ref{cogen-k}. 
\end{proof}

\begin{rem} We have already seen in Lemma \ref{cogen-k} that $\cogen^1(M)$ consists of the modules $N$ such that $\al_N$ is an isomorphism. It is also straightforward to see that $\cogen(M)$ consists of the modules $N$ with $\al_N$ a monomorphism. \\
If we now consider a faithfully balanced $\La$-module $M$,  $\Ga = \End_{\La}(M)$ and $\Bild (-,M)$ the essential image of the functor $(-,M)$, then we have 
\[ \cogen^1 ({}_{\Ga}M) \subseteq \Bild (-,M) \subseteq  \cogen ({}_{\Ga}M)  .\]
Let $\Bild (-,M)_{\oplus}$ be the full subcategory of $\Ga$-$\modu$ whose objects are summands of modules in $\Bild (-,M)$. Then it is easy to see from the previous proof that $\Bild (-,M)_{\oplus}$ consists of those modules $N$ such that $\al_N$ is a split monomorphism. \\
If ${}_{\La}M$ is a cogenerator, then $\Bild (-,M) = \cogen^1({}_{\Ga} M)$ and in particular $\Bild (-,M)$ is closed under summands in this case. 
\end{rem}

\begin{cor} Let $k \geq 1$. Let $M\in \Lamod$ be faithfully balanced and assume $\id {}_{\Gamma}M \leq k-1$, then we have 
 \[ 
\cogen^k(M) = \cogen^{k+1}(M) = \cdots = \cogen^{\infty} (M).
\]
\end{cor}

\begin{cor}\label{fbpdid}
Let $k\geq 1$ and $M$ be a faithfully balanced $\La$-module and $\Ext^i_{\La} (M,M)=0$ for $1\leq i \leq k-1$. Then we have
\begin{itemize}
\item[(1)] The functors $(-,{_{\La}}M), (-,{_{\Ga}}M)$ restrict to a duality 
\[  \{ M'\in \add({}_{\La}M) \mid \pd M' \leq k\}  \longleftrightarrow \{ P \in \add (\Ga ) \mid  \Omega_M^{-(k+1)} P=0\}.\]
\item[(2)] The functors $\kdual({_{\La}}M,-) , \kdual ({_{\Ga}}M,-) $ restrict to a duality 
\[  \{ M'\in \add({}_{\La}M) \mid \id M' \leq k\}  \longleftrightarrow \{ J \in \add (\kdual \Ga ) \mid  \Omega_M^{(k+1)} J=0\}.\]
\end{itemize}
\end{cor}
\begin{proof}
It is straightforward to check that the duality from Lemma \ref{duality} restricts to these equivalences. \end{proof}
We also recall the following result of Wakamatsu. 
\begin{thm} Let $M$ be a faithfully balanced $\La$-module and $\Ga=\End_{\La}(M)$. Assume $M$ is self-orthogonal both as left $\La$-module and right $\Ga$-module $($i.e., $\Ext^{\! >0}_{\La} ({}_{\La}M,{}_{\La}M)=0=\Ext^{\! >0}_{\Ga} ({}_{\Ga}M,{}_{\Ga}M)$ $)$. Then we have the following 
\begin{itemize} 
\item[(1)] If $\id{}_{\La}M <\infty $ and $\id {}_{\Ga}M <\infty$ $($or resp.  $\pd{}_{\La}M , \pd {}_{\Ga}M <\infty$ $)$, then they are equal. 
\item[(2)] If $\id{}_{\La}M , \id {}_{\Ga}M <\infty$ $($or resp.  $\pd{}_{\La}M , \pd {}_{\Ga}M <\infty$ $)$, then we have $\lvert \Ga \rvert =\lvert {}_{\La}M\rvert =\lvert {}_{\Ga}M\rvert =\lvert \La\rvert $ and $M$ is cotilting $($or resp. tilting$)$. 
\end{itemize}
\end{thm}

\begin{proof} 
(1) is the main result in \cite{Wa}. If  $\id{}_{\La}M , \id {}_{\Ga}M <\infty$, then it follows from the previous corollary that the is a $k$ such that $\Omega^k_M\kdual \La=0$, this implies that $M$ is cotilting and in particular $\lvert {}_{\La} M \rvert = \lvert \La \rvert$. 
\end{proof}

\begin{exa} Assume $\La $ is a self-injective algebra. Then a finite-dimensional $\La$-module is a faithfully balanced if and only if it is a cogenerator. In particular, any faithfully balanced module has at least $\lvert \La \rvert $ summands. 
\end{exa}

\section{Dualizing summands and the Auslander-Solberg assignment}

Auslander and Solberg introduced (in \cite[section 2]{ASoIII}) the following notion.
\begin{dfn} Let $M, L\in \Lamod$ and assume $M$ is a summand of $L$. We say $M$ is a \emph{dualizing summand} of $L$ if $L \in \cogen^1 (M)$. For $k\geq 0$, we say $M$ is a $k$-dualizing summand if $L \in \cogen^k(M)$. Thus a dualizing summand of $L$ is the same as a $1$-dualizing summand of $L$.
\end{dfn} 

By using the duality from Lemma \ref{duality} it is easy to find modules having a given faithfully balanced module as a dualizing summand. 

\begin{cor}  \label{gendualsum} Let $M$ be a faithfully balanced $\La$-module and $\Gamma =\End(M)$. Then the assignments $G\mapsto (G, M), L \mapsto (L,M)$ give inverse bijections between 
\begin{itemize}
\item[(1)] isomorphism classes of ${}_{\La}G\in \cogen^1(M)$ with $\La \in \add (G)$, and 
\item[(2)] isomorphism classes of modules $L\in \Gamod$ having ${}_{\Ga}M$ as a dualizing summand. 
\end{itemize}
\end{cor}

\begin{lem}\label{dualsum-cogen}
Let $M, L\in \Lamod$, $\Ga=\End_{\La}(M)$ and assume $M$ is a summand of $L$. Then $M$ is a dualizing summand of $L$ if and only if $\cogen^1(L)=\cogen^1(M)$.
\end{lem}
\begin{proof}
The ``if'' part is obvious. For the ``only if'' part, assume $M$ is a dualizing summand of $L$ and $X\in \cogen^1(L)$. Then there exists an exact sequence 
$0\to X \to L^X_0\to  L^X_1$ with $L^X_i\in \add(L)$ and $(-,L)$ exact on it. We apply 
$(-,{_{\La}}M)$ to it and the resulting complex remains exact, since $M \in \add(L)$. Now apply $(-,{_{\Ga}}M)$ to see $X \cong ((X,M), M)$. This proves $\cogen^1(L) \subseteq \cogen^1(M)$. To prove $\cogen^1(M) \subseteq \cogen^1(L)$, take any $Y\in \cogen^1(M)$ and take the minimal left $\add(L)$-approximations $f:Y\to L^Y_0$ and $\coKer f \to L^Y_1$. Then we get a complex $0\to Y\xrightarrow{f} L^Y_0 \to L^Y_1$. We need to show it is exact. By construction, we will obtain an exact sequence $(L^Y_1,M)\to (L^Y_0,M)\to (Y,M)\to 0$ after applying $(-,{_{\La}}M)$. Now apply $(-,{_{\Ga}}M)$ to yield an exact sequence $0\to ((Y,M),M)\to ((L^Y_0,M),M)\to ((L^Y_1,M),M)$ which is naturally isomorphic to the complex $0\to Y\xrightarrow{f} L^Y_0 \to L^Y_1$, as desired.
\end{proof}

There is another subcategory of $\Lamod$ that is closely related to $\cogen^k(M)$:
\[\copres^k(M) := \{ N\ |\ \exists \; \text{exact seq.}\; 0\to  N \to M_0 \to \cdots \to M_k \; \text{with}\; M_i \in \add (M) \}.\]
This subcategory is useful in characterizing tilting modules (see \cite{W}). It follows from the definitions that  $\cogen^0(M)=\cogen(M)=\copres^0(M)$ and $\cogen^k(M)\subseteq \copres^k(M)$ for any $M$ and $k\geq 1$. In particular, if $M$ is injective then $\cogen^k(M)=\copres^k(M)$ for any $k\geq 0$. We observe the following 
\begin{lem} \label{dualSummand} 
Let $M, N\in \Lamod$ and $L=M\oplus N$. 
For $k\geq 1$, if $N \in \cogen^k(M)$ $($i.e., $M$ is a $k$-dualizing summand of $L$ $)$, then $M$ is faithfully balanced if and only if $L$ is faithfully balanced. In this case we have $\cogen^k(M)=\cogen^k(L)$.  
Furthermore, if additionally $\copres^k(L)=\cogen^k(L)$ , then we also have $\copres^k(M)=\cogen^k(M)$. 
\end{lem}

\begin{proof} According to Lemma \ref{dualsum-cogen}, we may assume $k>1$. Since $L \in \cogen^k(M)\subseteq \cogen^1(M)$, it follows from Lemma \ref{dualsum-cogen} that $\cogen^1(L)=\cogen^1(M)$ and hence $M$ is faithfully balanced if and only if $L$ is faithfully balanced. Let us from now on assume that $M, L$ are faithfully balanced. We want to see that $\cogen^k(L) =\cogen^k(M)$. Let $\Ga =\End_{\La}(M)$. Since $L \in \cogen^1(M)$ we can find a generator $G\in\Gamod$ such that $L=(G,M)$ by Corollary \ref{gendualsum}. We observe that $L \in \cogen^{k} (M)$ implies $ \Ext^i_{\Ga}((L,M),M)=\Ext^i_{\Ga}(G,M)=0$ for $1 \leq i \leq k-1$. In other words ${}_{\Ga}M \in \bigcap_{i=1}^{k-1} \Ker \Ext^i_{\Ga}(G,-)$. But since $G$ is a generator we have that $\gen_{k+1} (G) = \Gamod$. We set $B= \End_{\La} (L) \cong \End_{\Ga} (G)^{op}$ and take $X \in \cogen^k({}_{\La}M)$. Now, observe 
$(X,L) \cong (((X,M),M), (G,M)) \cong (G, (X,M))$ is an isomorphism of left $B$-modules. 
The dual statement in Lemma \ref{fullfaith} (2) gives that we have natural isomorphisms 
\[ \Ext^i_{\Ga}((X,M),M) \to \Ext^i_B ((G,(X,M)),(G,M)) \cong \Ext^i_B((X,L),L) \]
for $1 \leq i \leq k-1$. This implies by Lemma \ref{cogen-k} that $\cogen^{k} (M) = \cogen^{k} (L)$. Furthermore, since 
$\cogen^k(M) \subseteq \copres^k(M) \subseteq \copres^k(L)$ are always fulfilled, an equality $\cogen^k(M) = \copres^k(L)$ implies they are all equal. 
\end{proof}

\begin{rem}\label{MoritasResult}
Let $M$ be a faithfully balanced module. 
Morita \cite[Theorem 1.1]{Mo58} has shown that for every indecomposable module $N$ the following are equivalent: 
\begin{itemize}
\item[(1)]
$M\oplus N$ is faithfully balanced, 
\item[(2)] $N \in \gen (M)$ or $N \in \cogen (M)$.
\end{itemize}
In particular, $M \oplus P\oplus I$ is faithfully balanced for every projective module $P$ and injective module~$I$.
\end{rem}

\begin{exa}
Let $H$ be a cogenerator, then every summand of $H$ of the form $\kdual \La\oplus X$ is a $k$-dualizing summand for every $k\geq 0$.  
\end{exa}

Now we look at triples $(\La , M, G)$ where $\La $ is a finite-dimensional algebra and $M$ and $G$ are finite-dimensional left $\La$-modules. 
We define the following equivalence relation between these triples: $(\La , M, G) $ is equivalent to $(\La', M', G' )$ if there is a Morita equivalence $\La$-$\modu \to \La'$-$\modu$ restricting to equivalences $\add (M) \to \add (M')$ and $\add (G) \to \add (G')$. We denote by $[\La ,M, G] $ the equivalence class of a triple. 

\begin{dfn} We consider the following assignment 
\[ [\La , M, G] \mapsto [\Gamma, N, L]  \]
with $\Gamma = \End(M) $, $N={}_{\Ga} M$, $L = (G, M)$ and call this the Auslander-Solberg assignment.
\\
There is a dual assignment 
\[ [\La , M, H] \mapsto [\Gamma, N, R]  \]
with $\Ga,N$ as before and $R=\kdual (M,H)$ which we call the dual Auslander-Solberg assignment.
\end{dfn}

From Corollary \ref{gendualsum} we see that the Auslander-Solberg assignment gives a one-to-one correspondence between the following 
\begin{itemize}
\item[(1)] $[ \La, M, G]$ with $\La \in \add (G)$, $G \in \cogen^1(M)$,
\item[(2)] $[ \Gamma , N, L]$ with $N \in \add (L)$, $\Gamma \oplus L \in \cogen^1(N)$. 
\end{itemize}
The previous bijection has an obvious dual version using the dual Auslander-Solberg assignment and $\gen $, $H$ and $R$ instead of $\cogen $, $G$ and $L$, respectively. 

We are going to refine this assignment, our first refinement needs the following definition. Here we denote for $\Ga$-modules $N$ and $X$ by $\Omega_N  X$ the kernel of the minimal right $\add (N)$-approximation $N_X \to X$. For $k\geq 1$ we define inductively  $\Omega_N^k X:= \Omega_NX$ if $k=1$ and 
$\Omega_N^kX:=\Omega_N (\Omega_N^{k-1} X)$ for $k\geq 2$. Dually, we define $\Omega_N^-X $ as the cokernel of a minimal left $\add (N)$-approximation 
$X \to N^X$ and $\Omega^{-k}_NX$ inductively as before.

\begin{dfn} Let $k$ be a non-negative integer and $L,N, R\in \Lamod$. An exact sequence 
\[  0\to L \to N_0 \to N_{1} \to \cdots  \to N_k \to R \to 0 \]
is called a $k$-$\add(N)$-dualizing sequence from $L$ to $R$ if
\begin{itemize}
    \item[(i)] $N_i \in \add (N)$ for $i\in \{ 0,\ldots , k\}$,
    \item[(ii)] the functors $(-,N)$ and $\kdual (N, -)$ are exact on it, 
    \item[(iii)] $\add(R) = \add(N\oplus \Omega_N^{-(k+1)} L )$ and  $\add(L) = \add(N\oplus \Omega_N^{k+1} R )$.
\end{itemize} 
In this case we say $L$ is the left end and $R$ is the right end of this $k$-$\add(N)$-dualizing sequence.
\end{dfn}
This has the following consequences for the ideal quotient categories $\add(L)/\add (N), \add (R) / \add (N)$ (for the definition of an ideal quotient category see \cite[A.3]{ASS}): 
\begin{lem}\label{equiv}
Let $0\to L \to N_0 \to N_{1} \to \cdots  \to N_k \to R \to 0$ be a $k$-$\add(N)$-dualizing sequence from $L$ to $R$ for some $k\geq 0$ in $\Lamod$. Then we have an equivalence 
\[\Omega_N^{-(k+1)}:\add(L)/\add(N) \longleftrightarrow \add(R)/\add(N):\Omega_N^{k+1}.\]
\end{lem}
\begin{proof}
We claim that given a short exact sequence $\eta: 0\to U\xrightarrow{f} N_0\xrightarrow{g} V\to 0$ with $N_0\in \add(N)$ and such that the functors $(-,N)$ and $(N,-)$ are exact on it, then we have an equivalence $\Omega_N^{-1}:\add(U)/\add(N)\leftrightarrow \add(V)/\add(N):\Omega_N^{1}$. Take a map $\alpha:X\to Y$ in $\add(U)$ and consider the following commutative diagram 
\[
\xymatrix{\eta^X:& 0\ar[r] & X\ar[r]^{f^X}\ar[d]_{\alpha}& N_0^X\ar[r]^{g^X}\ar@{-->}[d]_{\beta}& \Omega_N^{-1}X\ar[r]\ar@{-->}[d]_{\gamma} & 0\\
\eta^Y:& 0\ar[r] & Y\ar[r]^{f^Y}& N_0^Y\ar[r]^{g^Y}& \Omega_N^{-1}Y\ar[r] & 0}
\]
where $\eta^X$ and $\eta^Y$ are both direct summands of $\eta$ by our assumption. In particular, we have $\Omega_N\Omega_N^{-1}X\cong X$ and $\Omega_N\Omega_N^{-1}Y\cong Y$ in $\add(U)/\add(N)$. Assume there is another map $\beta':N_0^X\to N_0^Y$ such that $\beta' f^X=f^Y \alpha$ and denote by $\gamma'$ the induced map on cokernels. Then we have $(\beta-\beta')f^X=0$ and thus there exists a unique $\theta:\Omega_N^{-1}X\to N_0^Y$ such that $\beta-\beta'=\theta g^X$. It follows that $\gamma-\gamma'=g^Y \theta$. Now assume $\alpha$ factors as $\alpha=\alpha_2\alpha_1$ though an object $N'\in \add(N)$, then since $f^X$ is a left $\add(N)$-approximation there is a map $\phi:N_0^X\to N'$ such that $\alpha_1=\phi f^X$. Thus we have $\alpha=\alpha_2\alpha_1=(\alpha_2\phi)f^X$, and a diagram chasing gives a map $\psi:\Omega_N^{-1}X\to N_0^Y$ such that $\gamma=g^Y \psi$. These proves that the map $\Hom_{\add(U)/\add(N)}(X,Y)\to \Hom_{\add(V)/\add(N)}(\Omega_N^{-1}X,\Omega_N^{-1}Y), \alpha \mapsto \gamma$ is well defined. Similarly, we have a map $\Hom_{\add(V)/\add(N)}(\Omega_N^{-1}X,\Omega_N^{-1}Y)\to \Hom_{\add(U)/\add(N)}(X,Y), \gamma\mapsto \alpha$. Clearly, these two maps are mutually inverse and this proves the claim. Now the lemma follows by induction on $k$.
\end{proof}

Let $X\in \Lamod$ and $k \geq 1$ be an integer. We define $\tau_k X= \tau (\Omega_{\La}^{k-1} X)$ and $\tau_k^- X:=\tau^-(\Omega_{\La}^{-(k-1)} X)$. We occasionally use the conventions $X^{\perp_{1\sim  k}}:= \bigcap_{i=1}^k \Ker \Ext^i (X,-)$ and ${}^{{}_{1\sim  k} \perp} X := \bigcap_{i=1}^k \Ker \Ext^i(-,X)$. 

\begin{lem} \label{easy}
Let $M$ be a faithfully balanced $\La$-module and $\Ga =\End_{\La} (M)$. Then, for $k \geq 1$, the assignment 
$X,Y \mapsto  (X,M) , \kdual (M,Y) $ gives a self-inverse bijection $($up to seeing $X,Y$ as $\La$ or as $\Ga$-modules$)$ between the following sets of pairs of $\La$-modules and $\Ga$-modules 
\[
\{ 
{}_{\La}G,{}_{\La}H \mid 
\begin{aligned}
G=\tau_k^-H \oplus \La &\in \cogen^1(M) \cap  {}^{{}_{1\sim (k-1)} \perp} M \\
H=\tau_k G \oplus \kdual\La &\in \gen_1(M) \cap M^{\perp_{1\sim (k-1)}}
\end{aligned}
\}
\]

and
\[
\{ 
{}_{\Ga} L,{}_{\Ga } R \mid \exists\ \text{ a }k\text{-}\;{}_{\Ga}M\text{-dualizing sequence from }L\text{ to }R\}.
\]
If all modules are basic, we have $\kdual (M, \tau_k G) = \Omega_M^{-(k+1)} (G,M)$.
\end{lem}

\begin{proof} The bijection follows from Lemma \ref{duality} and the observation 
$\nu_{\Ga} (M', M) =\kdual (M, M')$ for every $M'\in\add(M)\subseteq \cogen^1(M)$ from Lemma \ref{fullfaith}. The rest statements are obvious.
\end{proof}

\begin{cor}
Let $G,H$ be as in the bijection of Lemma \ref{easy}, then we have an equivalence \[\tau_k:\underline{\add(G)}\longleftrightarrow \overline{\add(H)}: \tau_k^-,\]
where $\underline{\add(G)}$ (resp. $\overline{\add(H)}$) denotes the projective (resp. injective) stable category, see \cite{AB}. 
\end{cor}

\begin{proof}
This follows from the equivalence of Lemma \ref{equiv} by pre- and postcomposing with $(-,M)$ and $\kdual (M,-)$ and then use the previous Lemma \ref{easy}. 
\end{proof}

\begin{exa} 
Obviously triples $[\La , M, M] $ correspond to triples $[\Ga , N, \Ga ]$ and since $M$ is a generator we conclude that $N$ is a projective $\Ga $-module. The module $M$ is a generator-cogenerator (i.e., $\La \oplus \kdual \La \in \add (M)$) if and only if $N$ is projective-injective. \\
Furthermore, $\add (M)= \add (\tau_k M \oplus \La ) $ and $M$ being a generator-cogenerator with $\Ext^i(M,M)=0$, $1 \leq i \leq k-1$ corresponds to a dualizing sequence 
$0 \to \Gamma \to N_0 \to \cdots \to N_k \to \kdual \Ga \to 0$ with $N_i$ projective-injective and $\mcP( G)$ a projective generator, $\mcI (\Ga )$ an injective cogenerator. But this is equivalent to $\Ga$ being $k$-minimal Auslander-Gorenstein which means by definition $\id {}_{\Ga}\Ga \leq k+1 \leq {\rm domdim} {}_{\Ga}\Ga$. These algebras have been studied by Iyama and Solberg in \cite{IS}.
\end{exa}

Our previous results also enable us to understand all faithfully balanced modules in an easy example.
\begin{exa} Let $\La_n=K(1\to 2\to \cdots \to n)$. Then, the faithfully balanced modules for $\La_2$ are the module which are generator or cogenerators. In general every tilting (and automatically cotilting) $\La_n$-module $T$ that is coming from a slice in the Auslander-Reiten quiver fulfills that every indecomposable module is either cogenerated or generated by $T$. By Remark \ref{MoritasResult} we conclude that every module having $T$ as a summand is faithfully balanced. Clearly, faithfully balanced modules must have $P_1=I_n$ as a summand. But even if a module has a tilting module as a summand it is not necessarily faithfully balanced, for example $T_0=P_1\oplus S_1 \oplus S_3$ is a tilting $\La_3$-module but $P_1 \oplus S_1 \oplus S_2 \oplus S_3$ is not faithfully balanced.
The $21$ faithfully balanced modules for $\La_3$ are: $T_0 \oplus I \oplus P$ for a projective $P$ and an injective $I$, modules with one of the other four tilting modules as a summand (since these four come from slices) and the module $P_2 \oplus I_2 \oplus P_1$. \\
We call two modules $N,M$ equivalent if $\cogen^1(N)=\cogen^1(M)$ and $\gen_1(N)= \gen_1(M)$. Then we consider the partial order on equivalence classes 
\[ 
N \leq M \; \Leftrightarrow \cogen^1(N) \subseteq \cogen^1(M), \gen_1(N) \supseteq \gen_1 (M)  
\]
The Hasse diagram for the $20$ equivalence classes (the $2$ generator-cogenerators are equivalent) of faithfully balanced modules for $\La_3$ is the following. 
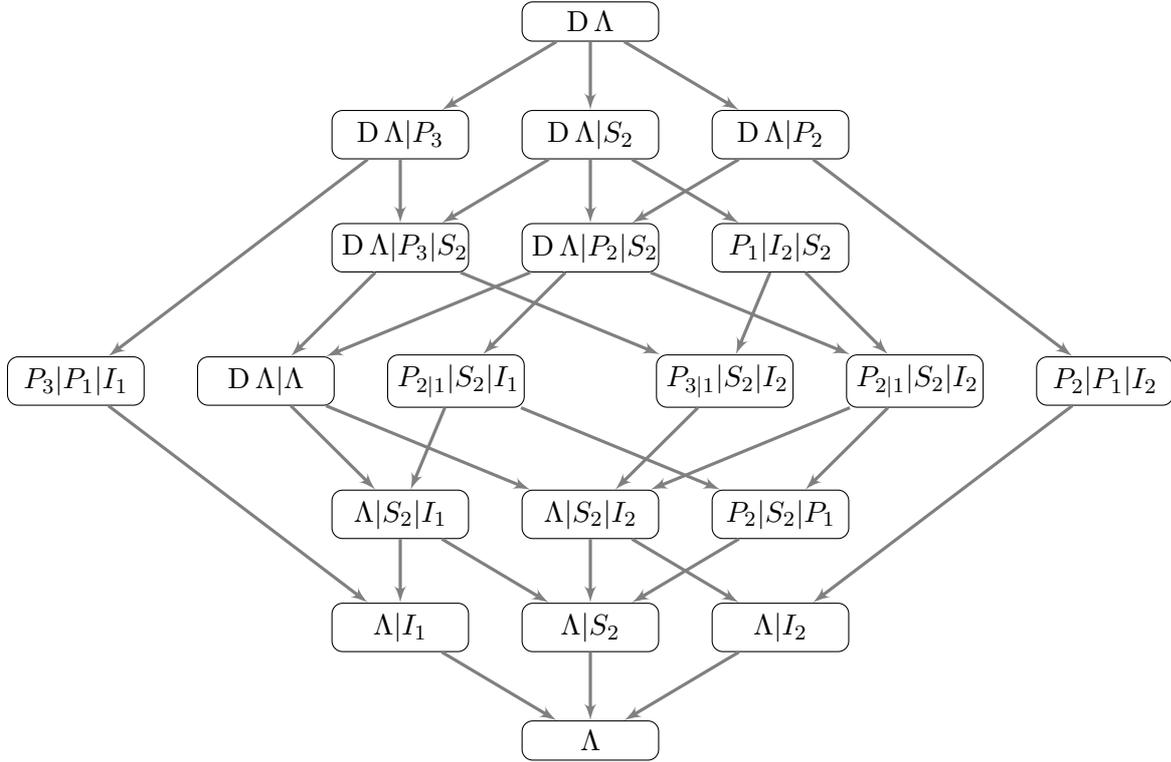
\tikzstyle{block} = [rectangle, draw, fill=white!20,
    text width=4em, text centered, rounded corners, minimum height=1em, node distance=1.5cm]
\tikzstyle{line} = [draw, very thick, color=black!50, -latex']

\begin{tikzpicture}[scale=6, node distance = 1cm, auto]
    \node [block] (DLa) {$\kdual \La $};
    \node [block, below of=DLa] (DLaS2) {$\kdual \La \lvert S_2$};
    \node [block, below of=DLaS2] (DLaP2S2) {$\kdual \La \lvert P_2 \lvert S_2$};
    \node [block, left of=DLaS2, node distance=2.5cm] (DLaP3) {$\kdual \La \lvert P_3$};
     \node [block, right of=DLaS2, node distance=2.5cm] (DLaP2) {$\kdual \La \lvert P_2$};
    \node [block, left of=DLaP2S2, node distance=2.5cm] (DLaP3S2) {$\kdual \La \lvert P_3 \lvert S_2$};
    \node [block, right of=DLaP2S2, node distance=2.5cm] (P1I2S2) {$P_1 \lvert I_2 \lvert S_2$};
    
  \node [block, below left of=DLaP2S2, node distance=2.5cm] (P2S2P1I1) {$P_{2\lvert 1} \lvert S_2\lvert I_1$};
   \node [block, below right of=DLaP2S2, node distance=2.5cm] (P3S2P1I2) {$P_{3\lvert 1} \lvert S_2 \lvert I_2$};
    \node [block, left of=P2S2P1I1, node distance=2.5cm] (LaDLa) {$\kdual\La \lvert \La $};
    \node [block, right of=P3S2P1I2, node distance=2.5cm] (P2S2P1I2) {$P_{2\lvert 1} \lvert S_2 \lvert I_2$};
  
   \node [block, left of= LaDLa, node distance=2.5cm] (P3P1I1) {$P_3 \lvert P_1 \lvert I_1 $};
    \node [block, right of=P2S2P1I2, node distance=2.5cm] (P2P1I2) {$P_{2} \lvert P_1 \lvert I_2$};
  
  \node [block, below left of=P3S2P1I2, node distance=2.5cm] (LaS2I2) {$\La \lvert S_2\lvert I_2$};
  \node [block, left of= LaS2I2, node distance=2.5cm] (LaS2I1) {$\La \lvert S_2 \lvert I_1 $};
    \node [block, right of=LaS2I2, node distance=2.5cm] (P2S2P1) {$P_{2} \lvert S_2 \lvert P_1$};

  \node [block, below of=LaS2I2] (LaS2) {$\La \lvert S_2$};
  \node [block, left of= LaS2, node distance=2.5cm] (LaI1) {$\La \lvert I_1 $};
    \node [block, right of=LaS2, node distance=2.5cm] (LaI2) {$\La \lvert I_2$};
  
   \node [block, below of=LaS2] (La) {$\La$};
  
    \path [line] (DLa) -- (DLaS2);
    \path [line] (DLa) -- (DLaP3);
    \path [line] (DLa) -- (DLaP2);
    \path [line] (DLaS2) -- (DLaP2S2);
  \path [line] (DLaS2) -- (DLaP3S2);
  \path [line] (DLaS2) -- (P1I2S2);
  \path [line] (DLaP3) -- (DLaP3S2);
  \path [line] (DLaP2) -- (DLaP2S2);
  \path [line] (DLaP3) -- (P3P1I1);
  \path [line] (DLaP2) -- (P2P1I2);
   \path [line] (P3P1I1) -- (LaI1);
  \path [line] (P2P1I2) -- (LaI2);
   \path [line] (LaI1) -- (La);
  \path [line] (LaI2) -- (La);
   \path [line] (LaS2) -- (La);
   \path [line] (DLaP3S2) -- (LaDLa);
  \path [line] (DLaP3S2) -- (P3S2P1I2);
     \path [line] (DLaP2S2) -- (LaDLa);
  \path [line] (DLaP2S2) -- (P2S2P1I1);
   \path [line] (DLaP2S2) -- (P2S2P1I2);
   \path [line] (P1I2S2) -- (P3S2P1I2);
   \path [line] (P1I2S2) -- (P2S2P1I2);
   \path [line] (LaDLa) -- (LaS2I1);
   \path [line] (LaDLa) -- (LaS2I2);
   \path [line] (P2S2P1I1) -- (LaS2I1);
   \path [line] (P2S2P1I1) -- (P2S2P1);
   \path [line] (P2S2P1I2) -- (LaS2I2);
   \path [line] (P2S2P1I2) -- (P2S2P1);
   \path [line] (P3S2P1I2) -- (LaS2I2);
   \path [line] (LaS2I1) -- (LaI1);
    \path [line] (LaS2I1) -- (LaS2);
     \path [line] (LaS2I2) -- (LaI2);
    \path [line] (LaS2I2) -- (LaS2);
     \path [line] (P2S2P1) -- (LaS2);
\end{tikzpicture}

where for example $P_{3\lvert 1} \lvert S_2 \lvert I_2 = P_3\oplus P_1\oplus S_2\oplus I_2$.
\end{exa}

\section{Combining the cogenerator and the cotilting correspondence.}

As an application of faithfully balanced modules we give a simultaneous generalization of the cogenerator and the cotilting correspondence. We look at modules $M$ which are of the form $M= C\oplus X$ with $C$ a cotilting module and $X \in {}^{{}_{0<}\perp} C$.
If $C=\kdual \La$, then $M$ is an arbitrary cogenerator. If $X=0$, then $M$ is a cotilting module. By Lemma \ref{dualSummand} we know that $M$ is faithfully balanced and if $\id C \leq k$, then $\cogen^{t} (M) = \cogen^{t} (C)$ for all $t \geq k-1$. So, what is the corresponding pair  to a pair $[\La , M]$ as just described?

We will need the following definition. 
\begin{dfn} Let $\Ga$ be a finite-dimensional algebra and $N,J$ a left $\Ga$-modules with $J$ injective. We say that $N$ is a $J$-restricted $k$-cotilting module if the following holds
\begin{itemize}
\item[(i)] there is an exact sequence 
$0 \to N \to J_0 \to \cdots \to J_k \to 0$ with $J_i \in \add J$, $ 0 \leq i \leq k$, 
\item[(ii)] $N$ is self-orthogonal,
\item[(iii)] there is an exact sequence 
$0 \to N_k \to \cdots \to N_0 \to J \to 0$ with $N_i \in \add N$, $0 \leq i \leq k$.
\end{itemize}
\end{dfn}
The following is straightforward to see. 
\begin{lem} \label{res-cotilt} If $N$ is in 
$\cogen^1(J)$ for an injective $\Ga$-module $J$, the following are equivalent
\begin{itemize}
\item[(1)]
$N$ is $J$-restricted $k$-cotilting, 
\item[(2)] $\kdual (N,J)$ is a  left $k$-cotilting $B:=\End (J)^{op}$-module and $\Ext^i_B(\kdual J, \kdual (N,J))=0$ for $1 \leq i \leq k$. 
\end{itemize}
\end{lem}

\begin{proof} 
Let $B= \End_{\Ga} (J)^{op}$. 
The injective module ${}_{\Ga}J$ induces a restriction functor $\kdual (-,J) \colon \Gamma\!-\!\modu \to B \!-\!\modu $ which has a fully faithful right adjoint $r=({}_B\kdual J,-)$. By Lemma \ref{duality} we get an equivalence of categories 
$\kdual (-,J) \colon \cogen^{k+1}(J) \longleftrightarrow 
\bigcap_{i=1}^k \Ker \Ext_B^i(\kdual J, -)\colon r$  
for every $k \geq 1$ using that ${}_B\kdual J$ is a generator. 
\\
Assume (1), then it is easy to check that $\kdual (N,J)$ is a $k$-cotilting $B$-module since $\kdual (-,J)$ is exact and apply Lemma \ref{fullfaith} (2) to prove the self-orthogonality. Since $N \in \cogen^{k+1} (J)$, we can use the equivalence just mentioned to see that (2) is fulfilled. \\
Assume (2), since $\Ext^i_B(\kdual J, \kdual (N,J))=0$ for $1 \leq i \leq k$ we have that $r$ is exact on an injective coresolution of $S:=\kdual (N,J)$ and on the exact sequence $0 \to S_k \to \cdots \to S_0 \to \kdual B \to 0$ with $S_i \in \add S$. Since $N= (\kdual J, S)$, we can use again Lemma \ref{fullfaith} (2) to see that $N$ is self-orthogonal.   
\end{proof}

\begin{lem} \label{cotilt-plusPerp}
The assignment $[\La ,M] \mapsto [\Ga = \End (M), N={}_{\Ga}M]$ restricts to a bijection between 
\begin{itemize}
\item[(1)] $[\La , M]$ such that $\cogen^{k-1}(M)$ is the perpendicular category of a $k$-cotilting module $C$ which is a summand of $M$.
\item[(2)] $[\Ga ,N]$ such that $N$ is faithfully balanced $J$-restricted $k$-cotilting module for some injective module $J$. 
\end{itemize}
Furthermore, if $[\La , M=C\oplus X] $ is mapped to $[\Ga, N]$ as explained before, then we have $\lvert M \rvert =\lvert \La \rvert +\lvert X\rvert $ and $\lvert N\rvert =\lvert \Ga\rvert -\lvert X\rvert$ and $\End (C) \cong \End_{\Ga}(J)^{op}$, under this isomorphism ${}_{\End (C)} C \cong \kdual (N,J)$ and $\End_{\End(J)}(J)\cong \End_{\End(C)}( (M,C) ) \cong \End_{\La} (M)=\Ga$, therefore ${}_{\Ga}J$ is also faithfully balanced. 
\end{lem}

We remark that in the previous lemma in (1) the tilting module $C$ does not have to be mentioned since it can be reobtained as the Ext-injectives in $\cogen^{k-1}(M)$. Recall that a module $I\in \cogen^{k-1}(M)$ is Ext-injective if $\Ext^1(N,I)=0$ for all $N\in \cogen^{k-1}(M)$. Similar, in (2) a restricted cotilting module $N$ is restricted to a unique injective module which is obtained as the direct sum of the injectives appearing in an injective coresolution of $N$.

\begin{proof}
$(1)\mapsto (2):$ Let $[\La , M]$ be as in (1). Since $\kdual \La \in \cogen^t(M)$ for all $t\geq 0$ (cf. remark at the beginning of this section), we conclude that ${}_{\Ga}M=N$ is self-orthogonal by Lemma \ref{fullfaith} (2). Since $C$ is a $k$-cotilting module, we have two exact sequences of $\La$-modules
\[
\begin{aligned}
(1) & \quad 
0 \to C_k \to C_{k-1}\to \cdots \to C_0 \to \kdual \La \to 0  \\
(2) &\quad 
0  \to C \to I_{0}\to \cdots \to I_{k-1} \to I_k\to 0  
\end{aligned}
\]
with $C_i \in \add(C), I_i \in \add\kdual \La$, $0\leq i\leq k$.  
Since $M \in {}^{{}_{0<}\perp }C$ we have $\Ext^i_{\La} (M,C)$ for all $i\geq 1$, this implies that $\kdual (M,-)$ is exact on both of the sequences, so if we denote $J=\kdual (M,C)\in \add \kdual \Ga $, then we obtain two exact sequences with $N={}_{\Ga}M=\kdual (M, \kdual \La)$ 
\[
\begin{aligned}
(1') & \quad 
0 \to N \to J_0 \to \cdots \to J_{k-1} \to J_k \to 0  \\
(2') &\quad 
0  \to N_k \to N_{k-1}\to \cdots \to N_0 \to J\to 0  
\end{aligned}
\]
with $J_i \in \add J, N_i\in \add N$, $0\leq i \leq k$. Sequence $(1')$ implies $\id N\leq k$. Since $N$ is self-orthogonal, we see that the functor $\kdual (N,-)$ is exact on sequence $(1')$ and $(2')$. This implies using sequence $(2')$ that $J\in \gen_k(N)$ and $\Omega_N^{k+1} J = 0$. \\
$(2) \mapsto (1)$ Let $\La = \End_{\Ga} (N)$ and $M= {}_{\La} N$, $C= {}_{\La} \kdual (N,J) \in \add(M)$. 
If $N$ is a $J$-restricted $k$-cotilting, we get an isomorphism $0=\Ext_{\Ga}^i(J, \kdual \Ga ) \to  \Ext_{\La}^i(M,C)$ for every $i \geq 1$ by using Lemma \ref{fullfaith}. In particular, $M \in \bigcap_{i\geq 1} \Ker \Ext^i_{\La} (-,C)$ and $C$ self-orthogonal. 
It is straight-forward to see that the functor $\kdual (N,-)$ is exact on the two exact sequences in the definition of the $J$-restricted $k$-cotilting module and that these yield the two exact sequences to see that $C$ is a $k$-cotilting module. 

\end{proof}
We can refine the previous two lemmas, for that we will use the following four assignments for triples of finite-dimensional algebras together with two modules: 
\begin{itemize}
\item
the Auslander-Solberg assignment (AS), 
\item
the dual Auslander-Solberg assignment (dual AS), 
\item 
swap $s ([\La , M, G] ):= [\La , G, M] $, and 
\item 
passing to the opposite algebra $(-)^{\rm{op}}([\La , M, G] ):= [\La^{\rm op} , \kdual M, \kdual G] $. 
\end{itemize}
We remark that if all involved modules are faithfully balanced then each of the assignments is self-inverse. 

\begin{thm}
We consider the following triples 
\begin{itemize}
\item[(1)] $[\La , M , C]$ such that $C\in \add(M)$, 
$\cogen^{k-1}(M)=\cogen^{k-1}(C) = {}^{{}_{0<}\perp} C$,   
\item[(2)] $[\Ga ,N , J]$ such that $J\in \add \kdual \Ga$, 
$N$ is faithfully balanced $J$-restricted $k$-cotilting module,
\item[(3)] $[B, G, Q]$ such that $Q$ is a $k$-cotilting module, $B\in \add G$ and $G\in \cogen^{k-1}(Q)$.
\end{itemize}
Then the following diagram of bijective assignments is well-defined and commutes 
\[
\xymatrix{
(1)\ar[rr]^{\text{(dual AS)}} \ar[dr]_{s\circ (AS)\circ s}&&(2)\ar[dl]^{(-)^{\rm op}\circ (AS)\circ s} \\
&(3)&
}
\]

\end{thm}
\begin{proof}
The correspondence between (1) and (2) is Lemma \ref{cotilt-plusPerp}, using also its proof to see $J=\kdual (M,C)$ in this case. 
The correspondence between (2) and (3) is Lemma \ref{res-cotilt}. \\
Let $[\La, M, C]$ be as in (1) corresponding to $[\Ga , N, J]$ under the dual AS assignment. 
Now, we observe in Lemma \ref{cotilt-plusPerp} also that we have $\End_{\Ga}(J)^{\rm op} \cong \End_{\La} (C)=:B$ and using this isomorphism 
we have 
${}_{B}\kdual J \cong {}_B(C,M)$, ${}_{B} \kdual (N,J) \cong {}_B C$. This implies that the whole diagram is commutative and the correspondence between (1) and (3) is a consequence of this. 
\end{proof}

\begin{exa} Let $\La$ be the path algebra of the quiver $1\to 2 \to 3$ over some field. 
Let $M$ be the $\La$-module $P_2 \oplus P_1 \oplus S_2 \oplus I_2$. The we have $M=C \oplus X$ for $C=P_1 \oplus I_2\oplus S_2$ a $1$-cotilting module and $X=P_2 \in \cogen (C)$. We identify $\Ga $ with the commuting square 
\[ 
\begin{aligned}
\xymatrix@-1pc{ &b\ar[dr]^{\beta}& \\ a\ar[ur]^{\al}\ar[dr]_{\gamma}&& c\\ &d\ar[ur]_{\delta}& }  &\quad \xymatrix{ \\
,\beta \alpha -\delta \gamma = 0 \\ }
\end{aligned}
\]
via $a=[P_2], b=[P_1], c=[I_2], d=[S_2]$. Then ${}_{\Ga} M= I_b \oplus I_c \oplus P_b$ is faithfully balanced, self-orthogonal and it has injective dimension $1$, the injective coresolution of $P_b$ is given by 
$0\to P_b \to I_c \to I_d \to 0$ and we observe that $(M,-)$ is exact on it. We consider the injective module $J= I_b \oplus I_c \oplus I_d$, the previous exact sequence  shows $J \in \gen_1(M)$ and $\Omega^{2}_MJ=0$.  

\end{exa}

\section{On categories relatively cogenerated by a module}

Let $M \in \Lamod$. We recall from \cite{ASoI} that one can associate two additive subbifunctors $\mbF_M, \mbF^M \subseteq \Ext^1(-,-)$ to the subcategory $\add(M)$ defined for $(C,A) \in (\Lamod)^{op} \times \Lamod$ as follows 
\[ 
\begin{aligned}
\mbF^M (C,A)&= \{0 \to A \to B \to C\to 0 \mid \Hom_{\La}(-,M) \text{ is exact on it}\}\\
\mbF_M (C,A)&=\{0 \to A \to B \to C\to 0 \mid \Hom_{\La}(M,-) \text{ is exact on it}\}.
\end{aligned}
\]
An exact sequence in $\La$-modules is called $\mbF^M$ exact if and only if $\Hom_{\La}(-,M)$ is exact on it and the category $\mcI (\mbF_M) =\add (M\oplus \kdual \La)$ is called the category of $\mbF_M$-injectives. \\
An exact sequence in $\La$-modules is called $\mbF_M$ exact if and only if $\Hom (M,-)$ is exact on it and the category $\mcP (\mbF_M) =\add (M\oplus \La)$ is called the category of $\mbF_M$-projectives.

If $\mbF \subseteq \Ext^1(-,-)$ is a subbifunctor, we will say a monomorphism $f:X \to Y$ is an $\mbF$-monomorphism if the short exact sequence $0\to X \xrightarrow{f} Y \to \coKer f \to 0 $ is $\mbF$-exact, dually we define $\mbF$-epimorphism. We say a left exact sequence of morphisms is $\mbF$-exact if all inclusions of images are $\mbF$-monomorphisms, dually we define a right exact map to be $\mbF$-exact if all epimorphisms on cokernels are $\mbF$-epimorphisms. Compositions of $\mbF$-monomorphisms (resp. $\mbF$-epimorphisms) are again $\mbF$-monomorphisms (resp. $\mbF$-epimorphisms).

In the two new exact structures, we have
\begin{itemize}
\item[(1)] $\cogen^k(M)$ is the category of modules $N$ such that
there exists an $\mbF^M$-exact sequence  
\[
0\to N\to M_0 \to \cdots \to M_k
\]
with $M_i\in \add (M)$. Since $M$ is $\mbF^M$-injective, this sequence can be seen as the beginning of an $\mbF^M$-injective coresolution.
\item[(2)] $\gen_k(M)$ is the category of modules $N$ such that
there exists an $\mbF_M$-exact sequence  
\[
M_k \to \cdots \to M_0 \to N\to 0
\]
with $M_i\in \add (M)$. Since $M$ is $\mbF_M$-projective, this sequence can be seen as the beginning of an $\mbF_M$-projective resolution.
\end{itemize}

By \cite[Proposition 1.7]{ASoI}, we have that $\mbF^M, \mbF_M$ are both additive subbifunctors of $\Ext^1_{\La}(-,-)$ with enough projectives and enough injectives. Therefore, one can define for $\mbF\in \{\mbF^M,\mbF_M\}$ the derived right functors $\Ext^i_{\mbF}(-,-)$, $i\geq 1$, these are defined by using $\mbF$-injective coresolutions or $\mbF$-projective resolutions. 

There exist additive subbifunctors of $\Ext^1_{\La}$ which are not of the form $\mbF^M$ or $\mbF_M$, see \cite{Bu-Closed} or \cite{DRSS}. However, according to \cite{ASoII}, the existence of $\mbF$-cotilting modules is equivalent to $\mbF$ is of the form $\mbF=\mbF_G=\mbF^H$ for a generator $G$ and a cogenerator $H$, and in this case $H=\tau G \oplus \kdual\La$ and $G=\tau^-H \oplus \La$. Such a functor is called an additive subbifunctor (of $\Ext^1_{\La}$) of \emph{finite type}. As one of our main results, we will prove (in section 8) the relative (co)tilting correspondence. So, in this paper, we will only consider the additive subbifunctors of finite type. Note that, by definition, for any module ${_{\La}}M$ we have $\mbF_{M}=\mbF_{M\oplus \La}$ and $\mbF^{M}=\mbF^{M\oplus \kdual\La}$.

We define two new full subcategories of $\Lamod$
\[ 
\begin{aligned}
\cogen^k_{\mbF}(M) &:= \bigg\{ N \;\bigg\vert  \;\begin{aligned} \exists\; \mbF\text{-exact seq.} \; 0\to & N \to M_0 \to \cdots \to M_k \ \text{with}\  M_i \in \add (M),\ \text{and s.t.} \ \\  \Hom(M_k,M)\to & \cdots \to \Hom (M_0,M) \to \Hom(N,M) \to 0 \;\;\text{ is exact}\end{aligned} \bigg\} \\ 
\gen_k^{\mbF} (M) &:= \bigg\{ N \;\bigg\vert  \;\begin{aligned} \exists \; \mbF\text{-exact seq} \; M_k & \to \cdots \to M_0 \to N \to 0 \text{with}\ M_i \in \add (M),\ \text{and s.t.}\\  \Hom(M, M_k)\to & \cdots \to \Hom (M, M_0) \to \Hom(M, N) \to 0 \;\;\text{ is exact}\end{aligned} \bigg\}.
\end{aligned}
\]
Similarly, we can define $\copres^k_{\mbF}(M)$ and $\pres_k^{\mbF}(M)$. Then, we have $\cogen^k(M) =\copres^k_{\mbF^M} (M)= \cogen^k_{\mbF^M} (M)$ and $\gen_k(M)= \pres_k^{\mbF_M}(M)= \gen_k^{\mbF_M}(M)$. 

\begin{exa} \label{exa-cogen-F} Let $\mbF=\mbF_G=\mbF^H$ for a generator $G$ and a cogenerator $H$ and $M$ be a module with $\Ext^i_{\La}(G,M)=0$ (resp. $\Ext^i_{\La}(M,H)=0$), $1 \leq i \leq k+1$ for some $k\geq 0$. Then one has 
\[ 
\cogen^k_{\mbF} (M)=\cogen^k (M) \cap \bigcap_{i=1}^{k+1} \Ker \Ext^i_{\La} (G,-) \quad (\text{resp. } \gen_k^{\mbF}(M)=\gen_k(M)\cap \bigcap_{i=1}^{k+1} \Ker \Ext^i_{\La} (-,H) \;).
\]
\end{exa}

\begin{lem} Let $\mbF=\mbF_G=\mbF^H$ for a generator $G$ and a cogenerator $H$. 
A module $Z\in \cogen^k(M)$ is in $\cogen_{\mbF}^k(M)$ if and only if the short exact sequences  \[0 \to \Omega_M^{-i} Z \xrightarrow{f_i} M_i \to \Omega_M^{-(i+1)} Z \to 0\]
with $f_i$ minimal left $\add (M)$-approximation are $\mbF$-exact for $0 \leq i \leq k$.
\end{lem}

\begin{proof} 
It is enough to observe the following: If $f\colon X \to Y$ is an $\mbF$-monomorphism and $\mbF=\mbF^H$, then this is equivalent to $(f,H)$ being surjective. So if an $\mbF$-monomorphism $f$ factors as $f= \al \beta$, then 
$\beta $ also has to be an $\mbF$-monomorphism. 
\end{proof} 

\begin{exa} 
Let $M$ be any module and $k \geq 0$, then $\cogen^k(M)= \cogen^k_{\mbF^M} (M)= \copres^k_{\mbF^M}(M)$ is closed under summands and 
is $\mbF^M$-extension closed since $M$ is $\mbF^M$-injective.

For $k \geq 1$, it is closed under kernels of $\mbF^M$-epimorphisms $X \to Y$ with $X,Y \in \cogen^k(M)$. 
For $k= \infty $ it is also closed under cokernels of $\mbF^M$-monomorphisms $X \to Y$ with $X,Y \in \cogen^{\infty} (M)$. So, one can define the derived category $\mathsf{D}^b_{\mbF^M}(\cogen^k(M))$, see \cite{Nee,KelDer}. 
It is completely unknown which informations these encode. 
\end{exa}

\subsection{The relative version of faithfully balancedness}

Recall that for a finite-dimensional algebra $\La$ a module ${_{\La}}M$ is faithful if and only if $\La\in\cogen(M)=\cogen^0(M)$, and it is faithfully balanced if and only if $\La\in\cogen^1(M)$. So it makes sense to call a faithful module 0-faithful and call a faithfully balanced module 1-faithful. Of course one can define the notion of $k$-faithful module for any non-negative integer $k$. Since in the relative setting balancedness doesn't make sense, we introduce the following definition.
 
 \begin{dfn} Let $\mbF\subseteq \Ext^1_{\La}(-,-)$ be an additive subbifunctor of finite type and $k$ a non-negative integer. We say a module $M$ is $k$-$\mbF$-faithful if $\mcP(\mbF) \subseteq \cogen^k_{\mbF}(M)$. In particular, a $1$-$\mbF_{\La}$-faithful module is just a faithfully balanced module.
 \end{dfn}
 
 Easy examples of $1$-$\mbF$-faithful modules are $\mbF$-(co)tilting modules (see section 8) and modules which have $G$ or $H$ as a summand. Here is an other easy example. 

\begin{exa}\begin{itemize}
\item[(1)] Let $\La$ be a finite-dimensional algebra, $P_1, \ldots, P_n$ its indecomposable projectives and assume that there is a subset $I \subseteq \{1, \ldots , n\} $ such that 
$M:= \bigoplus_{i\in I} \bigoplus_{j\geq 0} \tau^{-j}P_i$ is  finite-dimensional and faithfully balanced. Then $G=M \oplus \La$ and  $H= M \oplus \kdual \La$ fulfill $\mbF_G =\mbF^H=:\mbF$. Clearly, we have $G \in \cogen^1_{\mbF}(M)$, so $M$ is $1$-$\mbF$-faithful. 
\item[(2)]
  Let $\La$ be a basic Nakayama algebra and assume $M= \bigoplus_{X \colon \text{ indec, not simple} }X$ is faithfully balanced \footnote{this is the case if $\La$ has no simple 
projective-injective and $\tau^- S$ is not simple injective for every $S$ simple projective - for example $\mathbb{A}_n$ fulfills this for $n \geq 3$.}. 
Let $G=M \oplus \bigoplus_{P_i \colon \text{ simple proj}} P_i$ and $H = M \oplus  \bigoplus_{I_i \colon \text{ simple inj}} I_i$. Then we claim: 
\[ \{1\text{-}\mbF \text{-faithful modules}\} = \{ M' \oplus S \mid S \text{ semi-simple}, \add (M')= \add (M) \} \]
Since $M$ is $\mbF$-projective-injective, $M$ has to be summand of every $1$-$\mbF$-faithful module. On the other hand, let $S$ be a semi-simple module, we want to see that $M\oplus S$ is $1$-$\mbF$-faithful. Assume that there is a simple projective $P\notin \add (S)$, since $(P,S)=0=(S,P)$ we have that the minimal left $\add (M\oplus S)$ equals the minimal left $\add (M)$ and the minimal left $\add (H)$-approximation, in particular $G \in \cogen_{\mbF} (M \oplus S)$. Now, we look at the cokernel of the approximation $X= \Omega^-_{M\oplus S} P$, since $M$ is faithfully balanced we have $X \in \cogen (M)$, in particular $X$ has no simple injective summand. 
So, every simple summand $S'\notin \add (S)$ of $X$ has a minimal left $\add(H)$-approximation which coincides with a minimal left $\add(M)$- and  $\add (M \oplus S)$-approximation which is an $\mbF$-monomorphism and therefore, we conclude that $G \in \cogen^1_{\mbF} (M \oplus S)$.   
\end{itemize}
\end{exa}
 
The main result of this subsection is the following 
 
\begin{thm} \label{Rel-fb} Let $\mbF\subseteq \Ext^1(-,-)$ be an addtive subbifunctor of the form $\mbF=\mbF_G=\mbF^H$ 
for a generator $G$ and a cogenerator $H$. 
The following are equivalent for every module $M$ and every $k\geq 0$.
\begin{itemize}
\item[(1)] $G \in \cogen^{k}_{\mbF}(M)$.
\item[(2)] $H \in \gen_{k}^{\mbF}(M)$. 
\end{itemize}
\end{thm}

 
Let $M\in \Lamod$ and $\Gamma = \End_{\La} (M)$. We define 
\[\Sigma = \End_{\La} (H) \quad \text{and} \quad \Delta = \End_{\La} (G).
\]
We first remark that  generators and cogenerators are faithfully balanced, in particular this applies to $H$ and $G$ and we have
\[
\begin{aligned}
\Lamod &= \cogen (H) &&= \cogen^1(H) &&= \cogen^2(H) &&= \cdots &&= \cogen^{\infty} (H) \\ 
\Lamod &= \gen (G) &&= \gen_1(G) &&= \gen_2(G) &&= \cdots &&= \gen_{\infty} (G).
\end{aligned}
\]
By Lemma \ref{duality} we have dualities of categories 
\[
\begin{aligned}
(-, {}_{\La}H)\colon \quad \Lamod & \longleftrightarrow \cogen^1({}_{\Sigma}H) &&\colon (-, {}_{\Sigma }H)\\ 
\kdual ({}_{\La}G, - )\colon \quad \Lamod & \longleftrightarrow \gen_1({}_{\Delta}G)&&\colon \kdual ({}_{\Delta} G,-).
\end{aligned}
\]

The key step in the proof is given by the following lemma.

\begin{lem} \label{cogenF} Keep the above notations. For $0 \leq k \leq \infty $ we have
\begin{itemize}
\item[(1)]
The following are equivalent 
\begin{itemize}
\item[(1a)] $N\in \cogen_\mbF^k(M)$.
\item[(1b)] ${}_{\Si}(N, H) \in \gen_k({}_{\Si} (M, H))$.
\item[(1c)] Consider the natural map $
(M, H) \otimes_{\Gamma } (N,M) \to (N, H)$, $f \otimes g \mapsto f\circ g $.  
\begin{itemize}
\item[(i)] For $k=0$: It is an epimorphism . 
\item[(ii)] For $k\geq 1$: It is an isomorphism and $ \Ext^i_{\Gamma}( (N,M), \kdual (M, H) )=0$ for $1\leq i \leq k-1$. 
\end{itemize}
\end{itemize}
\item[(2)]
The following are equivalent
\begin{itemize}
\item[(2a)] $N\in \gen_k^{\mbF}(M)$.
\item[(2b)] $(G,N)_{\Delta} \in \gen_k ((G,M)_{\Delta} )$.
\item[(2c)] Consider the natural map $(M,N) \otimes_{\Gamma }(G,M) \to  (G, N)$, $f\otimes g \mapsto f\circ g$. 
\begin{itemize}
\item[(i)] For $k=0$: It is an epimorphism. 
\item[(ii)] For $k\geq 1$: It is an isomorphism and $\Ext^i_{\Gamma}((G,M), \kdual (M,N))=0$ for $1\leq i \leq k-1$. 
\end{itemize}
\end{itemize}
\end{itemize}
\end{lem}

\begin{proof}
It is easy to see the equivalence of (1a) and (1b) using that the duality $(-,H)$ restricts to a duality of categories 
\[ 
(-,{_{\La}}H) \colon \cogen^k_{\mbF}(M) \longleftrightarrow \cogen^1({}_{\Sigma }H) \cap \gen_k ({}_{\Sigma} (M,H))\colon (-,{_{\Si}}H).
\]
To see that the map from the right to the left is well-defined it is important to observe that ${}_{\Si}H$ is an injective module (since $H$ is a cogenerator), therefore the functor $(-,{}_{\Si}H)$ is exact. 
Similarly, it is easy to see the equivalence of (2a) and (2b) using the second equivalence mentioned above. 
For the equivalence of (1b) and (1c) we translate the statement of (1c) into the characterization from Lemma \ref{cogen-k}. The most important observation is the following $E:= \End_{\Sigma} ((M, H) )= \Gamma^{\op}$. The natural map from Lemma \ref{cogen-k} (for the category $\gen_k{}_{\Si}(M,H)$) is:
\[ 
\begin{aligned}
\Hom_{\Si} ((M,H), (N,H)) \otimes_E (M,H) &\to (N, H) \\
f \otimes g & \mapsto f (g).
\end{aligned}
\]
First observe $E=  \Gamma^{\op}$ means left (resp. right) $E$-modules are naturally right (resp. left) $\Gamma$-modules and $X_E \otimes_E {}_EL \cong L_{\Gamma} \otimes_{\Gamma} {}_{\Gamma}X$.
Secondly, since $H$ is a cogenerator we have 
\[
\Hom_{\Sigma} ((M,H), (N,H) )  = (N,M) 
\]
With this identifications the map from before becomes the natural map mentioned in (1c). \\
The equivalence of (2b) and (2c) is analogue. We set $C= \End_{\Delta} ((G, M) )= \Gamma $.
By lemma \ref{cogen-k} we have to look at the natural map 
\[ 
\begin{aligned}
 \Hom_{\Delta} ( (G, M) ,  (G, N)) \otimes_C (G,M) & \to (G, N) \\
f \otimes g & \mapsto f (g)
\end{aligned}
\]
We have an isomorphism of right $\Gamma $-modules since $G$ is a generator 
\[
\Hom_{\Delta} ( (G, M) ,  (G, N)) =  (M,N) 
\]
With this identifications the map from before becomes the natural map in (2c). 
\end{proof}

We observe that the proof of Theorem \ref{Rel-fb} is a direct consequence of the previous lemma: By setting $N= G$ in part (1) and $N= H$ in part (2), we obtain the same maps in (1c), (2c) and therefore the claim follows.

\begin{lem} \label{relduality} Let $\mbF=\mbF_G=\mbF^H$ for a generator $G$ and a cogenerator $H$.  
Let $M\in \Lamod$, $\Gamma = \End_{\La}(M)$, $L=(G,M)$ and $R= \kdual (M,H)$. \\
If we assume that $\La \in \cogen^1_{\mbF}(M)$ and $H\in \gen_1(M)$ then the duality 
$(-,{_{\La}}M) \colon \cogen^1(M) \leftrightarrow \cogen^1(M) \colon (-,{_{\Ga}}M)$  restricts to a duality 
$\cogen_{\mbF^H}^1({_{\La}}M) \leftrightarrow \cogen^1_{\mbF^R}({_{\Ga}}M)$. 
Furthermore, it restricts to a duality 
\[  
(-,{_{\La}}M) \colon \cogen^k_{\mbF^H}(M) \longleftrightarrow \cogen^1_{\mbF^R} (M) \cap \bigcap_{i=1}^{k-1} \Ker \Ext^i_{\Gamma } (-,R) \colon (-,{_{\Ga}}M)
\]
In particular, $G \in \cogen^k_{\mbF^H} (M)$ is equivalent to $L \in \cogen^1_{\mbF^R}(M)$ and $\Ext^i_{\Gamma} (L,R) =0$ for $1 \leq i \leq k-1$.
\end{lem}

\begin{proof}
Since $H \in \gen_1(M)$ we have that $\kdual (M,R)= \kdual (M,\kdual (M,H)) \to H$ is an isomorphism. 
So, it is enough to proof that 
$(-, {}_{\La}M) $ maps $ \cogen^1_{\mbF^H} (M)$ to $ \cogen^1_{\mbF^R} (M)$ and use $R \in \gen^1({}_{\Gamma}M)$ to get the quasi-inverse by symmetry. 

Let $X \in \cogen^1_{\mbF^H}(M)$. We choose a projective presentation $P_1 \to P_0 \to X \to 0$. By applying $(-, {_\La}M)$ we get an exact sequence of $\Gamma $-modules $0 \to (X,M) \to (P_0 ,M) \to (P_1, M)$ is with $(P_i,M) \in \add(M)$.  
We apply $(-,R)$ to get a complex $((P_1,M), R) \to ((P_0,M), R) \to ((X,M), R) \to 0$. We would like to see that it is exact. By Hom-Tensor adjunction it identifies with the first row in the following commutative diagram
\[
\xymatrix{
\kdual [(M,H) \otimes_{\Gamma} (P_1, M)] \ar[r]& \kdual [(M,H)\otimes_{\Gamma} ( P_0, M)] \ar[r]& \kdual [(M,H) \otimes_{\Gamma} (X,M)] \ar[r]& 0\\
\kdual(P_1,H)\ar[r]\ar[u]&\kdual(P_0,H)\ar[u]\ar[r]& \kdual (X,H)\ar[r]\ar[u]& 0
}
\]
Note that the arrows up are the dual of the natural maps $(M,H) \otimes_{\Ga} (Y,M) \to (Y,H)$ given by $f \otimes g \mapsto 
f\circ g$. By Lemma \ref{cogenF} we know that this natural map is an isomorphism if and only if $Y \in \cogen_{\mbF^H}^1(M)$. By assumption we have $P_1,P_0,X \in \cogen_{\mbF^H}^1(M)$ and the first row identifies with the complex in the second row. But the exactness of the second row follows since $\kdual (-,H)$ is right exact. This proves $(X,M)\in \cogen^1_{\mbF^R} (M)$. For the symmetry, we need to see  $\Gamma \in \cogen_{\mbF^R}^1(M)$. 
But $\Gamma = (M , M) $ and $M \in \cogen^1_{\mbF^H}(M)$ implies the claim by the argument just given. 

The further restriction follows directly from Lemma \ref{cogenF}.

\end{proof}



Of course there is a dual version of the previous lemma which we will leave out.

If ${}_{\La}M$ is $1$-$\mbF^H$-faithful, then ${}_{\Ga} M$ does not have to be $1$-$\mbF^R$-faithful (with $\Ga =\End_{\La}(M)$ and $R=\kdual (M,H)$). We give an example for this:

\begin{exa}\label{F-balance-asym}
Let  $\La$ be the path algebra of $1\xrightarrow{\alpha} 2\xrightarrow{\beta} 3$ modulo the relation $\beta\alpha=0$. Let $G=\La \oplus S_1$, $H=\kdual \La \oplus S_2$ and $\mbF=\mbF^H=\mbF_G$. Then $M:=G$ is clearly $1$-$\mbF^H$-faithful and $\pd_{\mbF} M =0$. Let us look at $\Ga =\End_{\La}(M)$, since we have irreducible morphisms $S_3 \to \begin{smallmatrix} 2 \\ 3 \end{smallmatrix} \to \begin{smallmatrix} 1 \\ 2 \end{smallmatrix}\to S_1$, 
we can identify it with the following bound path algebra $d\to c\to b \to a$ modulo all path of length $2$. We have ${}_{\Ga} M = (P_1, M) \oplus (P_2 , M) \oplus (P_3 , M) = P_b \oplus P_c \oplus P_d $ and ${}_{\Ga} R:= \kdual (M,H) = {}_{\Ga}M \oplus \kdual (M,S_2)$. We apply $\kdual (M,-) $ to an injective coresolution $0 \to S_2 \to I_2 \to I_1$ to obtain a projective presentation 
$P_b=(P_1,M) \to P_c=(P_2, M) \to \kdual (M, S_2) \to 0$. This implies $\kdual (M,S_2) \cong S_c$ and therefore $\tau^- R = \tau^- S_c =S_d$. It is easy to see that $S_d \notin \cogen (_{\Ga} M)$ implying  $\tau^- R \notin \cogen^1_{\mbF^R} (M)$. This shows 
${}_{\Ga}M$ is not $1$-$\mbF^R$-faithful. 
\end{exa}

Thus the property of being $1$-$\mbF$-faithful is not as nicely symmetric as being faithfully balanced. Nevertheless, we can get the symmetry again if we restrict to the following special case. 
\begin{pro} Let $\mbF=\mbF_G=\mbF^H$ for a generator $G$ and a cogenerator $H$.  
Let $M$ be a faithfully balanced $\La$-module, $\Gamma = \End_{\La} (M)$, $L=(G,M)$ and $R= \kdual (M,H)$. \\
If $M\in \add(H)$ $($or equivalently, $\kdual \Ga \in \add (R)$ $)$, then the following are equivalent: 
\begin{itemize}
\item[(1)]
${}_{\La}M$ is $1$-$\mbF^H$-faithful.
\item[(2)]
${}_{\Ga}M$ is $1$-$\mbF^R$-faithful.  
\end{itemize}
Dually, if $M\in \add(G)$, then ${}_{\La}M$ is $1$-$\mbF_G$-faithful if and only if ${}_{\Ga}M$ is $1$-$\mbF_L$-faithful. 
\end{pro}

\begin{proof}
We assume $M \in \add (H)$. 
Assume $G \in \cogen^1_{\mbF^H}(M)$, we have to see $\tau^- R \in \cogen_{\mbF^R}^1(M)$. \\
Since $H \in \gen_1^{\mbF^H}(M)$ implies that we have an $\mbF$-exact sequence 
\[ 0 \to \Omega^{2}_M H \to M_1 \to M_0 \to H \to 0 \]
with $M_i \in \add (M)$. 
Since $M \in \add (H)$, this implies $\Omega_M^{2} H\in \cogen_\mbF^1(M)$. 
We apply $\kdual (M,-)$ to the last three terms of the four term sequence and obtain an injective copresentation of $R$. We apply $(-,M)$ to the first three terms and observe and get an exact sequence 
\[ (M_0,M) \to (M_1,M) \to (\Omega_M^{2}H ,M) = \tau^-R \to 0\]
in particular this proves the claim. 
\end{proof}

\subsection{Strong dualizing sequences}
 
\begin{dfn} Let $0 \to L \to M_0 \to M_1 \to \cdots \to M_k \to R \to 0 $ be a $k$-$\add(M)$-dualizing sequence in $\Gamod$ for some non-negative integer $k$. We say it is strong if $\kdual (L,-)$ is exact on it.   
\end{dfn}
We can characterize it as follows. 
\begin{lem}\label{strongdualseq}  
A $k$-$\add(M)$-dualizing sequence as in the above definition is strong if and only if one (equivalently all) of the following equivalent statement is fulfilled: 
\begin{itemize}
\item[(1)] $\kdual (L,-)$ is exact on it, i.e., it is an $\mbF_L$-exact sequence $($or equivalently, $R \in \gen^{\mbF_L}_{k} (M) $ $)$.
\item[(2)] $(-,R)$ is exact on it, i.e., it is an $\mbF^R$-exact sequence $($or equivalently, $L \in \cogen_{\mbF^R}^{k} (M)$ $)$.
\item[(3)] Consider the natural map $ (M, R)\otimes_{\La}(L,M) \to (L,R) $, where $\La=\End_{\Ga} (M)$.
\begin{itemize}
\item[(i)] For $k=0$: It is an epimorphism . 
\item[(ii)] For $k\geq 1$: It is an isomorphism and $\Ext^i_{\La} ((L,M),\kdual (M,R)) =0$ for $1 \leq i \leq k-1$. 
\end{itemize}
\end{itemize}
\end{lem}
\begin{proof}
We will prove $(1)$ and $(3)$ are equivalent and the equivalence of $(2)$ and $(3)$ can be proved dually.

We consider the following commutative diagram
\[
\scalebox{0.9}{
\xymatrix @C=0.3cm{
0 \ar[r] & \kdual(L,R) \ar[r]^{i'}\ar[d]^{i} & \kdual(L,M_k) \ar[r]^{f}\ar[d]^{\cong} & \kdual(L,M_{k-1}) \ar[r]\ar[d]^{\cong} & \cdots \ar[r] & \kdual(L,M_0) \ar[d]^{\cong} \\
0 \ar[r] & \kdual((M,R)\otimes(L,M)) \ar[r]^{j}\ar[d]^{\cong} & \kdual((M,M_k)\otimes(L,M)) \ar[r]^{g}\ar[d]^{\cong} & \kdual((M,M_{k-1})\otimes(L,M)) \ar[r]\ar[d]^{\cong} & \cdots \ar[r] & \kdual((M,M_0)\otimes(L,M)) \ar[d]^{\cong}\\
0 \ar[r] & ((L,M), \kdual(M,R)) \ar[r] & ((L,M), \kdual(M,M_k)) \ar[r] & ((L,M), \kdual(M,M_{k-1})) \ar[r] & \cdots \ar[r] & ((L,M), \kdual(M,M_0)).
}
}
\]

Assume (1), then the first row is exact. Since the functor $((L,M),-)$ is left exact, the sequence $0 \to \kdual((M,R)\otimes(L,M)) \to \kdual((M,M_k)\otimes(L,M)) \to \kdual((M,M_{k-1})\otimes(L,M))$ is exact. For $k=0$, we have $ji$ is a monomorphism and so is $i$. This shows the natural map $(M, R)\otimes_{\La}(L,M) \to (L,R) $ is an epimorphism. For $k\geq 1$, we have an induced isomorphism on kernels
\[
\kdual(L,R)=\ker f  \xrightarrow{\cong} \ker g=\kdual((M,R)\otimes(L,M)).
\]
This proves the natural map $ (M, R)\otimes_{\La}(L,M) \to (L,R) $ is an isomorphism. Now the exactness of the first row implies the exactness of the last row which is equivalent to $\Ext^i_{\La} ((L,M),\kdual (M,R)) =0$ for $1 \leq i \leq k-1$. Conversely, assume (3). If $k=0$, then the map $i$ is a monomorphism and so is $i'$. If $k\geq 1$, then the last row is exact and the natural map $ (M, R)\otimes_{\La}(L,M) \to (L,R) $ is an isomorphism will imply the first row is isomorphisc to the last row. So we have, in both cases, that the first row is exact. Since the functor $\kdual (L,-)$ is right exact, (1) follows from the exactness of the first row.
\end{proof}

\begin{rem}
From the proof of the above lemma we see that for any $X$ if $N\in \cogen^1_{\mbF^X}(M)$ then the natural map $(M,X)\otimes(N,M)\to (N,X)$ is an isomorphism. The converse holds true if $X$ is a cogenerator (cf. Lemma \ref{cogenF}). Similarly, we have if $N\in \gen_1^{\mbF_X}(M)$ then the natural map $(M,N)\otimes(X,M)\to (X,N)$ is an isomorphism.
\end{rem}

\begin{lem}\label{strongness} Let $\Ga $ be a finite-dimensional algebra and $0 \to L \to M_0 \to \cdots \to M_k \to R \to 0$ be a $k$-$\add(M)$-dualizing sequence of $\Ga$-modules with $M$ faithfully balanced. 
Define $\La = \End_{\Ga}(M)$, $G= (L,M)$ and $H= \kdual (M,R)$. If $\Gamma \in \cogen_{\mbF^{R}}^1(M)$ and $R \in \gen_1 (M)$ then for every $k \geq 1$ the functor $(-,M)$ restricts to a duality 
\[ 
 \cogen_{\mbF^{R}}^k (M) \longleftrightarrow \cogen_{\mbF^{H}}^1(M) \cap \bigcap_{i=1}^{k-1} \Ker \Ext^i_{\La} (-,H \oplus M).  
\]
In particular, $L \in \cogen_{\mbF^R}^k (M)$ is equivalent to 
$\Ext^i_{\La} (G, H \oplus M)=0$, $1 \leq i \leq k-1$. 
\end{lem}
\begin{proof}
The case $k=1$ follows directly from Lemma \ref{relduality}. For $k>1$ we note that $\mbF^R=\mbF^{R \oplus \kdual \Ga}$ and then apply Lemma \ref{relduality} using the cogenerator $R\oplus \kdual \Ga $ (in place of $H$).
 \end{proof}

\begin{lem} \label{CorRelFB} Let $M$ be a faithfully balanced $\La$-module and $\Ga =\End_{\La} (M)$. 
Let $k \geq 1$. Then, the assignment 
$X,Y \mapsto  (X,M) , \kdual (M,Y) $ gives a self-inverse bijection $($up to seeing $X,Y$ as $\La$ or as $\Ga$-modules$)$ between the following sets of pairs of $\La$-modules and $\Ga$-modules 
\[
\{ 
{}_{\La}G,{}_{\La}H \mid 
\begin{aligned}
G=\tau_k^-H \oplus \La &\in \cogen_{\mbF^H}^1(M) \cap  {}^{{}_{1\sim (k-1)} \perp} (M\oplus H) \\
H=\tau_k G \oplus \kdual\La &\in \gen^{\mbF_G}_1(M) \cap (M\oplus G)^{\perp_{1\sim (k-1)}}
\end{aligned}
\}
\]

and
\[
\{ 
{}_{\Ga} L,{}_{\Ga } R \mid \exists\ \text{a strong }k\text{-}\;{}_{\Ga}M\text{-dualizing sequence from }L\text{ to }R\}.
\]

\end{lem}
\begin{proof}
This follows from Lemma \ref{easy}, Lemma \ref{strongdualseq} and Lemma \ref{strongness}.
\end{proof}

\begin{exa} 
Let $M$ be a faithfully balanced $\La$-module and assume that it has a summand $X\oplus \tau^-X$ with $X$ not injective. We define $G=\La \oplus \tau^-X$, $H= \kdual \La \oplus X$ and $\mbF= \mbF_G =\mbF^H$. 
Then, by definition we have $G \in \cogen^1(M) = \cogen^1_{\mbF}(M)$ and $H \in \gen_1 (M) = \gen^\mbF_1(M)$. 
Therefore, we obtain for $\Ga =\End_{\La}(M)$ a strong $\add({}_{\Ga}M)$-dualizing sequence with a projective-plus-$M$ left end and an injective-plus-$M$ right end. \\
\end{exa}

Now, we can formulate a relative version of the generator/ cogenerator and Morita-Tachikawa correspondence. 
\begin{cor}\label{relgencogencorres}
\begin{itemize}
\item[(1)] $($relative generator correspondence$)$ \\
The Auslander-Solberg assignment $[\La, M, G] \mapsto [\End(M), M, (G,M)]$ is an involution on the set of triples 
$[\La , M , G]$ with $\La \oplus M \in \add (G)$ and $M$ is $1$-$\mbF_G$-faithful.  
\item[(2)] $($relative cogenerator correspondence$)$ \\
The dual Auslander-Solberg assignment $[\La, M, H] \mapsto [\End(M), M, \kdual (M,H)]$ is an involution on the set of triples 
$[\La , M , H]$ with $\kdual \La \oplus M \in \add (H)$ and $M$ is $1$-$\mbF^H$-faithful. 
\item[(3)] $($relative Morita-Tachikawa correspondence$)$ \\
The assignment $[\La , M, G,H]\mapsto [\End (M), M, L=(G,M), R=\kdual (M,H)]$ is a bijection between
\begin{itemize}
    \item[*] $[\La ,M, G,H]$ with $\La\in \add (G), \kdual \La \in \add(H)$, $G= \La \oplus \tau^- H$ and $M\in \add(G) \cap \add(H)$ is $1$-$\mbF_G$-faithful,  and 
\item[*] $[\Ga , N, L,R]$ with $L,R$ are the ends of a strong $\add(N)$-dualizing sequence with $\Ga \in \add (L)$ and $\kdual \Ga \in \add(R)$. 
\end{itemize}
\end{itemize}
\end{cor}

\subsection{$\mbF$-dualizing summands}

Of course, we can also consider relative dualizing summands. 
\begin{dfn} 
Let $\mbF=\mbF_G=\mbF^H$, $M, L\in \Lamod$ and assume $M$ is a summand of $L$. We say $M$ is an $\mbF$-dualizing summand of $L$ if $L \in \cogen^1_{\mbF}(M)$. For $k \geq 0$, we say it is a $k$-$\mbF$-dualizing summand if $L \in \cogen^k_{\mbF}(M)$.  
\end{dfn}

Relative dualizing summands have the properties which we expect from them: 

\begin{lem} \label{reldualizing} 
Let $\mbF=\mbF_G=\mbF^H$ and 
$M, N$ be $\La$-modules and $L=M\oplus N$, $k\geq 1$.
If $N \in \cogen^k_{\mbF}(M)$ $($i.e., $M$ is $k$-$\mbF$-dualizing summand of $L$ $)$, then $M$ is $1$-$\mbF$-faithful if and only if $L$ is $1$-$\mbF$-faithful. \\
If  $H \in \gen_1(M)$, then  $\cogen^k_{\mbF}(M)=\cogen^k_{\mbF}(L)$.  
Furthermore, in this case if also  $\copres^k_{\mbF}(L)=\cogen^k_{\mbF}(L)$  then we  have $\copres^k_{\mbF}(M)=\cogen^k_{\mbF}(M)$. \\
In particular, if $M$ is $1$-$\mbF$-faithful, then $M \oplus P\oplus I$ is $1$-$\mbF$-faithful for every $\mbF$-projective module $P$ and $\mbF$-injective module~$I$.
\end{lem}

\begin{proof} Let $\Si = \End_{\La} (H)$.
We consider the duality for $M$ from Lemma \ref{cogenF}: 
\[ 
(-,H) \colon \cogen^k_{\mbF}(M) \longleftrightarrow \cogen^1({}_{\Sigma }H) \cap \gen_k ({}_{\Sigma} (M,H))\colon (-,H)
\]
and also for $L$ we have 
\[ 
(-,H) \colon \cogen^k_{\mbF}(L) \longleftrightarrow \cogen^1({}_{\Sigma }H) \cap \gen_k ({}_{\Sigma} (L,H))\colon (-,H).
\]
Since $(M,H)$ is a summand of $(L,H)$ and $(L,H) \in \gen_k(M,H)$ follows that 
$\gen_1(L,H) \subseteq \gen_1(M,H)$ (dual argument to $1$-dualizing summand situation). \\
Furthermore, we claim:
if $H \in \gen_1(M)$, then ${}_{\Si}(M,H)$ is faithfully balanced (and therefore, the claim follows from the dual of Lemma \ref{dualSummand} and using the duality from above again). 
So, assume there is an exact sequence $M_1 \to M_0 \to H \to 0$ with $M_i \in \add(M)$ and $(M,-)$ exact on it. Apply $(-,H)$ to it, to obtain an exact sequence $0\to \Si \to (M_0,H) \to (M_1,H)$. Apply $(-,(M,H))$ to it and using $((X,H), (Y,H))= (Y,X)$ for all $\La$-modules $X,Y$ you can identify the result with the complex 
$(M,M_1) \to (M,M_0) \to (M,H) \to 0$ which we know is exact since $H \in \gen_1(M)$. This proves $\Si \in \cogen^1((M,H))$ and therefore the claim.  The remaining claims are proven as in Lemma \ref{dualSummand}. 
\end{proof}

\begin{exa} Let $G$ be a generator and $\mbF=\mbF_G$. Then a $1$-$\mbF$-faithful summand of $G$ is the same as an $\mbF$-dualizing summand of $G$. These are easily determined as follows, let $H=\kdual \La \oplus \tau G$ and $P_1 \to P_0 \to H \to 0$ a minimal $\mbF$-presentation with $P_i \in \add(G)$. Then, the $1$-$\mbF$-faithful summands of $G$ are the summands $P$ of $G$ with $P_1\oplus P_0 \in \add(P)$. Of course, with a dual statement one can find the $1$-$\mbF$-faithful (i.e., the $\mbF$-codualizing) summands of $H$.  
\end{exa}

\section{Relative Auslander-Solberg and Auslander correspondence}
We generalize the notion of dominant dimension to the relative setting.
\begin{dfn}
Let $\Ga$ be a finite-dimensional algebra and  $\mbF=\mbF_G=\mbF^H$ for a generator ${_{\Ga}}G$ and a cogenerator ${_{\Ga}}H$. Consider the minimal $\mbF$-coresolution of $G$ by $\mbF$-injectives
\[
0 \to G \to H_0 \to H_1 \to  H_2 \to \cdots .
\] 
We define $\domdim_{\mbF} \Ga=k$ if there exists an integer $k$ such that $H_{i}\in \add(G)$ for $0\leq i \leq k-1$ and $H_k \notin \add(G)$. If $H_i\in \add(G)$ for all $i\geq 0$  then we define $\domdim_{\mbF} \Ga=\infty$.   
\end{dfn}

\begin{rem}
As is in the classical case, our definition of $\mbF$-dominant dimension is left-right symmetric in the following sense: A functor $\mbF=\mbF_G=\mbF^H$ determines a functor $\mbF_{\kdual H}=\mbF^{\kdual G}=:\mbF^*$ in the category $\Ga^{op}$-mod and vice versa, and $\domdim_{\mbF} \Ga=k$ if and only if $\domdim_{\mbF^*} \Ga^{op}=k$.
\end{rem}

\subsection{Relative Auslander-Solberg correspondence}

\begin{lem}\label{relAus-Goralg}
Let $\mbF=\mbF_G=\mbF^H$ with $G$ and $H$ basic and assume ${}_{\Ga}M$ is a module such that $\add (M) = \add (G) \cap \add (H)$.
Then the following are equivalent for every $k \geq 1$. 
\begin{itemize}
\item[(1)] There is an $\mbF$-exact sequence $0 \to G \to M_0 \to M_1 \to \cdots \to M_k\to H \to 0$ with $M_i \in \add (M)$. 
\item[(2)] $\domdim_{\mbF} \Ga \geq k+1 \geq \id_{\mbF} G.$
\item[(3)] $\domdim_{\mbF} \Ga \geq k+1 \geq \pd_{\mbF} H.$
\end{itemize}
\end{lem}

\begin{proof}
(1) $\Rightarrow$ (2) and (1) $\Rightarrow$ (3) are obvious. We prove (2) $\Rightarrow$ (1) and (3) $\Rightarrow$ (1) is dual.

Assume (2) then we have an $\mbF$-exact sequence 
\[
0 \to G \to M_0 \to M_1 \to \cdots \to M_{k-1} \to M_k' \to H' \to 0
\]
with $M_i\in \add(M)$ for $0\leq i \leq k-1$, $M_k'\in \add(M)$ and $H'\in \add(H)$. We may assume this $\mbF$-exact sequence is a successive composition of minimal left $\add(M)$-approximations of the cokernels. By the dual version of Lemma \ref{HigherTilt} we have $M\oplus H'$ is an $\mbF$-tilting module with $\id_{\mbF}(M\oplus H')=0$. By Lemma \ref{mutation} (1) we know that $M\oplus H'$ is basic and hence $M\oplus H'=H$. Now the desired $\mbF$-exact sequence in (1) can be obtained by adding $M\xrightarrow{1} M$ to $M_k' \to H'$.
\end{proof}

\begin{thm}\label{relAScorres}
Let ${}_{\La}M$ be a faithfully balanced module and $\Ga = \End_{\La} (M)$. The assignment $X,Y\mapsto (X,M), \kdual (M,Y)$ gives a self-inverse bijection between the following sets of pairs of modules 
\begin{itemize}
\item[(1)] \{${}_{\La}L,{}_{\La}R\mid {}_{\La}M \oplus \La \in \add (L), {}_{\La}M \oplus \kdual \La \in \add (R), L= \tau_k^- R \oplus \La, R = \tau_k L \oplus \kdual \La  $, $ \Ext_{\La}^i (L,R) =0$, $1 \leq i \leq k-1$ such that there exists a strong $\add({}_{\La}M)$-dualizing sequence with left end $L$ and right end $R$\}. 
\item[(2)] \{${}_{\Ga}G,{}_{\Ga}H\mid M \oplus \Ga \in \add (G), M \oplus \kdual \Ga \in \add (H), G= \tau^- H \oplus \Ga, H = \tau G \oplus \kdual \Ga  $ such that there exists a strong  $k$-$\add({}_{\Ga}M)$-dualizing sequence with left end $G$ and right end $H$\}. 
\end{itemize}
\end{thm}

\begin{proof}
Combine Lemma \ref{CorRelFB} and Lemma \ref{relAus-Goralg}.
\end{proof}

\subsection{Relative Auslander correspondence}
\begin{lem} \label{finGldim}
Let $k \geq 1$ and assume $\domdim_{\mbF} \Ga \geq k+1$. Let  
${}_{\Ga}M$ be a module with $\add (M) = \add (G) \cap \add (H)$. Then we have $\cogen_\mbF^{k}(M) = \Omega_\mbF^{k+1} (\Gamod)$ and $\gen_k^{\mbF}(M)= \Omega_\mbF^{-(k+1)}(\Gamod)$. Furthermore, the following are equivalent: 
\begin{itemize}
\item[(1)] $\cogen^{k}_{\mbF}(M) = \add (G)$.
\item[(2)] $\gen_{k}^{\mbF}(M) = \add (H)$.
\item[(3)] $\gldim_{\mbF} \Ga \leq k+1$.
\end{itemize}
\end{lem}

\begin{proof}
Since $\domdim_{\mbF} \Ga \geq k+1$ and $\add (M) = \add (H) \cap \add (G)$, we have clearly  $\add (G) \subseteq \cogen_\mbF^{k}(M) \subseteq  \Omega_\mbF^{k+1} (\Gamod)$. On the other hand, we prove in Lemma \ref{HigherTilt} that in this case: 
\[\cogen^k_{\mbF}(M) = \bigcap_{i\geq 1} \Ker \Ext^i_{\mbF} (-,C)\] for 
$C=M \oplus \Omega_M^{k+1}H$ and $\id_{\mbF} C\leq k+1$. So given $X \in \Omega_\mbF^{k+1} (\Gamod)$, there is an $Y \in \Gamod$ such that $X= \Omega_\mbF^{k+1} Y$ and then by dimension shift for $i \geq 1$ 
\[ \Ext^i_{\mbF}(X,C) = \Ext^{i+k+1}_{\mbF}(Y, C) = 0\]
since $\id_{\mbF}C \leq k+1$. In particular, $X \in \cogen^ k_{\mbF}(M)$. One can prove $\gen_k^{\mbF}(M) = \Omega_\mbF^{-(k+1)}(\Gamod )$ with the dual argument. 

Now clearly, $\gldim_{\mbF} \Ga \leq k+1$ is equivalent to $ \Omega^{k+1}_{\mbF}(\Gamod)\subseteq \add (G)$ and by the just proved result, we conclude it is equivalent to (1). The equivalence of (3) and (2) can be proven with the analogous argument. 
\end{proof}

\begin{dfn}
Let $M\in \Lamod$ and assume that there is a strong $\add(M)$-dualizing sequence with left end $L$ and right end $R$.

We say that $M$ is a \kLRcluster 
\ if
\begin{itemize}
   
\item[(i)] $\La \in \cogen_{\mbF^R}^1(M)$ and $\kdual \La \in \gen_1^{\mbF_L} (M)$,
\item[(ii)] $\cogen_{\mbF^R}^1(M) \cap \bigcap_{i=1}^{k-1} \Ker \Ext^i_{\La} (-,R) = \add (L)$ and $\gen^{\mbF_L}_1(M) \cap \bigcap_{i=1}^{k-1} \Ker \Ext^i_{\La} (L,-) = \add (R)$. \end{itemize}
Let $\Ga$ be a finite-dimensional algebra and $\mbF=\mbF_G$ for a generator $G$. Then we say $\Ga$ is a $k$-$\mbF$-Auslander algebra if $\domdim_{\mbF} \Ga \geq k+1 \geq \gldim_{\mbF} \Ga $. 
\end{dfn}

\begin{thm}\label{relAuscorres} $($relative Auslander correspondence$)$ \\
Let $k \geq 1$. There is a one-to-one correspondence between isomorphism classes of basic \kLRclusters ${}_{\La} M$ $($for some $L,R$ $)$ and finite-dimensional algebras $\Ga $ with an exact structure given by $\mbF=\mbF_G=\mbF^H$ such that $\domdim_{\mbF} \Ga \geq k+1 \geq \gldim_{\mbF} \Ga $. 
The correspondence is induced by the assignment \\
\[
[\La , M, L,R]\mapsto [\Ga= \End_{\La}(M), {}_{\Ga}M, G=(L,M), H= \kdual (M,R)].
\] 
\end{thm}

\begin{proof}
Let $M$ be an \kLRcluster  and $\Ga = \End_{\La} (M)$, $G=(L,M), H=\kdual (M,R)$ and $\mbF=\mbF_G=\mbF^H$. Since $L \in \cogen^1_{\mbF^R} (M)\cap \bigcap_{i=1}^{k-1} \Ker \Ext^i_{\La} (-,R) $, we have $G \in \cogen_\mbF^k(M) $ by Lemma \ref{relduality}. Similarly, from $\La \in \add (L), \kdual \La \in \add (R)$ we conclude that ${}_{\Ga} M \in \add (G) \cap \add (H)$ and therefore $\domdim_{\mbF} \Ga\geq k+1 $. By the same lemma, we also have $\cogen_\mbF^k(M)= \add (G)$ and therefore by Lemma \ref{finGldim} $\gldim_{\mbF} \Ga \leq k+1$. \\
Conversely, by Lemma \ref{finGldim} and Lemma \ref{relduality} we can also conclude the other implication. 
\end{proof}

The easiest example can be found for $k=1$. Here, for a $1$-cluster tilting pair $(L,R)$ with respect to $M$ we have $G=L$ is a generator, $H=R$ is a cogenerator with $\mbF=\mbF_G=\mbF^H$ and the definition shortens to a module $M$ such that  $\cogen^1_{\mbF}(M) = \add (G)$ and $\gen_1 ^{\mbF}(M)= \add (H)$ is fulfilled.

Here are some easy examples of $1$-$\mbF$-Auslander algebras. 
\begin{exa}\label{ex-FAusalg-1}
\begin{itemize}
\item[(1)] Let $\mbF=\mbF_{\La}$ and $M$ be a projective-injective module such that $\cogen^1 (M) = \add (\La )$ and $\gen_1(M)=\add (\kdual \La )$. Then, by Lemma \ref{finGldim} it is easy to see that this is equivalent to $\domdim \La \geq 2 \geq \gldim \La $ and it is well-known that this characterizes $\La$ to be an Auslander algebra. 
\item[(2)] Assume $\mbF=\mbF_G=\mbF^H$ and $G=H$ is a generator-cogenerator, in this case we say $\La$ is $\mbF$-selfinjective. A classification of $\mbF$-selfinjective algebras can be found in \cite[section 5]{ASo-Gor}. For example, if $G$ is an Auslander generator ($=$ $1$-cluster tilting module), this is fulfilled.  
Then, if we choose $M=G=H$, then we have $\cogen^1_{\mbF}(M)=\cogen^1(M)=\add (M)=\gen^1(M)=\gen_1^{\mbF} (M)$
and this gives us another example. 
\item[(3)] Let $\Ga$ be the path algebra of $1 \to 2 \to 3$ and let $M= P_2 \oplus P_1 \oplus I_2$. We define $G:= \Ga \oplus M$ and $H:= \kdual \Ga \oplus M$, then it is easy to see $\mbF_G=\mbF^H =: \mbF$ and $\cogen_\mbF^1(M) = \add (G), \gen_1^{\mbF}(M) = \add (H)$. 
\item[(4)]
Let $\Ga$ be the path algebra of the following quiver: $\xymatrix@C=1.5pc@R=1pc{&1\ar[d]& \\ 3\ar[r]&2\ar[r]&4.}$\\
Let $M:= P_2 \oplus \tau^- P_2\oplus \tau^{-2}P_2$, $G = M \oplus P_1\oplus P_3 \oplus P_4, H=M \oplus I_1 \oplus I_3 \oplus I_4$ and $\mbF:= \mbF_G=\mbF^H$. Then we have $\mbF$-exact sequences 
\[
\begin{aligned}
 0 \to P_4 &\to P_2 \to \tau^-P_2 \to I_4 \to 0 \\
 0\to P_3 &\to \tau^-P_2 \to  \tau^{-2}P_2 \to I_3 \to 0\\
  0\to P_1 &\to \tau^-P_2 \to  \tau^{-2}P_2 \to I_1\to 0
\end{aligned}
\]
which show $\domdim_{\mbF} \Ga = 2$. It also easy to see that $2= {\rm max}_X \{\pdim_{\mbF} X\} (= \gldim_{\mbF} \Ga )$ , since the three missing indecomposables which are not in $\add G$ or $\add H$ are $2, \begin{smallmatrix}1\\2\end{smallmatrix},
 \begin{smallmatrix}3\\2\end{smallmatrix} $ which appear as cosyzygies of the three injectives in the $\mbF$-exact sequences and so all have $\pdim_{\mbF}=1$. 
 We have $\La= \End_{\Ga}(M)$ is given by the following quiver with relations 
 \[ 
 \xymatrix{ \bullet \ar@<-.5ex>[r] _{\alpha}\ar@<.5ex>[r]^{\beta}& \bullet\ar@<-.5ex>[r]_{\gamma} \ar@<.5ex>[r]^{\delta} & \bullet  }\quad \gamma \alpha = \delta \beta = \delta \alpha + \gamma \beta =0.
 \]
\item[(5)]
Let $\La $ be the path algebra modulo the relations:  
$\xymatrix@C=1.5pc@R=-0.2pc{
& b \ar[dr]^{\beta}  & &&\\
a\ar[ur]^{\alpha} \ar[dr]_{\gamma}  &&d ,&  \beta \alpha - \delta \gamma =0 \\
&c \ar[ur]_{\delta} & &&}$

Its Auslander-Reiten quiver is drawn in the following graphic, in the square boxes you find  $M= P_d\oplus \begin{smallmatrix} b& \\ &d\\ c& \end{smallmatrix} \oplus P_a \oplus \begin{smallmatrix} &b \\ a &\\ &c \end{smallmatrix}\oplus I_a$ and together with the remaining circled modules $G= M \oplus P_b \oplus P_c$ 
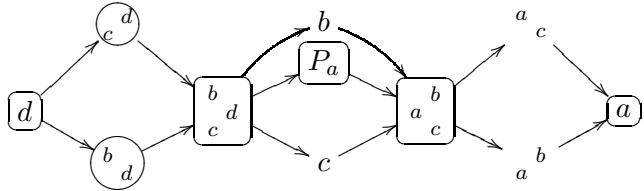
\[
\xymatrix@C=1.5pc@R=-0.2pc{ & *+[o][F-]{\begin{smallmatrix} &d\\ c& \end{smallmatrix}} \ar[ddr]&& b \ar@/^/[ddr]&&{ \begin{smallmatrix} a&\\ &c \end{smallmatrix}}\ar[ddr]& \\ 
&&& *+[F-:<3pt>]{P_a} \ar[dr]&&& \\ 
*+[F-:<3pt>]{d} \ar[uur] \ar[dr] && *+[F-:<3pt>]{\begin{smallmatrix} b& \\ &d\\ c& \end{smallmatrix}} \ar@/^/[uur] \ar[dr]\ar[ur]&&*+[F-:<3pt>]{ \begin{smallmatrix} &b \\ a &\\ &c \end{smallmatrix}}\ar[uur]\ar[dr]&& *+[F-:<3pt>]{a}\\ 
& *+[o][F-]{\begin{smallmatrix} b&\\ &d \end{smallmatrix}} \ar[ur]&&c\ar[ur]&& { \begin{smallmatrix} &b \\ a& \end{smallmatrix}} \ar[ur]& }
\]
It is very easy to see that $M$ is faithfully balanced and $\mathbf{F}= \mathbf{F}_G= \mathbf{F}^{M\oplus \kdual \La}$ fulfills $\domdim_{\mathbf{F}} \La =2 = \gldim_{\mathbf{F}} \La$. 
Now we look at $\Gamma = \End_{\La} (M)$, this is given by the path algebra of the following quiver with the overlapping zero-relations
\[
\xymatrix@-1pc{  &&3\ar[dr]^{\delta}& & &\\  1\ar[r]_{\alpha}  & 2 \ar[rr] _{\beta} \ar[ur]^{\gamma} && 4 \ar[r]_{\varepsilon} & 5,& \quad \beta \alpha =0= \delta \gamma \alpha, \; \; \; \varepsilon \beta=0=\varepsilon \delta \gamma }
\]
the vertices $1,2,3,4,5$ correspond to the summands $d,\, \begin{smallmatrix} b& \\ &d\\ c& \end{smallmatrix} ,\, P_a ,\,  \begin{smallmatrix} &b \\ a &\\ &c \end{smallmatrix},\, a$ 
in the given order. To calculate ${}_{\Ga}M=(\La , M) = \kdual (M, \kdual \La) $ we look at its four indecomposable summands 
\[
\begin{aligned}
&P_3=(P_a, M)= \kdual (M,I_a)=I_5 \quad \quad & &I_3 = \kdual (M, I_d) = (P_d, M)=P_1\\
&(P_b,M) = \kdual (M, I_b) \quad \quad & & (P_c,M)= \kdual (M, I_c)
\end{aligned}  
\]
then we apply $(-,M)$ to the $\mathbf{F}$-exact sequence $0 \to P_b \to \begin{smallmatrix} b& \\ &d\\ c& \end{smallmatrix} \to \begin{smallmatrix} &b \\ a &\\ &c \end{smallmatrix} \to I_b \to 0$ and obtain a 
projective presentation $0\to P_5 \to P_4 \to P_2 \to (P_b,M) \to 0$ (and similar for $(P_c,M)) $. From this we conclude that $(P_b,M)$, $(P_c,M)$ are two regular modules in different homogeneous tubes for the full subquiver $\widetilde{A}_2$, more precisely: 
\[(P_b,M)=:R_0\colon \xymatrix@-1.5pc{&&K\ar[dr]^{1} &&\\0 \ar[r] &K \ar[rr]_{0}\ar[ur]^{1} && K\ar[r] &0 } \quad \quad  (P_c,M)=:R_1\colon \xymatrix@-1.5pc{&&K\ar[dr]^{1} &&\\0 \ar[r] &K \ar[rr]_{1}\ar[ur]^{1} && K\ar[r] &0 } \]
then we have  $\tau^{ (\pm)} R_j = R_j$, $j=0,1$. We set now $\widetilde{G}= P_2 \oplus P_4\oplus P_5 \oplus M, \widetilde{H}=I_1\oplus  I_2\oplus I_4 \oplus M$ and define  $\widetilde{\mathbf{F}}:= \mathbf{F}_{{}_{\Ga}\widetilde{G}} = \mathbf{F}^{{}_{\Ga}\widetilde{H}}$, observe that $\add ({}_{\Ga}M) = \add(\widetilde{G}) \cap \add (\widetilde{H})$. The following sequences are $\widetilde{\mathbf{F}}$-exact (setting $R_{01}=R_0 \oplus R_1$)
\[ 
\xymatrix@-1.5pc{
0\ar[r] & P_2\ar[r]& R_{01} \ar[r] & I_3 \ar[r]& I_1 \ar[r]&0 \\
0\ar[r] & P_4\ar[r]& P_3 \ar[r] & I_3 \ar[r]& I_2 \ar[r]&0 \\
0\ar[r] & P_5\ar[r]& P_3 \ar[r] & R_{01} \ar[r] & I_4 \ar[r]&0 }
\]
We deduce $\domdim_{\widetilde{\mathbf{F}}} \Ga =2=\idim_{\widetilde{\mathbf{F}}} \kdual \Ga $.
We have $\End_{\Ga} (\widetilde{G} )^{op} \cong  \End_{\La} (G)$ has $\gldim \leq 4$ since $\gldim_{\mathbf{F}} \La \leq 2$ by Appendix, Lemma \ref{gldimF}, therefore by the same argument and the observation 
$ 2=\idim_{\widetilde{\mathbf{F}}} \kdual \Ga$ we conclude $\gldim_{\widetilde{\mathbf{F}}} \Ga \leq 2$. 
\end{itemize} 
\end{exa}

Here are examples of higher $\mbF$-Auslander algebras. 
\begin{exa}\label{ex-FAusalg-2}
\begin{itemize}
 \item[(1)] 
 A \kLRcluster $M$ with $L=M=R$ is just the same as a $k$-cluster tilting module in the sense of \cite{IyAR}. In this case, $\Gamma = \End_{\La} (M)$, $G=(M,M)= \Gamma $, $H= \kdual (M,M) = \kdual \Ga $, so $\mbF = \mbF_{\Ga}$ and so $\domdim_{\mbF} \Ga = \domdim \Ga$, $\gldim_{\mbF} \Ga = \gldim \Ga$ and we reobtain a higher Auslander algebra (this is the Krull-dimension zero case of Iyama's Auslander correspondence, see \cite{IyAC}). 
\item[(2)] Let $\Ga$ be the path algebra of $1 \to 2 \to \cdots \to n$.  Let $M_t:= \bigoplus_{i\neq t} \bigoplus_{j\geq 0} \tau^{-j}P_i$, $G_t=M_t \oplus P_t$, $\mathbf{F}_t= \mathbf{F}_{G_t}$, $1<t \leq n$, then 
  $\Ga$ has the structure of a $(t-2)$-$\mathbf{F}_t$-Auslander algebra for $t\geq 3$ and for $t=2$ we have $\domdim_{\mathbf{F}_2} \La = 1 = \gldim_{\mathbf{F}_2} \La$.  
  For large $n$ we have that $\La_3 = \End_{\La} (M_3) $ is a representation-infinite algebra with an $\mathbf{F}$-Auslander structure.  
  \item[(3)]
We consider the following quiver (of Dynkin type $E_6$) 
 $Q$\[\xymatrix@-1.5pc{ &&f&&\\
 a\ar[r]& b\ar[r]& c\ar[r] \ar[u]& d \ar[r]&e }\]
 For $x \in \{ a,b,d,e,f\}$ we define $M_x = \bigoplus_{y \neq x} \bigoplus_{j\geq 0} \tau^{-j}P_y$, $G_x=M_x \oplus P_x$, $\mathbf{F}_x=\mathbf{F}_{G_x}$. 
 Then an inspection if the AR-quiver gives the following for the path algebra $\Ga=KQ$: 
  $\Ga$ is a $2$-$\mathbf{F}_a$- and $2$-$\mathbf{F}_b$-Auslander algebra,  a $4$-$\mathbf{F}_d$- and $4$-$\mathbf{F}_f$-Auslander algebra and $6$-$\mathbf{F}_e$-Auslander algebra
 \item[(4)] Let $\Ga=K(1 \to 2 \to \cdots \to n)$ for some integer $n>3$ and we define $ M:= \bigoplus_{i=1}^{n-1}\bigoplus_{j\geq 0} \tau ^{-j} P_i$, $G= M\oplus P_n$, $H=M \oplus I_1$  and $\mbF=\mbF_G=\mbF^H$.
We find the minimal $\mbF$-projective resolution of $I_1$ (which is also the minimal $\mbF$-injective resolution of $P_n$)  as follows 
\[ 
0 \to P_n \to \begin{smallmatrix} n-1 \\ n \end{smallmatrix}  \to \begin{smallmatrix} n-2 \\ n-1 \end{smallmatrix} \to \cdots \to  \begin{smallmatrix} 1 \\ 2 \end{smallmatrix} \to I_1 \to 0  \quad (*)
\]
from this we conclude $\pdim_{\mbF} \kdual \Ga = n-1 $ and  $\domdim_{\mbF} G = n-1$. One can easily see that the highest $\pdim_{\mbF}$ is obtained at an injective module and therefore $\gldim_{\mbF} \Ga =n-1$, so we have an $(n-2)$-$\mbF$-Auslander algebra.

Let $\La = \End_{\Ga} (M)$, we denote by $P_{[M_i]}, I_{[M_i]}, S_{[M_i]}$  the projective, injective and semi-simple $\La $-module associated to $M_i\in \add (M)$. 
Let  $L={}_{\La} (G,M)= \La \oplus (P_n, M)$, $R= \kdual (M,H)= \kdual \La \oplus  \kdual (M, I_1) $ and ${}_{\La} M \in \add (L) \cap \add (R)$. Then we have 
$\Pi:= (\bigoplus_{1 < j< n}P_j , M) = \kdual (M, \bigoplus_{1 < j< n}I_j)$ is a projective-injective $\La $-module, $M= \Pi \oplus P_{[P_1]}\oplus I_{[P_1]}$, $L=\La \oplus I_{[P_1]}$, 
$R= \kdual \La \oplus P_{[P_1]}$. We verify $(\Pi, P_{[P_1]})=((S_i,M), P_{[P_1]})=(I_{[P_1]}, P_{[P_1]})=0$ for $3\leq i \leq n$ and $(I_1,M)=0$, $(S_2,M)=S_{[\begin{smallmatrix} 1 \\ 2 \end{smallmatrix}] }$. 
We apply $(-,M)$ to $(*)$ and obtain an exact sequence of $\La$-modules
\[ 
0 \to 0=(I_1, M) \to P_{[\begin{smallmatrix} 1 \\ 2 \end{smallmatrix}]} \to P_{[\begin{smallmatrix} 2 \\ 3 \end{smallmatrix}] }\to \cdots \to P_{[\begin{smallmatrix} n-1 \\ n \end{smallmatrix}] }\to I_{[P_1]} \to 0 \quad (**)
\]
This implies $\pdim I_{[P_1]} = n-2$. Now, apply $(-, P_{[P_1]})$ to $(**)$ and obtain $K=(S_{[\begin{smallmatrix} 1 \\ 2 \end{smallmatrix}] }, P_{[P_1]}) \cong \Ext_{\La}^1((S_3,M), P_{[P_1]}) = \Ext_{\La}^{n-2}(I_{[P_1]}, P_{[P_1]})$. 

We would like to see that ${}_{\La} M$ is a \nLRcluster with respect to $L$ and $R$ as before. Since we easily verify $\cogen^1_{F^H}({}_{\Ga} M) = \add (G \oplus \bigoplus_{3 \leq i <n} S_i)$ and $(S_i,M)= \Omega^{n-i}I_{[P_1]}$ by the exact sequence $(**)$, 
we have $\cogen^1_{F^R} (M) = \add (L \oplus \bigoplus_{3\leq i < n} \Omega^{n-i} I_{[P_1]})$. Now, we conclude 
\[
\begin{aligned}
&\quad \cogen^1_{F^R} (M)\cap \bigcap_{i=1}^{n-3} \Ker \Ext^i_{\La} (-,R) \\ 
&= \add (L \oplus \bigoplus_{3\leq i < n} \Omega^{n-i} I_{[P_1]})\cap \bigcap_{i=1}^{n-3} \Ker \Ext^i_{\La} (-,  P_{[P_1]}) \\
&= \add (L)
\end{aligned}
\]
where we use the calculation of $ \Ext_{\La}^{j}(I_{[P_1]}, P_{[P_1]})$, $j \geq 1$ from before.

\end{itemize}
\end{exa}

There are further examples of converting $K\mathbb{A}_n$
into a relative Auslander algebra. Here is another family of these: 
\begin{exa}\label{ex-FAusalg-3}
We fix $\Ga =K(1\to 2 \to \cdots \to n)$ for some integer $n \geq 3$ and we will also allow quotients by certain admissible $2$-sided ideals $I$. 
Our aim is to describe a family of $\mbF$-Auslander algebras which \emph{interpolate} between Iyama's example \cite[Example 2.4]{IyAR} and the usual exact structure on $\Gamod$. 
We study the following class of generators $G_{\ell}:= \Ga \oplus \bigoplus_{1\leq i\leq \ell} \bigoplus_{j>0} \tau^{-j}P_i$, $1< \ell < n-1$ \footnote{For  $\ell =1$  we  have $\mbF_1=\Ext^1_{\Ga}$ and observe $\domdim \Ga = 1 = \gldim \Ga$; $\ell =n, n-1$ are already studied in the previous examples} . 
 \begin{itemize}
\item[1.] If $(n-\ell +1)\lvert n$ (or equivalently, $(n-\ell +1)\lvert (\ell -1)$), then $\Ga $ is a $(2\frac{\ell -1}{n-\ell +1})$-minimal $\mbF_{\ell}$-Auslander-Gorenstein algebra 
(i.e., $\domdim_{\mbF_{\ell}} \Ga \geq 2\frac{\ell -1}{n-\ell +1} +1 \geq \idim_{\mbF_{\ell}} G_{\ell} $), where $\mbF_{\ell}=\mbF_{G_{\ell}}$.
If $\ell < n-1$ and $n-\ell+1$ does not divide $n$, then $\Ga$ is not a minimal $\mbF_{\ell}$-Auslander-Gorenstein algebra. \\
\emph{proof:} For $\ell \leq k \leq n$ we look at the $\mbF_{\ell}$-injective resolution of $P_k$ and here we keep track the sequence of tops (they are all simple) of the $\mbF_{\ell}$-injectives appearing, it fulfills 
$a_1 = \ell, a_2= k-(n-\ell -1), a_t= a_{t-2} -(n-\ell +1)$ for all $t\geq 3$. Now, the condition to be a  minimal $\mbF_{\ell}$-Auslander-Gorenstein algebra is equivalent to that there is one $t$ (for all $k$) such that $a_t=1$. 
Since $t$ has to work for all $k$ (and $\ell < n-1$), we conclude that $t$ has to be uneven, say $t=2s+1$ (then it is an $2s$-minimal $\mbF_{\ell}$-Auslander-Gorenstein algebra). Now, the recursion tells us 
$1=a_t= a_{t-2}- (n-\ell +1) = a_{t-2s} - s(n-\ell +1)= \ell -s(n-\ell +1)$, so it follows $s=\frac{\ell -1}{n-\ell +1}$. 

\item[2.] But from the shape of the Auslander-Reiten quiver of $\Gamma$ we can conclude that the maximal $\pdim_{\mbF_{\ell}}$ is obtained at an injective module, therefore $\gldim_{\mbF_{\ell}} \Gamma = \pdim_{\mbF_{\ell}} \kdual \Gamma$ and we have: \\
$\Ga$ is a $k$-$\mbF_{\ell}$-Auslander algebra (for some $k$) if and only if $(n-\ell +1) \lvert n$ and in this case $k = 2( \frac{\ell -1}{n-\ell+1})$. 

 \item[3.] Assume $I$ is a $2$-sided admissible ideal with $\{ X \mid IX=0\} \subseteq \{ X \mid \dim_K X \geq n-\ell+2\} $. We define $\overline{G}_{\ell}:= \Ga /I \otimes_{\Ga} G_{\ell} $ is a generator for $\Ga/I$ and we set $\overline{\mbF}_{\ell}:= \mbF_{\overline{G}_{\ell}}$. Since we can use the \emph{same} $\mbF$-projective and $\mbF$-injective resolutions (because of the choice of the ideal) we have: 
  $\Ga$ is a $k$-$\mbF_{\ell}$-Auslander algebra)  if and only if $\Ga /I$ is an $k$-$\overline{\mbF}_{\ell}$-Auslander algebra) \\
In particular, if we set $I=\rad^{n-\ell +1}(\Ga)$, then we have $\overline{G}_{\ell}  = \Ga /I$ and if $(n-\ell+1) \lvert n$ then we get a (non-relative) $2(\frac{\ell-1}{n-\ell +1})$-Auslander algebra. \\
If we allow $\ell=n-1$ (cf. previous example), this describes the $(n-1)$-Auslander algebra of Iyama \cite[Example 2.4]{IyAR}.
\end{itemize}
\end{exa}

Conceptually the same family can be defined more generally for Nakayama algebras, we explain this in the selfinjective Nakayama algebra case: 
\begin{exa}
Let $C_n$ be the oriented cycle quiver with arrows $i \to i+1 (\rm{mod}\; n)$ and $J \subseteq KC_n$ be the ideal generated by the arrows, $N \in \N$, we define $\Ga := KC_n /J^N$ (this is a self-injective Nakayama algebra). Let $n-\ell +1 < N$ and $M_{\ell\ell \geq n-\ell+1}$ be the direct sum of all modules of vector space dimension $\geq n-\ell +1$ and let $X_n$ be the direct sum of all modules having $S_n$ as a composition factor and vector space dimension $< n- \ell +1$, we define $G_{\ell} = M_{\ell \ell \geq n- \ell +1}\oplus X_n$ and $\mbF_{\ell}= \mbF_{G_\ell}$. Then $G_n$ is the Auslander generator and for $\ell=n-1$ we have $\Ga$ is an $(n-2)$- $\mbF_{n-1}$- Auslander algebra. Moreover, for 
$1<\ell < n-1$ we have $\Ga$ is a $k$-$\mbF_{\ell}$-Auslander algebra (for some $k$) if and only if $(n-\ell +1) \lvert n$, and in this case $k = 2 \frac{\ell -1}{n-\ell +1}$. The proof is exactly the same as in the previous example.  
\end{exa}

\section{The 4-tuple assignment}
Now we consider $4$-tuples $(\La , M, L, G)$ with $\La$ a finite-dimensional algebra and $M,L,G$ finite-dimensional $\La $-modules.
We define the following equivalence relation between these 4-tuples: $(\La, M, L, G)$ is equivalent to $(\La', M', L', G')$ if there is an equivalence of categories $\La$-$\modu \xrightarrow{\sim} \La'$-$\modu$ restricting to equivalences $\add (M) \xrightarrow{\sim} \add (M')$, $\add (L) \xrightarrow{\sim} \add (L')$ and $\add (G) \xrightarrow{\sim} \add (G')$. We denote by $[\La, M, L, G]$ the equivalence class of a 4-tuple and we may assume the algebra and all the modules appearing in the equivalence class to be basic.

To establish a relative version of cotilting correspondence which is an involution, we will need the following definition.
\begin{dfn} \label{bAS} 
We define the following assignment 
\[ [\La , M, L, G] \mapsto [\Gamma, N, \widetilde{L}, \widetilde{G} ]  \]
with $\Gamma = \End_{\La}(M) $, $N={}_{\Ga} M$, $\widetilde{L} = (G, M)$, $\widetilde{G} = (L,M)$ and call this the balanced Auslander-Solberg assignment or just the 4-tuple assignment. \\
The dual $4$-tuple assignment is the following 
\[ [\La , M, R, H] \mapsto [\Gamma, N, \widetilde{R}, \widetilde{H} ]  \]
with $\Gamma = \End(M) $, $N={}_{\Ga} M$, $\widetilde{R} = \kdual (M, H)$, $\widetilde{H} = \kdual (M, R)$. 
 Since, we will always consider pairs $(G,H)$ and $(L,R)$ which determine each other, we will in later proofs combine the two assignments into a $6$-tuple assignment 
 \[ [\La , M, L,R, G,H] \mapsto [\Gamma, N, \widetilde{L}, \widetilde{R},\widetilde{G}, \widetilde{H}]  \]
with $\Gamma = \End(M) $, $N={}_{\Ga} M$, $\widetilde{L} = (G, M)$, $\widetilde{R} = \kdual (M, H)$, $\widetilde{G} = (L,M)$, $\widetilde{H} = \kdual (M, R)$. 
 \end{dfn}

\begin{lem} \label{bAS-correspondence}
Keep the above notations. Then we have
\begin{itemize}
\item[(1)] 
The $4$-tuple assignment restricts to an involution on the set of $4$-tuples $[ \La, M, L, G]$ with 
$\La \in \add (G)$, $\mbF=\mbF_G$, 
$M$ is $1$-$\mbF$-faithful, $M$ is an $\mbF$-dualizing summand of $L$ and $L$ is the left end of an $\mbF$-exact strong $\add(M)$-dualizing sequence. 
\item[(2)] The dual $4$-tuple assignment restricts to an involution on the set of $4$-tuples $[\La, M, R, H]$ with 
$\kdual \La \in \add (H)$, $\mbF=\mbF^H$, 
$M$ is $1$-$\mbF$-faithful, $M$ is an $\mbF$-codualizing summand of $R$ and $R$ is the right end of an $\mbF$-exact strong $M$-dualizing sequence. 
\end{itemize}
\end{lem}

\begin{proof}
We take a $6$-tuple $[\La , M, L,R, G,H]$ with $\La \in \add(G), \kdual \La \in \add(H)$, $\mbF=\mbF_G=\mbF^H$, $M$ $1$-$\mbF$-faithful and there is an $\mbF$-exact strong $M$-dualizing sequence 
$0\to L \to M_0 \to M_1 \to R \to 0$. We want to see that applying the $6$-tuple assignment gives an involution. So consider 
$[\Gamma, N, \widetilde{L}, \widetilde{R},\widetilde{G}, \widetilde{H}  ]$ with $\Gamma = \End(M) $, $N={}_{\Ga} M$, $\widetilde{L} = (G, M)$, $\widetilde{R} = \kdual (M, H)$, $\widetilde{G} = (L,M)$, $\widetilde{H} = \kdual (M, R)$. 
Clearly, $\Ga \in \add(\widetilde{G}) $, $\kdual \Ga \in \add (\widetilde{H})$ since $M \in \add (L) \cap \add (R)$ and since $L$ and $R$ are ends of an $\add({}_{\La}M)$-dualizing sequence we have 
$\mbF_{\widetilde{G}}=\mbF^{\widetilde{H}}=: \widetilde{\mbF}$. Since $L$ is left end of a strong $\add(M)$-dualizing sequence, we have by Lemma \ref{CorRelFB} that $\widetilde{G} =(L,M) \in \cogen_{\widetilde{\mbF}}^1(N)$, this means $N$ is $1$-$\widetilde{\mbF}$-faithful. 
Since $M$ is $1$-$\mbF$-faithful we get a strong $\add(N)$-dualizing sequence 
$0 \to \widetilde{L}\to \widetilde{N_0} \to \widetilde{N_1} \to \widetilde{R}\to 0 $ by Lemma \ref{CorRelFB}. It split off the summand $0\to N\xrightarrow{1} N \xrightarrow{0}N \xrightarrow{1}N \to 0$ and obtain an exact sequence 
\[
0 \to \widetilde{L}'\to N_0 \to N_1 \to \widetilde{R}'\to 0  
\eqno{(*)}\]
with $N_i \in \add (N)$. 
The only missing property is that $(*)$ is $\widetilde{\mbF}$-exact. We first observe that $N_i = \kdual (M, I_i)$ with $0\to H' \to I_1 \to I_0$ is an injective copresentation, $H=\kdual \La  \oplus H'$. Since $(\widetilde{G}, -)=((L,M), -)$ is left exact, it is enough to check that it is also right exact on $(*)$. 
Now, since $L\in \cogen_{\mbF^H}^1(M)$ we have a natural isomorphism $\kdual (L,H) \to ((L,M), \kdual (M,H))$ by Lemma \ref{cogenF} (1). 
In particular, we have a natural isomorphism $(\widetilde{G}, N_i) = ((L,M), \kdual (M, I_i)) \to \kdual (L, I_i)$ since $I_i \in \add(H)$. 
This means when we apply $(\widetilde{G}, -)$ to the last three nonzero terms of $(*)$ we get an exact sequence which identifies under the just mentioned natural isomorphism with 
\[
\kdual (L, I_0) \to \kdual (L, I_1 ) \to \kdual (L, H') \to 0
\]
and this is exact. 
\end{proof}

\section{Relative cotilting theory}
Relative cotilting modules are introduced in \cite{ASoII}. 
\begin{dfn}\label{Fcotilting}
Let $\mbF=\mbF^H \subseteq \Ext^1_{\La}$ be an additive subbifunctor with $H$ a cogenerator. 
We call a $\La$-module $C$ a $k$-$\mbF$-cotilting module if
\begin{itemize}
    \item[(i)] it is $\mbF$-self-orthogonal (i.e., $\Ext^{\! >0}_{\mbF}(C,C)=0$),
    \item[(ii)] $\id_{\mbF} C \leq k$, and
    \item[(iii)] there is an $\mbF$-exact sequence $0 \to C_k \to \cdots \to C_1 \to C_0 \to H \to 0$ with $C_i \in \add (C)$. 
\end{itemize}  
\end{dfn}
We recall a result of Wei. Partially, it is already proven in \cite{ARapp}. 
\begin{thm}\label{weisresult} $($\cite[Theorem 3.10]{W}$)$ Let $\mbF\subseteq \Ext^1(-,-)$ be an additive subbifunctor with enough projectives and injectives, $C$ be a $\La$-module and let $k \geq 1$. Then the following are equivalent
\begin{itemize}
\item[(1)] $C$ is a $k$-$\mbF$-cotilting module.
\item[(2)] $\cogen^{k-1}_{\mbF}(C) = \bigcap_{i\geq 1} \Ker \Ext^i_{\mbF}(-,C)$.
\end{itemize}
In this case, we also have 
$ \copres^{k-1}_{\mbF}(C)=\cogen^{k-1}_{\mbF} (C)$ and  
\[ \cogen^{k-1}_{\mbF}(C) = \cogen^{k}_{\mbF}(C) = \cogen^{k+1}_{\mbF}(C) = \cdots = \cogen^{\infty }_{\mbF}(C).\]
\end{thm}

\begin{lem} \label{HigherTilt} 
Let $\mbF=\mbF_G=\mbF^H\subseteq \Ext^1(-,-)$ be an additive subbifunctor with enough projectives and injectives. Let $k \geq 1$ and $M$ an $\mbF$-self-orthogonal module. 
If $\id_{\mbF} M \leq 1$ and $H \in \gen^\mbF_{k-1} (M)$, then $C=M\oplus \Omega^k_{M} H$ is a $k$-$\mbF$-cotilting module. Furthermore, we have 
\[ \cogen_\mbF^{k-1} (M) = \bigcap_{i\geq 1} \Ker \Ext^i_{\mbF}(-,C). \]
Then $M$ is an $(k-1)$-$\mbF$-dualizing summand of $C$. 
\end{lem}

\begin{proof} 
It is straightforward to check $\id_{\mbF} C \leq k$ by induction on $k$. Now we check $C$ is $\mbF$-self-orthogonal: \\
(i) using the definition of $\Omega_M^kH $ by approximations one easily checks $\Ext_\mbF^i(M, \Omega_M^kH) =0$ for all $i \geq 1$,\\ 
(ii) then using $M$ $\mbF$-selforthogonal one shows $\Ext_\mbF^i(\Omega_M^kH , M) \cong \Ext_\mbF^{i+1}(\Omega^{k-1}_MH , M)=0$ for all $i \geq 1$ , here the last space is zero since $\id_{\mbF} M \leq 1$, \\
(iii) to see $\Ext_\mbF^i(\Omega^k_MH, \Omega_M^kH) =0$ for all $i \geq 1$ we use (i) and $\id_{\mbF} C \leq k$. More precisely, one applies $(-,\Omega^k_MH) $ to the $\mbF$-exact sequences 
$0 \to \Omega_M^tH \to M_{t-1} \to \Omega_M^{t-1}H \to 0$ with $M_{t-1} \in \add (M)$. We then can conclude 
$\Ext_\mbF^i (\Omega^k_MH, \Omega_M^kH) \cong  \Ext_\mbF^{i+1} (\Omega^{k-1}_MH, \Omega_M^kH) \cong \cdots \cong \Ext_\mbF^{i+k} (H, \Omega_M^kH)=0$ since $\id_{\mbF}C \leq k$. \\
Together with $H \in \gen_{k-1}^{\mbF}(M)$, we conclude that $C$ is an $k$-$\mbF$-cotilting module. Furthermore, it is easy to check $\cogen^{k-1}_{\mbF}(M) \subseteq \bigcap_{i\geq 1} \Ker \Ext^i_{\mbF}(-,C)$. 
We prove the other inclusion by induction over $k$. \\
Let $k=1$. 
By definition we have $C \in \cogen_{\mbF} (M)$ and this implies using Wei's result $\Ker \Ext^1_{\mbF}(-,C) = \cogen_{\mbF} (C) \subseteq \cogen_{\mbF}(M)$.\\
Let $k \geq 2$. Since $C$ does depend on $k$ we denote it in this part of the proof with $C_k$. We first observe \\
(i) $\bigcap_{i\geq 1} \Ker \Ext^i_{\mbF}(-,C_k) \subseteq \bigcap_{i\geq 1} \Ker \Ext^i_{\mbF}(-,C_{k-1})$. This is easy to see using that there is an $\mbF$-exact sequence $0 \to C_k \to M' \to C_{k-1} \to 0$ with $M'\in \add (M)$. \\
(ii) By induction hypothesis we may assume $\bigcap_{i\geq 1} \Ker \Ext^i_{\mbF}(-,C_k) \subseteq \cogen_\mbF^{k-2}(M)$. \\
Let $X \in \bigcap_{i\geq 1} \Ker \Ext^i_{\mbF}(-,C_k)$, so there exists an $\mbF$-exact sequence 
$0\to X \to M^0 \to \cdots \to M^{k-2} \to Z \to 0$
with $M^i \in \add (M)$, $(-,M)$ exact on it. 
We claim $Z \in \cogen_{\mbF}(M)= \Ker \Ext_\mbF^1(-, C_1)$. We split the sequence up in short $\mbF$-exact sequences $X:=X^0, Z:= X^{k-1}$ and $0 \to X^t \to M^t \to X^{t+1} \to 0$, $0 \leq t \leq k-2 $. Since $(-,M)$ is exact on the sequence for $t=k-2$, we conclude $\Ext^1_{\mbF}(Z,M)=0$. So, it is enough to see 
$\Ext_\mbF^1(Z, \Omega_M^1H)=0$. We first show:\\
(iii) $\Ext_\mbF^1(Z, \Omega_M^1H)\cong \Ext_\mbF^k(Z, \Omega_M^kH)$ by applying $(Z, -)$ to the sequences 
$0 \to \Omega^t_MH \to M_{t-1} \to \Omega_M^{t-1}H \to 0$ and concluding $\Ext_\mbF^i(Z, \Omega_M^{t-1}H) \cong \Ext_\mbF^{i+1} (Z, \Omega_M^tH)$ for all $i \geq 1$. Applying this iteratively gives (iii). Now, we prove: \\
(iv) $\Ext_\mbF^k(Z, \Omega_M^kH)\cong \Ext^1(X, \Omega_M^kH)$ by applying $(-, \Omega_M^kH)$ to the short exact sequences 
$0 \to X^t \to M^t \to X^{t+1} \to 0$ and conclude $\Ext^i_{\mbF}(X^t, \Omega_M^kH) \cong \Ext^{i+1}_{\mbF}(X^{t+1}, \Omega_M^kH) $ for all $i \geq 1$. Applying this iteratively gives (iv). \\
But since $X \in \bigcap_{i\geq 1} \Ker \Ext^i_{\mbF}(-,C_k)$ we have $\Ext^1(X, \Omega_M^kH)=0$ and therefore, using (iii) and (iv) this implies $\Ext_\mbF^1(Z, \Omega_M^1H)=0$.
\end{proof} 

\begin{rem}
If $C$ is a $1$-$\mbF$-cotilting module and $M$ an $\mbF$-dualizing summand, then we have $M=C$. 
Therefore, non-trivial $\mbF$-dualizing summands only appear in the theory of $\mbF$-cotilting modules with $\id_{\mbF} > 1$.  
\end{rem}

\begin{exa} Let $M\in \Lamod$ be rigid (i.e., $\Ext_{\La}^1(M,M)=0$) and also $X:=\Omega^-M$ be rigid, then for $H=X\oplus \kdual \La$, $\mbF=\mbF^H$ we have $\id_{\mbF} M\leq 1$, $M$ is $\mbF$-self-orthogonal and $X \in \gen_1(M)$. \\
If we now assume additionally that $M$ is faithfully balanced and $\Ext^{1,2}(M\oplus X,X)=0$, then we have 
\[ H\in \gen_1^{\mbF}(M) =\gen_1(M) \cap \bigcap_{i=1}^2 \Ker \Ext^i(-,X) \]
(cf. Example \ref{exa-cogen-F}) implying that $M$ is $1$-$\mbF$-faithful. 
In particular, we have then $C:=M\oplus \Omega_M H$ is a $1$-$\mbF$-cotilting module with 
\[ \cogen_{\mbF}(M)= \bigcap_{i\geq 1}\Ker \Ext^i_{\mbF}(-,C).
\]
\end{exa}

\begin{exa}
 Let $X$ be an arbitrary faithfully balanced module and $k\geq 1$. If $\tau X \in \cogen^{k-1}(X)$, then $\cogen^{k-1} (X) $ is the $\mbF^X$-perpendicular category $\bigcap_{i\geq 1} \Ker \Ext^i_{\mbF^X} (-, C)$ for the $\mbF^X$-$k$-coltilting module $C= X \oplus\Omega_X^k \kdual \La $.
If $\add (X)$ is, for example, $\tau$-stable then $\tau X \in \cogen^{k-1}(X)$. 
\end{exa}
More generally we will study the $\mbF$-cotilting modules obtained from a $1$-$\mbF$-faithful $\mbF$-injective module as special cotilting modules (in section 9).

Let us fix an $\mbF$-exact resolution by $\mbF$-projectives of $H$ (with $\add (H)= \mcI(\mbF) $) 
\[ \cdots\to P_2 \to P_1 \to P_0 \to H \to 0. \]

Then we obtain the relative version of \cite[Theorem 1.1]{IZ} as follows, let 
$\cotilt_n^{\mbF}(\La)$ be the set of basic isomorphism classes of $n$-$\mbF$-cotilting $\La$-modules. It is naturally a poset with respect $C \leq C'$ if and only if $C \in \bigcap_{i\geq 1} \Ker \Ext^i_{\mbF} (-, C') $. 

\begin{lem} Let $\mbF=\mbF_G=\mbF^H$ and $n \geq 1$, we define $P:= \bigoplus_{j=0}^{n-1} P_j$. 
If $\id_{\mbF} P \leq n$ and $\id_{\mbF} \Omega_P^n H \leq n$, then 
$C= P \oplus \Omega_P^n H$ is an $n$-$\mbF$-cotilting module and it is the minimum element in $\cotilt_n^{\mbF}(\La)$. \\
Furthermore, if $\id_{\mbF} P_j \leq j+1$, $1 \leq j \leq n-1$, then $\id_{\mbF} P\oplus \Omega_P^n H \leq n$. 
\end{lem}

\begin{proof}
We check that $\id_{\mbF} C \leq n$ implies that $C$ is $\mbF$-selforthogonal: Observe that $\Omega_P^nH = \Omega_\mbF^nH$ and let $i \geq 1$, then we have  
$ \Ext_\mbF^i(C,C) = \Ext_\mbF^i (\Omega_\mbF^nH, C)= \Ext_\mbF^{i+n}(H,C) =0$ since $\id_{\mbF}C \leq n$. \\
Since the last condition is fulfilled by definition of $C$, we can conclude that $C$ is an $n$-$\mbF$-cotilting module. \\
If $L\in \cotilt_n^{\mbF}(\La)$, then we have by definition of $C$ that 
$\Ext_\mbF^i(C,L) = \Ext^{i+n}_{\mbF} (H, L)=0$ since $\id_{\mbF} L \leq n$. Therefore $C$ is the minimum. \\
The last claim is a straight forward induction over $n$. 
For $n=1$ the claim follows from the previous lemma. For the induction step apply $(-,M)$ to the $\mbF$-exact sequence 
$0 \to \Omega_P^nH \to P_{n-1} \to \Omega_P^{n-1}H \to 0$, 
by hypothesis $\id_{\mbF} P_{n-1}\leq n$, $\id_{\mbF} \Omega_P^{n-1}H \leq n-1$ we conclude $\id_{\mbF} \Omega_P^nH\leq n$. 
\end{proof}

In particular, if $P_0, \ldots , P_{n-1}$ are $\mbF$-injective, this will be referred to as $\mbF$-$\domdim \La \geq n$, then the previous lemma applies.

\subsection{The relative cotilting correspondence}

We give a generalization of the cotilting correspondence to a relative set-up together with a relative dualizing summand - this is a generalization of Auslander-Solberg's main results in \cite{ASoII, ASoIII} which we reobtain as a corollary. 
We will use the $4$-tuple assignments for our theorem (see Definition \ref{bAS}, Lemma \ref{bAS-correspondence}).

As before, we fix an additive subbifunctor $\mbF=\mbF_G=\mbF^H$ of $\Ext^1_{\La}(-,-)$ for some generator $G$ and cogenerator $H$. 

Define 
\[
\mathsf{K}^{+, b}_{\mbF}(\add (H))=\{Y\in \mathsf{K}^+(\add (H)) \;\vert\; \exists\; n\in \Z \;\text{such that}\; \mathsf{H}^i(\Hom_{\La}(Y,H))=0 \; \text{for} \; i\geq n \}
\]
then we have $\mathsf{D}^b_{\mbF}(\Lamod) \simeq  \mathsf{K}^{+,b}_{\mbF}(\add (H))$ as triangulated categories, where $\mathsf{D}^b_{\mbF}(\Lamod)$ is the bounded derived category of the exact category $\Lamod $ with the exact structure induced by $\mbF$. For more on the derived category of an exact category we refer to \cite{Nee, KelDer, Pan}. As in the standard case, one can prove that an $\mbF$-self-orthogonal $\La$-module $L$ is an $\mbF$-cotilting module if and only if $\mathsf{Thick}(L)=\mathsf{K}^b(\add (H))$ where by $\mathsf{Thick}(L)$ we mean the smallest triangulated subcategory of $\mathsf{K}^b(\add (H))$ which contains $L$ and closed under direct summands. We also have the following lemma which can be proved by the same argument in the standard case (cf. \cite{CHU, AiI}).

\begin{lem}\label{mutation} Let $L=M \oplus U$ be a basic $\mbF$-cotilting module.
\begin{itemize}
\item[(1)] If there exists an $\mbF$-exact sequence  $0 \to U \xrightarrow{f} M_0 \to V \to 0$ with $f$ the left minimal $\add (M)$-approximation of $U$, then $M \oplus V$ is a basic $\mbF$-cotilting module with $\id_{\mbF} (M \oplus V) \leq \id_{\mbF} L$. 
Furthermore, this $\mbF$-exact sequence $($after adding $1_M$ to $f$ and its cokernel $)$ gives rise to a strong $0$-$\add(M)$-dualizing sequence with 
$\Ext^i(U \oplus M, V \oplus M)=0$ for $i \geq 1$. 

\item[(2)] If there exists an $\mbF$-exact sequence $0 \to V \to M_1 \xrightarrow{g} U \to 0$ with $g$ the right minimal $\add (M)$-approximation of $U$, then $M \oplus V$ is a basic $\mbF$-cotilting module with $\id_{\mbF} (M \oplus V) \leq \id_{\mbF} L +1$. 
Again this gives rise to a strong $0$-$\add(M)$-dualizing sequence with $\Ext^i(V \oplus M, U \oplus M)=0$ for $i \geq 1$.
\end{itemize}
\end{lem}

Now we are ready to present our improvement of Auslander and Solberg's results.  
Recall, the $4$-tuple assignment 
\[
 [\La, M, L, G] \mapsto [\Ga, {_{\Ga}M}, \widetilde{L}, \widetilde{G}]
\]
where $\Ga=\End_{\La}(M)$, $\widetilde{L}=(G,M)$ and $\widetilde{G}=(L,M)$.
We also consider the dual $4$-tuple assignment 
\[
[\La, M, R, H] \mapsto [\Ga, {_{\Ga}M}, \widetilde{R}, \widetilde{H}]
\]
where $\Ga=\End_{\La}(M)$, $\widetilde{R}=\kdual (M,H)$ and $\widetilde{H}=\kdual (M,R)$. 

\begin{thm} \label{relBBthm}
Keep the above notations. Then we have
\begin{itemize}
 \item[(1)]
The $4$-tuple assignment restricts to an involution on the set of $4$-tuples $[\La , M, L, G]$ satisfying 
\begin{itemize}
\item[(1a)] $\La \in \add (G)$, $\mbF=\mbF_G$, 
\item[(1b)] $L$ is $\mbF$-cotilting and $M$ is an $\mbF$-dualizing summand of $L$. 
\end{itemize}
\item[(2)]
The dual $4$-tuple assignment restricts to an involution on the set of $4$-tuples $[\La , M, R, H]$ satisfying 
\begin{itemize}
\item[(2a)] $\kdual \La \in \add (H)$, $\mbF=\mbF^H$, 
\item[(2b)] $R$ is $\mbF$-cotilting and
$M$ is an $\mbF$-codualizing summand of $R$ $($that is, $M\in \add(R)$ and $R \in \gen_1^{\mbF}(M)$ $)$. 
\end{itemize}
Furthermore, for an assignment $[\La, M, R, H] \mapsto [\Ga, {_{\Ga}M}, \widetilde{R}, \widetilde{H}]$ we have $\id_{\mbF^H} R= \id_{\mbF^{ \widetilde{H}}} \widetilde{R}. $
\end{itemize}
\end{thm} 
\begin{proof} 
We prove (1) and (2) together. 

We want to use Lemma \ref{bAS-correspondence}, so we first prove that (1b) (or (2b)) implies that $M$ is $1$-$\mbF$-faithful. To prove $M$ is $1$-$\mbF$-faithful we need to show the natural map $(M, H) \otimes_{\Ga} (G,M) \to (G, H)$ is an isomorphism, where $\Ga=\End_{\La}(M)$. Since $L$ is $\mbF$-cotilting it is $1$-$\mbF$-faithful and thus the natural map $(L, H) \otimes_{B } (G,L) \to (G, H)$ is an isomorphism, where $B=\End_{\La}(L)$. By Lemma \ref{cogenF} (1), $M$ being an $\mbF$-dualizing summand of $L$ is equivalent to that the natural map $(M, H) \otimes_{\Ga } (L,M) \to (L, H)$ is an isomorphism. Hence we have
\[
(M, H) \otimes_{\Ga} (G,M)\xrightarrow{\cong} ((M, H) \otimes_{\Ga} (L,M))\otimes_{B } (G,L) \xrightarrow{\cong} (L, H) \otimes_{B } (G,L) \xrightarrow{\cong} (G, H)
\]
as desired. Since $L$ is $\mbF$-cotilting and $M$ is an $\mbF$-dualizing summand of $L$, we have an $\mbF$-exact strong $\add(M)$-dualizing sequence $0\to L \to M_0 \to M_1 \to R \to 0$ with $M_i\in \add (M)$. By Lemma \ref{mutation} we see that $R$ is also an $\mbF$-cotilting module. Now, by Lemma \ref{bAS-correspondence} the 6-tuple assignment restricts to an involution on the set of 6-tuples $[\La , M, L, R, G, H]$ satisfying the conditions (1a), (1b), (2a) and (2b)  if we prove that $\widetilde{R}:=\kdual(M,H)$ and $\widetilde{L}:=(G,M)$ are $\widetilde{\mbF}$-cotilting modules, where $\widetilde{\mbF}:=\mbF_{\widetilde{G}}=\mbF^{\widetilde{H}}$, $\widetilde{G}=(L,M)$ and $\widetilde{H}=\kdual(M,R)$.

Assume $\id_{\mbF} R=n$, then we have $\mbF$-exact sequences
\[
0 \to R \to H^0 \to H^1 \to \cdots \to H^{n-1} \to H^n \to 0 \eqno{(*)}
\]
and
\[
0 \to R_n \to R_{n-1} \to \cdots \to R_1 \to R_0 \to H \to 0. \eqno{(**)}
\]
The functor $(M,-)$ is exact on both $(*)$ and $(**)$. Applying $\kdual (M,-)$ to $(**)$ we get an exact sequence 
\[
0 \to \kdual (M,H)= \widetilde{R} \to \kdual (M,R_0) \to \kdual (M,R_1) \to \cdots \to \kdual (M,R_{n-1}) \to \kdual (M,R_n) \to 0 \eqno{(\star\star)}
\]
of $\Ga$-modules, where each $\kdual (M,R_i)\in \add (\widetilde{H})$ is an $\widetilde{\mbF}$-injective module. We claim that this sequence is $\widetilde{\mbF}$-exact which will imply that $(\star\star)$ is an $\widetilde{\mbF}$-injective resolution of $\widetilde{R}$ and so $\id_{\widetilde{\mbF}} \widetilde{R} \leq n$. Consider the following commutative diagram
\[
\xymatrix @C=0.8cm{
0 \ar[r] & ((L,M),\kdual (M,H)) \ar[r]\ar[d]_{\cong} & ((L,M),\kdual (M,R_0)) \ar[r]\ar[d]_{\cong}  & \cdots \ar[r] & ((L,M),\kdual (M,R_n)) \ar[r]\ar[d]_{\cong} & 0 \\
0 \ar[r] & \kdual ((M,H)\otimes (L,M)) \ar[r] & \kdual ((M,R_0)\otimes (L,M)) \ar[r] & \cdots \ar[r] & \kdual ((M,R_n)\otimes (L,M)) \ar[r] &0 \\
0 \ar[r] & \kdual (L,H) \ar[r]\ar[u]^{\cong} & \kdual (L,R_0) \ar[r]\ar[u]^{\cong} & \cdots \ar[r] & \kdual (L,R_n) \ar[r]\ar[u]^{\cong} & 0
}
\]
The first row and the second row are naturally isomorphic by the Hom-Tensor adjunction, the second row and the last row are naturally isomorphic because $H, R\in \gen_1^{\mbF_L} (M)$. The last row is obtained by applying the functor $\kdual (L,-)$ to $(**)$ and it is exact . Hence the first row is exact and the claim follows.

Similarly, apply the functor $\kdual (M,-)$ to $(*)$ we will get an $\widetilde{\mbF}$-exact sequence
\[
0 \to \kdual (M,H_n) \to \kdual (M,H_{n-1})  \to \cdots \to \kdual (M,H_1) \to \kdual (M,H_0) \to \kdual (M,R)= \widetilde{H} \to 0 \eqno{(\star)}
\]
with $\kdual (M,H_i)\in \add (\widetilde{R})$.
Now applying the functor $\kdual(\widetilde{R},-)=\kdual ((M,H),-)$ to $(\star\star)$ we will get the first row of the following commutative diagram
\[
\xymatrix @C=0.8cm{
0 \ar[r] & (\kdual (M,H),\kdual (M,H)) \ar[r] & (\kdual (M,H),\kdual (M,R_0)) \ar[r]  & \cdots \ar[r] & (\kdual (M,H),\kdual (M,R_n)) \ar[r] & 0 \\
0 \ar[r] & (H,H) \ar[r]\ar[u]^{\cong}_{\kdual (M,-)} & (R_0,H) \ar[r]\ar[u]^{\cong}_{\kdual (M,-)} & \cdots \ar[r] & (R_n,H) \ar[r]\ar[u]^{\cong}_{\kdual (M,-)} & 0
}
\]
The lower row is exact because $(**)$ is $\mbF$-exact and the vertical arrows are isomorphisms because $H, R\in \gen_1 (M)$. Therefore the upper row is exact and this means $\Ext^i_{\widetilde{\mbF}}(\widetilde{R},\widetilde{R})=0$ for $i>0$. Combining $(\star)$ and $(\star \star)$, we see that $\widetilde{R}$ is an $\widetilde{\mbF}$-cotilting module. According to the proof of Lemma \ref{bAS-correspondence}, there is a strong $\add({_{\Ga}M})$-dualizing sequnce $0\to  \widetilde{L} \to \widetilde{M_0} \to \widetilde{M_1} \to \widetilde{R} \to 0$ with $\widetilde{M_i}\in \add ({_{\Ga}M})$. Again by Lemma \ref{mutation}, we conclude that $\widetilde{L}$ is an $\widetilde{\mbF}$-cotilting module.

Finally, since the dual 4-tuple assignment restricts to an involution we have $\id_{\mbF} R=\id_{\widetilde{\mbF}} \widetilde{R}$.



\end{proof}

\begin{cor}

\begin{itemize}
\item[(1)] The functors $(-, {_{\La}} M):\Lamod \longleftrightarrow \Gamod : (-, {_{\Ga}} M)$ restrict to dualities ${^{{}_{0<}\perp_{\mbF}} L} \longleftrightarrow {^{{}_{0<}\perp_{\widetilde{\mbF}}} \widetilde{L} }$ and ${^{{}_{0<}\perp_{\mbF}} R} \longleftrightarrow {^{{}_{0<}\perp_{\widetilde{\mbF}}} \widetilde{R}}$.

\item[(2)] We have $\id_{\mbF} R \leq \id_{\mbF} L \leq \id_{\mbF} R + 2$ and $\id_{\widetilde{\mbF}} \widetilde{R} \leq \id_{\widetilde{\mbF}} \widetilde{L} \leq \id_{\widetilde{\mbF}} \widetilde{R} + 2$. \end{itemize}
\end{cor}
\begin{proof}

(1) Given $X\in {^{{}_{0<}\perp_{\mbF}} L}$ we need to show that $(X, {_{\La}} M) \in {^{{}_{0<}\perp_{\widetilde{\mbF}}} \widetilde{L} }$ and it is enought to show $(X, {_{\La}} M) \in \copres^{\infty}_{{\widetilde{\mbF}}}(\widetilde{L})$ by Theorem \ref{weisresult}. Taking an $\mbF$-projective resolution $\cdots \to P_1 \to P_0 \to X \to 0$ of $X$ and applying $(-, {_{\La}} M)$ to get a complex $0 \to (X,M) \to (P_0, M) \to (P_1, M) \to \cdots$. A standard argument shows that it is ${\widetilde{\mbF}}$-exact and therefore $(X, {_{\La}} M) \in \copres^{\infty}_{{\widetilde{\mbF}}}(\widetilde{L})$. Now given $Y\in {^{{}_{0<}\perp_{\mbF}} R}$ we will prove that $(Y, {_{\La}} M)\in {^{{}_{0<}\perp_{\widetilde{\mbF}}} \widetilde{R}}$. Applying $((Y, M), -)$ to the $\widetilde{\mbF}$-injective resolution $(\star\star)$ of $\widetilde{R}$ gives a complex $0 \to ((Y,M), \widetilde{R}) \to ((Y,M), \kdual(M, R_0)) \to \cdots \to ((Y,M), \kdual(M, R_n)) \to 0$. One can easily check that it is in fact exact and thus $(Y, {_{\La}} M)\in {^{{}_{0<}\perp_{\widetilde{\mbF}}} \widetilde{R}}$. \\
(2) follows from Lemma \ref{mutation}.
\end{proof}

\begin{rem}
In particular, if we take $M=L$ to be the trivial $\mbF$-dualizing summand then we have $[\La, L ,L, G] \mapsto [\Ga, {_\Ga}L, \widetilde{L}=(G,L), \Ga]$ and thus $\widetilde{L}$ is a cotilting $\Ga$-module, ${_\Ga}L$ is a dualizing summand of $\widetilde{L}$ and $\id_{\mbF} L \leq \id_{\Ga} \widetilde{L} \leq \id_{\mbF} L + 2$. This gives \cite[Theorem 3.13]{ASoII}. The fact that the 4-tuple assignment restricts to an involution gives \cite[Theorem 2.8]{ASoIII}. 
\end{rem}


\subsection{Derived equivalence induced by an $\mbF$-dualizing summand}
Let $[\La , M, L, G]$ be a 4-tuple satisfying 
 $\La \in \add (G)$, $\mbF=\mbF_G$, $L$ is $\mbF$-cotilting and $M$ is an $\mbF$-dualizing summand of $L$. Then by Theorem \ref{relBBthm} the 4-tuple assignment gives a 4-tuple $[\Ga=\End_{\La}(M) , {_\Ga}M, \widetilde{L}=(G,M), \widetilde{G}=(L,M)]$ satisfying $\Ga \in \add (\widetilde{G})$, $\widetilde{\mbF}=\mbF_{\widetilde{G}}$, $\widetilde{L}$ is $\widetilde{\mbF}$-cotilting and ${_\Ga}M$ is an $\widetilde{\mbF}$-dualizing summand of $\widetilde{L}$. We consider the derived categories of exact categories 
$\mathsf{D}^b_{\mbF}(\Lamod)$ and $\mathsf{D}^b_{\widetilde{\mbF}}(\Gamod)$ and we will show the functors $(-,{_{\La}} M)$ and $(-,{_{\Ga}} M)$ induce a duality between triangulated categories $\mathsf{D}^b_{\mbF}(\Lamod)$ and $\mathsf{D}^b_{\widetilde{\mbF}}(\Gamod)$. 

\begin{pro}\label{relder-equi}
Let $[\La , M, L, G]$ be a $4$-tuple such that $\La \in \add (G)$, $\mbF=\mbF_G$, $L$ is $\mbF$-cotilting and $M$ is an $\mbF$-dualizing summand of $L$ and let $[\Ga, {_\Ga}M, \widetilde{L}, \widetilde{G}]$ be the corresponding $4$-tuple under the $4$-tuple assignment. Then the functors $(-,{_{\La}} M)$ and $(-,{_{\Ga}} M)$ induce a triangle duality between $\mathsf{D}^b_{\mbF}(\Lamod)$ and $\mathsf{D}^b_{\widetilde{\mbF}}(\Gamod)$.
\end{pro}
\begin{proof}
Let $B=\End_{\La} (L)$ and $\widetilde{B}=\End_{\Ga} (\widetilde{L})$, then $C:=(G,L)$ is a cotilting $B$-module and $\widetilde{C}:=(\widetilde{G},\widetilde{L})$ is a cotilting $\widetilde{B}$-module. By \cite[Proposition 4.4.3]{Bu-Closed} , the functor $(-, {_{\La}} L)$ induces a triangle duality between $\mathsf{D}^b_{\mbF}(\Lamod)$ and $\mathsf{D}^b(B\text{-mod})$ and the functor $(-, {_{\Ga}} \widetilde{L})$ induces a triangle duality between $\mathsf{D}^b_{\widetilde{\mbF}}(\Gamod)$ and $\mathsf{D}^b(\widetilde{B}\text{-mod})$.

We note that by Lemma \ref{fullfaith} (1) the composition
\[
\End_{B}(C)=((G,L),(G,L))\xrightarrow{\cong} \End_{\La}(G)^{op}\xrightarrow{\cong} ((G,M),(G,M))=\End_{\Ga}(\widetilde{L})=\widetilde{B}
\]
is an isomorphism of algebras. Similarly, we have $\End_{\widetilde{B}}(\widetilde{C})\cong \End_{\Ga}(\widetilde{G})^{op}\cong B$. Since ${_B} C$ is cotilting, ${_{\widetilde{B}}} C$ is also cotilting and we have
\[
{_{\widetilde{B}}} C=(B, {_B} C)=((L,L),(G,L))\xrightarrow{\cong} (G,L)\xrightarrow{\cong} ((L,M),(G,M))=(\widetilde{G},\widetilde{L})={_{\widetilde{B}}} \widetilde{C}
\]
by Lemma \ref{fullfaith} (1). It follows that the functors $(-,{_{B}} C)$ and $(-,{_{\widetilde{B}}} \widetilde{C})$ induce a triangle duality between $\mathsf{D}^b(B\text{-mod})$ and $\mathsf{D}^b(\widetilde{B}\text{-mod})$. The desired triangle duality follows by combining this duality and the above triangle dualities.
\end{proof}

\begin{rem}
As the above proof suggests, there exist triangle equivalences $\mathsf{D}^b_{\mbF}(\Lamod) \simeq \mathsf{D}^b(\widetilde{B}\text{-mod})$ and $\mathsf{D}^b_{\mbF}(\Gamod) \simeq \mathsf{D}^b(B\text{-mod})$. The dual version of Proposition \ref{relder-equi} shows that an $\mbF$-codualizing summand of an $\mbF$-tilting module will induce a relative derived equivalence.
\end{rem}

\subsection{$\mbF$-Gorenstein algebra}
Recall that an algebra $\La$ is called $Gorenstein$ if $\id({_\La}\La)<\infty$ and $\id(\La_\La)<\infty$. Define
\[
\mcP^{\infty}(\La)=\{X\in \Lamod | \pd_{\La}X<\infty \} \ \text{and}\
\mcI^{\infty}(\La)=\{Y\in \Lamod | \id_{\La}Y<\infty \}.
\]
Then $\La$ being Gorenstein is equivalent to $\mcP^{\infty}(\La)=\mcI^{\infty}(\La)$. Let $\mbF=\mbF_G=\mbF^H$ be a subbifunctor of $\Ext_{\La}^1$ and define 
\[
\mcP^{\infty}(\mbF)=\{X\in \Lamod | \pd_{\mbF}X<\infty \} \ \text{and}\
\mcI^{\infty}(\mbF)=\{Y\in \Lamod | \id_{\mbF}Y<\infty \}.
\]
Following \cite{ASo-Gor} we call an algebra $\mbF$-$Gorenstein$ if $\mcP^{\infty}(\mbF)=\mcI^{\infty}(\mbF)$, and $\mbF$-Gorenstein algebras can be chcaracterized as follows.

\begin{lem}\label{F-Gor}$($\cite[Proposition 3.3]{ASo-Gor}$)$
\begin{itemize}
\item[(1)] An algebra $\La$ is $\mbF$-Gorenstein if and only if there exists an $\mbF$-cotilting $\mbF$-tilting module. 
\item[(2)] An algebra $\La$ is $\mbF$-Gorenstein if and only if every $\mbF$-cotilting module is $\mbF$-tilting and every $\mbF$-tilting module is $\mbF$-cotilting. 
\end{itemize}
\end{lem}

\begin{cor}\label{relGor}
Let $[\La , M, L, G]$ be a $4$-tuple satisfying $\La \in \add (G)$, $\mbF=\mbF_G$, $L$ is $\mbF$-cotilting and $M$ is an $\mbF$-dualizing summand of $L$ and let $[\Ga, {_\Ga}M, \widetilde{L}, \widetilde{G}]$ be the corresponding $4$-tuple under the $4$-tuple assignment. Then $\La$ is an $\mbF$-Gorenstein algebra if and only if $\Ga$ is an $\widetilde{\mbF}$-Gorenstein algebra.
\end{cor}
\begin{proof}
Consider the 6-tuple assignment $[\La , M, L,R, G,H] \mapsto [\Ga, {_\Ga}M, \widetilde{L}, \widetilde{R},\widetilde{G}, \widetilde{H}  ] $ as in the proof of Theorem \ref{relBBthm}. Then $L,R$ are $\mbF$-cotilting modules and $\widetilde{L}, \widetilde{R}$ are $\widetilde{\mbF}$-cotilting modules. By Lemma \ref{F-Gor}, $\La$ is $\mbF$-Gorenstein if and only if $L$ and $R$ are $\mbF$-tilting modules, if and only if $\widetilde{L}, \widetilde{R}$ are $\widetilde{\mbF}$-tilting modules by the tilting version of Theorem \ref{relBBthm}, if and only if $\Ga$ is $\widetilde{\mbF}$-Gorenstein by Lemma \ref{F-Gor} again.
\end{proof}

\begin{rem}
\begin{itemize}
\item[(1)] The tilting version of Theorem \ref{relBBthm} implies that $\pd_{\widetilde{\mbF}} \widetilde{L}=\pd_{\mbF} L$, $\pd_{\mbF} L \leq \pd_{\mbF} R \leq \pd_{\mbF} L + 2$ and $\pd_{\widetilde{\mbF}} \widetilde{L} \leq \pd_{\widetilde{\mbF}} \widetilde{R} \leq \pd_{\widetilde{\mbF}} \widetilde{L} + 2$. Now by using \cite[Proposition 3.4]{ASo-Gor}, we see that $\pd_{\widetilde{\mbF}} \widetilde{G}=\id_{\widetilde{\mbF}} \widetilde{H} \leq \pd_{\widetilde{\mbF}} \widetilde{L}+\id_{\widetilde{\mbF}} \widetilde{L} \leq \pd_{\mbF} L+ \id_{\mbF} R +2$. 

\item[(2)] In particular, if we take $M=L$ then the above result gives  \cite[Proposition 3.1 and Proposition 3.6]{ASo-Gor}. 

\end{itemize}
\end{rem}

\section{Special cotilting} \label{specialTilt}
We assume throughout this section that $\mbF=\mbF_G=\mbF^H$ for a generator $G$ and a cogenerator $H$. The easiest situation where relative dualizing summands appear in relative cotilting modules are when these summands are 1-$\mbF$-faithful $\mbF$-injective modules. 

\begin{dfn} Let $C$ be an $\mbF$-cotilting module of $\id_{\mbF} C\leq r$. We say that $C$ is special if 
it has an $\mbF$-injective $(r-1)$-$\mbF$-dualizing summand $I$. This is equivalent to an $\mbF$-injective summand $I$ of $C$ such that $\cogen^{r-1}_{\mbF}(C) = \cogen^{r-1}_{\mbF}(I)$ by Lemma \ref{reldualizing}. We sometimes call $C$ $I$-special if it is special with respect to the $\mbF$-injective $I$.\\
Dually, we say an $\mbF$-tilting module $T$ of $\pd_{\mbF} T\leq r$ is special if it has a $\mbF$-projective summand $P$ such that $\gen_{r-1}^{\mbF}(T) = \gen_{r-1}^{\mbF}(P)$.
\end{dfn}

We look at a minimal $\mbF$-injective $\mbF$-coresolution of $G$ 
\[
0\to G \to I_0\to I_1\to I_2 \to \cdots
\]
and define $J_n= \bigoplus_{t\leq n} I_t$ (so in particular we have $G\in \cogen^n_{\mbF} (J_n)$

\begin{thm} 
Let $r\geq 1$. 
We consider the following three finite sets. 
\begin{itemize}
\item[(1)] Isomorphism classes of basic special cotilting modules of $\id_{\mbF} \leq r$. 
\item[(2)] Isomorphism classes of basic $\mbF$-injective modules $I$ with $G\in \cogen^{r-1}_{\mbF}(I)$. 
\item[(3)] Isomorphism classes of basic $I\in \add (H)$ with $J_{r-1}\in \add (I)$. 
\end{itemize}
Then the sets $(2)$ and $(3)$ are equal. Mapping $C$ to its maximal $F$-injective summand gives a bijection between $(1)$ and $(2)$. The inverse is given by 
mapping $I$ to $C_{I,r}:=I\oplus \Omega_I^rH$. 
\end{thm} 

\begin{proof}
Assume $J_{r-1}\in \add (I)\subset \add (H)$, then clearly $G\in \cogen^{r-1}_{\mbF}(J_{r-1})\subset \cogen^{r-1}_{\mbF}(I)$ and we conclude that (3) is a subset of (2). So assume $I \in \add (H)$ with $G \in \cogen^{r-1}_{\mbF}(I)$. Since the minimal $\mbF$-injective $\mbF$-exact $r$-copresentation (of $G$) must be a summand of any other $\mbF$-injective $\mbF$-exact $r$-copresentation, it follows that $J_{r-1}\in \add(I)$ and therefore the sets (2) and (3) are equal. \\
So let $C$ be an $I$-special $r$-$\mbF$-cotilting module and let $J$ be its maximal injective summand - of course $I \in \add (J)$ and clearly $\copres_{\mbF}^{r-1} (I) \subseteq \copres_{\mbF}^{r-1} (J) \subseteq \copres_{\mbF}^{r-1} (C) $.
Since $I,J $ are $\mbF$-injective and $C$ is $r$-$\mbF$-cotilting we conclude that these inclusions of subcategories coincide with  $\cogen_{\mbF}^{r-1} (I) \subseteq \cogen_{\mbF}^{r-1} (J) \subseteq \cogen_{\mbF}^{r-1} (C) $. 
Since $C$ is $I$-special it follows that they are all equal, in particular $J \in \cogen_{\mbF} (I)$ implies $J \in \add (I)$ and therefore $\add (I) = \add(J)$. This means the map is well-defined. It follows from lemma \ref{HigherTilt} that the assignment $I \mapsto C_I= I \oplus\Omega_I^r H$ is the inverse map. 
\end{proof}

Let $\Sigma_\mbF^r(\La )$ be the finite  subposet of the poset of isomorphism classes of basic $\mbF$-cotilting modules of $\id_{\mbF} \leq r$, where the partial order is given by inclusion of perpendicular categories... \\
Let $\add_{J_{r-1}}(H)$ be the lattice given by isomorphism classes of basic summands $I$ of $H$ such that $J_{r-1} \in \add(I)$ . The partial order is just given by inclusion of summands, the meet and join are defined in the obvious way. In particular, if $H=J_{r-1}\oplus X$ with $\lvert X \rvert =t$, then the lattice  $\add_{J_{r-1}}(H)$ is isomorphic to the power set $\mathcal{P}(\{1, 2, \ldots , t\})$  which is a poset  with respect to inclusion and a lattice with respect to intersection and union (sometimes also referred to as a $t$-dimensional cube). 

\begin{cor} The finite poset $\Sigma_\mbF^r(\La )$ is a lattice and the bijection from the previous theorem gives a lattice isomorphism 
\[ \Sigma_\mbF^r(\La) \to \add_{J_{r-1}}(H). \]
\end{cor}

We also observe that if an $I$-special $r$-$\mbF$-cotilting module $C$ has an $(r-1)$-$\mbF$-dualizing summand $M$, then $I \in \add (M)$. 

We give now several little applications, in particular connecting it with the other parts of the article. 

\subsection{Examples and applications}
\begin{itemize}
\item[(1)] Non-relative special tilting has been defined in \cite{PS1} and many special cases had been considered before, as APR-tilting and BB-tilting \cite{BGP}, \cite{BrenBut}, \cite{APR}, $n$-APR-tilts \cite{IO} or flip-flops for posets \cite{L}. 
Any endomorphism ring of a generator has a canonical special cotilt, this has been used to define desingularizations of orbit closures and quiver Grassmannians in \cite{CIFR}, \cite{CBSa}, \cite{PS2}.
\item[(2)] We explain that (non-relative) special cotilting naturally gives two recollements relating the cotilted algebras: 
Let $I$ be a $(k-1)$-faithful injective $\La$-module for $k\geq 1$ and $C=C_{I,k}=I \oplus \Omega^{k}_I \kdual \La$ the $I$-special $k$-cotilting module. 
Then $\Omega_I^{k}$ is an equivalence of categories $\add \kdual \La /\add I \to \add C/ \add I$ (for the definition of ideal quotients, see \cite[A.3]{ASS}) with quasi-inverse $\Omega_I^{-k}$ (this follows from \cite[Theorem 5.2]{AR-CM} with $\mathcal{X}=\add (I)$). 
Let $B=\End_{\La}(C)^{op}$, then ${}_B\kdual C$  is special $k$-tilting module with respect to the $(k-1)$-faithful projective module $P=(C,I)$. 
Let $P=B \varepsilon $ and $I=\kdual (e \La)$ for idempotents $e\in \La, \varepsilon \in B$. Then the equivalence $\Omega_I^{k}$ induced an isomorphism of algebras 
\[ (\La/(e))^{op} = \End_{\add \kdual \La /\add I }(\kdual \La ) \cong \End_{\add C/ \add I}(C )= (B/(\varepsilon ))^{op}\] 
Observe also $e\La e\cong  \End_{\La} (I)^{op} \cong \varepsilon B \varepsilon $, therefore we have two recollements with isomorphic ends induced by the idempotents $e, \varepsilon$. 
 \[
 \xymatrix{   && \Lamod  \ar[dd]_{F}\ar@<1ex>[dll]^{p} \ar@<-1ex>[dll]_{q } \ar@<1ex>[drr]|{e} && \\
 \La / (e) \text{-}\modu \quad  \ar[urr]|{i} \ar[drr]|{j} &&&&  \;\;
\eLamod \ar@<-2ex>[ull]_{\ell }  \ar@<0ex>[ull]^{r }  \ar@<0ex>[dll]_{\lambda}   \ar@<2ex>[dll]^{\rho } \\
  && B \text{-}\modu \ar@<1ex>[ull]^{\pi} \ar@<-1ex>[ull]_{\phi } \ar@<-1ex>[urr]|{\varepsilon } && \\
 }
 \]
Furthermore, the cotilting functor $F:=\kdual (-,C)$ commutes with the following functors from the recollements $\varepsilon \circ F=e, F\circ \ell =\lambda $.

\item[(3)] The standard cogenerator correspondence says that the assignment $[\La , {}_{\La}E] \mapsto [\Ga , {}_{\Ga}I]$ defined by $\Ga = \End(E), I={}_{\Ga}E$ gives a bijection between 
\begin{itemize}
\item[(a)] $[\La , {}_{\La}E]$ with $\kdual \La \in \add E.$
\item[(b)] $[\Ga , {}_{\Ga}I]$ with $I$ injective and $\Ga \in \cogen^1(I)$.
\end{itemize}
Let us denote $C_I$ to be the special $2$-cotilting $\Ga$-module which exists in situation (b). Then the AS-assignment $[\Ga , {}_{\Ga}I, C_I] \mapsto [\La , {}_{\La}E, G]$ with $G=(C_I,I)$ gives a natural extension of the cogenerator correspondence to a bijection between the following. 
\begin{itemize}
\item[(a')] $[\La , {}_{\La}E, G]$ with $\kdual \La \in \add E$ and $E$ is an $\mbF_G$-cotilting module.
\item[(b')] $[\Ga , {}_{\Ga}I, C_I]$ with $I$ injective and $\Ga \in \cogen^1(I)$, $C_I$ $2$-cotilting with $\cogen^1 (C_I) = \cogen^1 (I)$.
\end{itemize}
\end{itemize}

This can be generalized to the $4$-tuple assignment as follows: 

\subsubsection{Example of the relative cotilting correspondence using special cotilting} \label{ExOf4tuple}
This is our main example for theorem \ref{relBBthm}. Let us look at the $5$-tuple assignment $[\La, I, L, G,H] \mapsto [\Gamma=\End_{\La}(I) , {}_{\Gamma} I, \widetilde{L}=(G,I),\widetilde{G}=(L,I),  \widetilde{H}=\kdual (I,H)]$. 
Then this gives a involution on the following $5$-tuples $[\La, I, L, G,H]$ with  
$\La \in \add(G), \kdual \La \oplus I \in \add (H)$, $\mbF=\mbF^H=\mbF_G$ and $L$ is an $I$-special $2$-$\mbF$-cotilting module. \\

The proof goes as follows: By Theorem \ref{relBBthm} we know that $\widetilde{L}$ is again an $\widetilde{\mbF}$-cotilting module with $\widetilde{\mbF}= \mbF^{\widetilde{H}}$ and has an $\widetilde{F}$-dualizing summand ${}_{\Gamma}I$. 
So we need to see that $\idim_{\widetilde{\mbF}} \widetilde{L} \leq 2$, then $\widetilde{L}$ is the (uniquely determined) ${}_{\Gamma}I$-special $2$-$\mbF$-cotilting module.  
Recall that the assumption ensures that we have an $\mbF$-exact strong $I$-dualizing sequence $0 \to L \to I_0 \to I_1 \to H \to 0$ with $I_j \in \add (I)$, so we can see $R:=H$ as the right end of it. 
This has been used to show that for $\widetilde{G}:=(L,I), \widetilde{H} =\kdual (I,H)$ we have $\widetilde{\mbF}= \mbF^{\widetilde{H}}=\mbF_{\widetilde{G}}$. 
Now, apply $(-,I)$ to a minimal projective presentation of $G$ and $\kdual (I,-)$ to a minimal injective copresentation of $H$ to obtain an $\widetilde{\mbF}$-exact, strong ${}_{\Gamma}I$-dualizing sequence with left end $(G,I)=\widetilde{L}$ and right end $\kdual (I,H)=\widetilde{H}$. This ensures that $\idim_{\widetilde{\mbF}} \widetilde{L}\leq 2$ and therefore $\widetilde{L}$ is an ${}_{\Gamma}I$-special $2$-$\widetilde{\mbF}$-cotilting. \\

We remark that special $r$-(co)tilting requires an $\mbF$-injective $(r-1)$-$\mbF$-dualizing summand. In our previously considered assignments we looked only at $1$-$\mbF$-dualizing summands, that is why our example only works for $r=2$.


\subsubsection{Mutation and dualizing sequences induce special tilts on endomorphism rings}

\begin{lem}
 Let $0 \to L \to M_0 \to \cdots \to M_k \to R \to 0$ be an $\mbF$-exact strong $k$-$M$-dualizing sequence with $\Ext^j_{\mbF}(L,R)=0$ for $j \geq 1$ and $L$,$R$ be $\mbF$-selforthogonal. Let $B =\End(L)$ and $A = \End (R)$. Then $T=(L,R)$ is a special $k$-tilting $A$-module with respect to $P= (M,R)$ and $C=\kdual (L,R)$ is a special 
 $k$-cotilting $B$-module with respect to $I=\kdual (L,M)$. Furthermore, we have 
 $\End_{A}(T)\cong B^{op}$ and $\End_{B} (C) \cong A^{op}$. 
 \end{lem}

\begin{proof}
Apply $(-,R)$ to the strong dualizing sequence, setting $P_i={}_{A}(M_i,R)\in\add (P)$, we get an exact sequence of $A$-modules 
\[  0 \to A \to P_k \to \cdots \to P_0 \to T \to 0.\]
This shows $\pd T\leq k$ and $A $ has an $\add (T)$-resolution with all middle terms in $\add (P)$ ($\subseteq \add (T)$). Since the dualizing sequence is strong and by assumption $L \in {}^{{}_{\mbF,1\leq }\perp }R\cap \cogen^{\infty }_{\mbF} (R)$, we can use Lemma \ref{fullfaith},(2) to get an isomorphism $\Ext^j_{\mbF}(L,L) \to \Ext^j_{A} (T,T)$. Since $L$ is $\mbF$-selforthogonal, the module $T$ is selforthogonal. This implies that $T$ is a special $k$-tilting module with respect to $P$. 
Similarly, one can show that $C$ is a special $k$-cotilting module with respect to $I$. 
The last claim follows from Lemma \ref{fullfaith},(1).  
\end{proof}



\subsubsection{Passing to endomorphism rings of special cotilting modules}
Recall, that in the non-relative case the \emph{Brenner-Butler} assignment $(BB)\colon [\Si , J, L'] \mapsto [B=\End_{\Si} (L'), \kdual (L', J),{}_B L']$
maps $J$-special $t$-cotilting $\Si $-modules $L'$ to a $\kdual (L',J)$-special $t$-cotilting $B$-module and this assignment is an involution on these triples. \\
We explain how this relates to relative special cotilting: Let $H$ be a basic cogenerator, $\Si =\End_{\La}(H)^{op}$ and $\varepsilon \in \Si $ the projection onto the summand $\kdual \La$, then we have a pair of adjoint functors 
\[  \ell =\kdual (-,H) \colon \Lamod \rightleftarrows \Simod \colon \varepsilon = (\Si \varepsilon , -)
\]
(cf. Appendix) with $\Bild \ell = \gen_1 (\Si \varepsilon  )$. As always we set $\mbF=\mbF^H$. 
Then for $I \in \add (H)$ we have $\ell (I) \in \add (\kdual \Si )$ and:  
$H \in \gen^{\mbF}_{t-1}(I) \Leftrightarrow \kdual \Si \in \gen_{t-1}(\ell (I)), \bigoplus_{j\geq 1}^{t} \Omega_{\ell (I)}^{j}\kdual \Sigma  \in \gen_1 (\Si \varepsilon  )$.\\ 

The assignment $[\La, I, L, H] \mapsto [\Sigma = \End_{\La} (H)^{op}, \ell(I), \ell (L)]$ injects an $I$-special $t$-$\mbF^H$-cotilting modules $L$ to an $\ell (I)$-special $t$-cotilting $\Sigma $-module $\ell (L)$. Any $J$-special $t$-cotilting $\Sigma$-module $L'$ for some $J \in \add ( \kdual \Sigma  )$ is in the image of this assignment if and only if $\bigoplus_{j\geq 1}^{t} \Omega_{J}^{j}\kdual \Sigma  \in \gen_1 (\Si \varepsilon )$. 
The assignment $[\La, I, L, H] \mapsto [B = \End_{\La} (L), \kdual(L,I), \kdual (L,H)]$ injects an $I$-special $t$-$\mbF^H$-cotilting modules $L$ to an $\kdual (L,I)$-special $t$-cotilting $B$-module $\kdual (L,H)$. 
In fact, combining the assignments we get a commuting triangle as follows 
\[ \xymatrix{ &[\La , I, L, H]\ar[dl]_{\ell}\ar[dr]^{\kdual (L,-)} & \\
[\Sigma = \End_{\La} (H)^{op}, \ell(I), \ell (L)] \ar@{<->}[rr]^{(BB)}&& [B = \End_{\La} (L), \kdual(L,I), \kdual (L,H)] }
\]

\begin{exa}
Here are the endomorphism rings of the special cotilts of the relative Auslander algebras for $\La=K(1\to 2\to 3\to 4\to 5)$, $G=P_5 \oplus M, M=\bigoplus_{i=1}^4\bigoplus_{j\geq 0} \tau^{-j}P_i$, $\mbF=\mbF_G$ from Example \ref{ex-FAusalg-2}, (4). We choose $I=M\in \add (H) $, then the $M$-special tilting and cotilting modules conincide with: $G,M\oplus S_4, M\oplus S_3, M\oplus S_2, H$. Their respective endomorphism ring is shown by the quiver with relations below.  
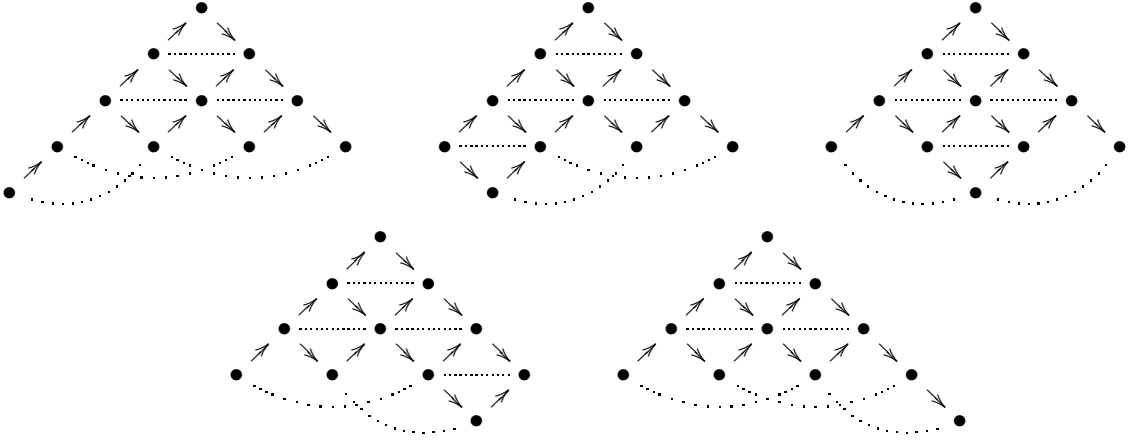
\[ 
\xymatrix @-1.5pc{
&&&& \bullet \ar[dr]&&&&\\
&&& \bullet \ar[ur] \ar[dr]\ar @{..}[rr]&& \bullet \ar[dr]&&&\\
&&\bullet \ar[ur]\ar[dr]\ar @{..}[rr]&&  \bullet \ar[dr]\ar[ur]\ar @{..}[rr]&&\bullet  \ar[dr]&& \\
&\bullet \ar[ur]\ar @{..}@/_1pc/[rrrr]& & \bullet \ar[ur]\ar @{..}@/_1pc/[rrrr]& & \bullet \ar[ur]& & \bullet &\\
\bullet \ar[ur]\ar @{..}@/_1pc/[urrr]&&&& &&&&
}
\xymatrix @-1.5pc{
&&&& \bullet \ar[dr]&&&&\\
&&& \bullet \ar[ur] \ar[dr]\ar @{..}[rr]&& \bullet \ar[dr]&&&\\
&&\bullet \ar[ur]\ar[dr]\ar @{..}[rr]&&  \bullet \ar[dr]\ar[ur]\ar @{..}[rr]&&\bullet  \ar[dr]&& \\
&\bullet \ar[ur] \ar[dr]\ar @{..}[rr]& & \bullet \ar[ur]\ar @{..}@/_1pc/[rrrr]& & \bullet \ar[ur]& & \bullet &\\
&&\bullet \ar[ur] \ar @{..}@/_1pc/[urrr]&&& &&&
}
\xymatrix @-1.5pc{
&&&& \bullet \ar[dr]&&&&\\
&&& \bullet \ar[ur] \ar[dr]\ar @{..}[rr]&& \bullet \ar[dr]&&&\\
&&\bullet \ar[ur]\ar[dr]\ar @{..}[rr]&&  \bullet \ar[dr]\ar[ur]\ar @{..}[rr]&&\bullet  \ar[dr]&& \\
&\bullet \ar[ur] \ar @{..}@/_1pc/[drrr]& & \bullet \ar[ur]\ar[dr]\ar @{..}[rr]& & \bullet \ar[ur]& & \bullet &\\
&&&&\bullet \ar[ur] \ar @{..}@/_1pc/[urrr]&&& &
}
\]
\[
\xymatrix @-1.5pc{
&&&& \bullet \ar[dr]&&&&\\
&&& \bullet \ar[ur] \ar[dr]\ar @{..}[rr]&& \bullet \ar[dr]&&&\\
&&\bullet \ar[ur]\ar[dr]\ar @{..}[rr]&&  \bullet \ar[dr]\ar[ur]\ar @{..}[rr]&&\bullet  \ar[dr]&& \\
&\bullet \ar[ur] \ar @{..}@/_1pc/[rrrr]& & \bullet \ar[ur]\ar @{..}@/_1pc/[drrr]& & \bullet \ar[ur] \ar[dr]\ar @{..}[rr]& & \bullet &\\
&&& &&&\bullet \ar[ur] &&
}
\xymatrix @-1.5pc{
&&&& \bullet \ar[dr]&&&&\\
&&& \bullet \ar[ur] \ar[dr]\ar @{..}[rr]&& \bullet \ar[dr]&&&\\
&&\bullet \ar[ur]\ar[dr]\ar @{..}[rr]&&  \bullet \ar[dr]\ar[ur]\ar @{..}[rr]&&\bullet  \ar[dr]&& \\
&\bullet \ar[ur]\ar @{..}@/_1pc/[rrrr] & & \bullet \ar[ur]\ar @{..}@/_1pc/[rrrr]& & \bullet \ar[ur] \ar @{..}@/_1pc/[drrr]& & \bullet \ar[dr]&\\
&&& &&&&&\bullet 
}
\]

\end{exa}

\section{Appendix: Embedding into an abelian category} 
We fix $\Delta = \End_{\La}(G)^{op}$ and $e\in \Delta $ projection onto the summand $\La$ (resp. $\Sigma = \End_{\La} (H)^{op}$ and $\varepsilon \in \Sigma$ the projection onto $\kdual \La$), $r= \Hom_{\La} (G, -)$ then we have a pair $(e,r)$ of adjoint functors (resp. $\ell= \Sigma \varepsilon \otimes_{\La} -= \kdual \Hom_{\La}(-, H)$, then we have an adjoint pair $(\varepsilon, \ell) $)
\[ 
e \colon \Demod \rightleftarrows \Lamod \colon r  \quad (\text{resp.} \; \; \ell \colon \Lamod \rightleftarrows \Simod \colon \varepsilon )
\]
with $e$ is exact and $r$ is fully faithful, maps $\mbF$-exact sequences to exact sequences and $\add(G)$ to $\add (\Delta )$. In particular, it maps $\mbF$-projective resolutions to projective resolutions and we get induced isomorphisms 
\[  \Ext^i_{\mbF}(M,N) \to \Ext^i_{\Delta} (r(M), r(N)), \quad i \geq 0.\] 
Dually, $\varepsilon $ is exact, $\ell $ is fully faithful, maps $\mbF $-exact sequences to exact sequences and $\add (H) $ to $\add (\kdual \Sigma )$, it maps $\mbF$-injective resolutions to injective resolutions and induces isomorphisms on the Ext-groups $\Ext^i_{\mbF}(M,N) \to \Ext^i_{\Sigma } (\ell (M), \ell (N)),  i \geq 0$.  We have  \[\Bild r = \cogen^1 (\kdual (e \Delta))\text{ and }\Bild \ell = \gen_1 (\Sigma \varepsilon ). \]

It is also easy to see: If $T$ is a relative tilting $\La$-module, then $r(T)$ is a tilting $\Delta $-module: Conversely, every tilting $\Delta $-module in $\cogen^1(J)$ restricts under $e$ to a relative tilting module. This gives a bijection, respecting the partial order (given by inclusion of perpendicular categories). \\
If $C$ is a relative cotilting module then $\ell (C)$ is a cotilting $\Sigma$-module and every cotilting module in $\Bild \ell = \gen^1( \Sigma \varepsilon )$ restricts under $\varepsilon $ to a relative cotilting module. 


Furthermore, in \cite{ASoII} Auslander and Solberg showed 
\[ 
\gldim_{\mbF} \La \leq \gldim \Delta \leq \gldim_{\mbF} \La +2. 
\]

\begin{lem} \label{gldimF}
Let $\La$ and $\Delta $ be as before and $k\geq 0$. Then the following are equivalent: 
\begin{itemize}
    \item[(1)]  $\pd_{\mbF} \kdual \La \leq k$ and $\gldim \Delta \leq k+2$,
    \item[(2)] $\gldim_{\mbF} \La \leq k$,
    \item[(3)] $\idim_{\mbF} \La \leq k$ and $\gldim \Si \leq k+2$.
\end{itemize}
\end{lem}

\begin{proof}
$(1) \Rightarrow (2)\colon $ Let $J= \kdual  (e\Delta ) $. 
Clearly, $\gldim_{\mbF} \La \leq k $ if and only if $\Ext^{k+1}_{\Delta } (\cogen^1(J), \cogen^1(J)) =0$.
We have $J=\kdual (e\Delta )= r(\kdual \La )$ and it is easily seen that 
$\pd_{\mbF} \kdual \La \leq k $ is equivalent to $\pd {}_{\Delta }J \leq k$.\\
We claim the stronger implication: $\gldim \Delta \leq k+2$ and $\pd J \leq k$ implies $\Ext^{k+1} (\cogen^1 (J) ,\Demod )=0$ (i.e., $\pd X \leq k$ for all $X \in \cogen^1 (J)$). \\
If we have an exact sequence $0 \to A \to J_0 \to B \to 0$ with $J_0\in \add (J)$ and we apply a functor $(-,Y)$ then we get a dimension shift $\Ext^i(A, Y) \cong \Ext^{i+1}(B,Y)$ for all $i \geq k+1$. 
In particular, we have for $X \in \cogen^1 (J)$: $\Ext^{k+1}(X,Y) \cong \Ext^{k+2}( \Omega^-X, Y) \cong \Ext^{k+3} (\Omega^{-2} X, Y) =0$ since we assume that $\gldim \Delta \leq k+2$. \\
$(2) \Rightarrow (1)$ Clearly, if $\gldim_{\mbF} \La \leq k $, then $\pd_{\mbF} \kdual \La \leq k$. By Auslander-Solberg's result (see before) we also have 
\[ 
\gldim \Delta \leq \gldim_{\mbF} \La +2\leq k+2.
\] 
The equivalence of (2) and (3) is proven analogously. 
\end{proof}

\begin{exa}
Let $\La = K(1 \to 2 \to \cdots \to n )$. Then there are $2^N$ with $N=\sum_{k=1}^{n-1} k$ basic generators $G$. The minimal $\mbF$-global dimension is $0$ which is obtained if and only of $G$ is the Auslander generator. 
The maximal $\mbF$-global dimension is $n-1$ (cf. Example \ref{ex-FAusalg-2}, (4)). 
\end{exa}


\bibliographystyle{amsalpha}
\bibliography{Gen}

\end{document}